\documentclass[11pt]{article}
\usepackage{amsmath}
\usepackage{hyperref}
\usepackage{amsmath}
\usepackage{amsfonts}
\usepackage{graphicx}
\usepackage{amssymb}
\usepackage{leftidx}
\usepackage{color}
\usepackage{booktabs}
\usepackage{multirow}
\usepackage{bm}
\usepackage{booktabs}
\usepackage{float}
\usepackage{graphicx}
\usepackage{caption}
\usepackage{makecell}
\usepackage{amsmath}
\usepackage{hyperref}
\usepackage{amsmath}
\usepackage{amsfonts}
\usepackage{graphicx}
\usepackage{amssymb}
\usepackage{leftidx}
\usepackage{color}
\usepackage{booktabs}
\usepackage{multirow}
\usepackage{bm}
\usepackage{booktabs}
\usepackage{float}
\usepackage{graphicx}
\usepackage{caption}
\usepackage{makecell}
\usepackage{esint}
\usepackage{mathrsfs}

\allowdisplaybreaks
\newtheorem{theorem}{Theorem}[section]
\newtheorem{lemma}[theorem]{Lemma}
\newtheorem{proposition}[theorem]{Proposition}

\newtheorem{example}[theorem]{Example}
\newtheorem{definition}[theorem]{Definition}
\newtheorem{corollary}[theorem]{Corollary}
\newtheorem{remark}[theorem]{Remark}

\numberwithin{equation}{section}

\DeclareMathOperator{\sgn}{sgn}
\DeclareMathOperator{\supp}{supp}

\newenvironment{proof}[1][Proof]{\noindent\textbf{#1.} }{\ \rule{0.5em}{0.5em}}

\begin{document}

\title{Selection in the Zero-Noise Limit for High-Dimensional Diffusions with Measurable Drift%
}
\author{Chang Li$\thanks{%
Email: chang\_li@ruc.edu.cn.}$, Liangquan Zhang$^{1}\thanks{%
The corresponding author: L. Zhang acknowledges the financial support partly
by the National Nature Science Foundation of China (Grant No. 12571486,
12171053) and the Fundamental Research Funds for the Central Universities,
and the Research Funds of Renmin University of China (No.23XNKJ05). Email:
xiaoquan51011@163.com.}$ \\
%EndAName
{\small 1. School of Mathematics, }\\
{\small Renmin University of China, Beijing 100872, China}}
\maketitle

\begin{abstract}
This paper investigates the zero-noise limit of high-dimensional small-noise
diffusion processes governed by the stochastic differential equation (SDE)
\begin{equation*}
\mathrm{d}X_{t}^{\varepsilon }=b(X_{t}^{\varepsilon })\,\mathrm{d}t+\varepsilon \,\mathrm{d}W_{t},\quad
X_{0}^{\varepsilon }=0,\quad \varepsilon >0,
\end{equation*}%
where the drift coefficient $b$ is assumed to be measurable and bounded.
Under this condition, the associated deterministic ordinary differential
equation (ODE) $\dot{x}_{t}=b(x_{t})$ may admit multiple Filippov solutions,
i.e., solutions in the sense of differential inclusions, due to the lack of
Lipschitz continuity or uniqueness criteria. However, the introduction of
non-degenerate additive noise restores well-posedness: the perturbed system
admits a unique strong (pathwise) solution for each $\varepsilon >0$. Our
analysis is based on the radial-spherical decomposition. We prove that, under certain assumptions, the
radius dominates a one-dimensional comparison process and escapes to
infinity at a linear rate, and the extrinsic angular martingale converges.
Under an explicit separable gradient structure, expressed in terms of a
radial potential $V(r,\omega )=\langle b(r\omega ),\omega \rangle $, the
angle converges, for every fixed noise, almost surely to the set of global
angular maximizers of $V$, and its distance from that set tends to zero in
probability as $\varepsilon\rightarrow0$. This selection uses finite
angular quadratic variation, an infinite gradient clock with a
convergent perturbation, a gradient-energy lemma, convergence to a
single critical point, and direct exclusion of non-maxima through
an expanding-coordinate comparison. Two planar examples outside the
smooth bounded class, an unbounded cubic field and the discontinuous
field, are treated by independent finite-time path-law arguments.
In one dimension, scale-function weights select the extremal solution; an
explicit example gives positive-branch selection, while a Bessel-process
construction shows that the conclusion reverses for singular entrance
solutions. The limiting measure is compact, carried by immediate-departure
Filippov paths, and, under a Lipschitz radial-graph condition, purely
singular with Hausdorff dimension at most $d-1$. A continuous-mapping
criterion for exit-time convergence is also established.

\end{abstract}

\tableofcontents

%\begin{abstract}
%We obtain two results: i) The
%compactness of control domain takes important role in existence of OC; ii)
%In deterministic systems, the terminal constraints and finiteness of costs
%can be strictly satisfied simultaneously. However, in stochastic systems,
%due to the diffusion of noise and the coupling with terminal constraints,
%the optimal control cost will explode, and thus, we can only resort to near
%optimal control. These results can be verified by kinds of linear or nonlinear stochastic systems,
%\end{abstract}

%\title{ Non-Existence of Optimal Control vs. Near Optimal Existence in Stochastic Systems}

%\title{Non-Existence of Optimal Controls and
%Near-Optimal Existence in Stochastic System}

%\title{A Counterexample for the Non-Existence of Optimal Controls and Existence of Near-Optimal Controls in Stochastic Systems\\A Concrete Counterexample: Optimal Control Non-Existence and Near-Optimal Existence in Stochastic Systems}

\noindent \textbf{AMS subject classifications:} 93E20, 60H15, 60H30.

\noindent \textbf{Key words: } Filippov solution; geometric measure theory;
radial-spherical decomposition; stochastic differential equations; singular
limiting measure; weak convergence; zero-noise limit.

\section{Introduction}

\paragraph{Background and Motivation.}

The following ordinary differential equation
\begin{equation}
\left\{
\begin{array}{crl}
\mbox{\rm d}\xi ^{x}(t) & = & b\left( \xi ^{x}\left( t\right) \right) %
\mbox{\rm d}t,\text{\qquad }t\geq 0, \\
\text{ }\xi ^{x}\left( 0\right) & = & x\in \mathbb{R}^{d},%
\end{array}%
\right.  \label{1.1}
\end{equation}%
may have many solutions or have no solution at all if $b:\mathbb{R}%
^{d}\rightarrow \mathbb{R}^{d}$ is not Lipschitz continuous. This equation
can be regularized by adding the white noise $\varepsilon $\textrm{d}$W_{t}$
to its right-hand side with any positive small intensity $\varepsilon >0$
and $d$-dimensional Brownian motion $W$ which is the $d$-dimensional
coordinate process on the classical Wiener space $\left( \Omega ,\mathcal{F}%
,P\right) $, i.e., $\Omega $ is the set of continuous functions from $\left[
0,+\infty \right) $ to $\mathbb{R}^{d}$ starting from $0$ ($\Omega =C\left( %
\left[ 0,+\infty \right) ;\mathbb{R}^{d}\right) $ with the metric of the
uniform convergence), $\mathcal{F}$ the completed Borel $\sigma $-algebra
over $\Omega $, $P$ the Wiener measure and $W$ the canonical process: $%
W_{s}\left( \omega \right) =\omega _{s}$, $s\in \left[ 0,+\infty \right) ,$ $%
\omega \in \Omega .$ By $\left\{ \mathcal{F}_{s},0\leq s<+\infty \right\} $
we denote the natural filtration generated by $\left\{ W_{s}\right\} _{0\leq
s<+\infty }$ and augmented by all $P$-null sets.

More precisely, for any bounded Borel function $b:\mathbb{R}^{d}\rightarrow
\mathbb{R}^{d}$, $x\in \mathbb{R}^{d}$ and $d$-dimensional Brownian motion $%
W $, there exists a unique strong solution to the following stochastic
differential equation (SDE in short)
\begin{equation}
\left\{
\begin{array}{crl}
\mbox{\rm d}X^{x,\varepsilon }(t) & = & b\left( X^{x,\varepsilon }\left(
t\right) \right) \mbox{\rm d}t+\varepsilon \mbox{\rm d}W\left( t\right) ,%
\text{\qquad }t\geq 0, \\
\text{ }X^{x,\varepsilon }\left( 0\right) & = & x\in \mathbb{R}^{d},%
\end{array}%
\right.  \label{1.2}
\end{equation}%
for any fixed $\varepsilon >0$. This result can be seen in \cite{Bah, GEMM}.

A natural question concerns the behavior of the limit of perturbed SDEs (\ref%
{1.2}) with respect to the ODE (\ref{1.1}), as $\varepsilon \rightarrow
0^{+} $. With regard to this issue, we elaborate on the research
significance and challenges from the following three aspects: this issue
elucidates the impact of random disturbances on non-unique dynamical
systems, offering a critical perspective for understanding the long-term
behavior of complex systems. In disciplines such as physics, biology, and
economics, numerous systems are subject to noise interference, and their
deterministic models may exhibit non-uniqueness. Investigating the
zero-noise limit selection problem facilitates predicting the behavior of
these systems. Naturally, analyzing nonlinear dynamics and stochastic
processes in high-dimensional spaces is inherently intricate. When coupled
with the non-uniqueness arising from the Peano phenomenon, this renders the
problem highly challenging. In the classical Lipschitz case, we refer to
Freidlin and Wentzell \cite{Fr} and the references therein. However, when $b$
is merely continuous, more complex scenarios may arise. To highlight this
phenomenon, we will present a precise one-dimensional example (adapted from
\cite{BB}) as follows:

\begin{example}
\label{exa1}Put $b\left( x\right) =2\sgn(x)\sqrt{\left\vert x\right\vert }$.
The ODE
\begin{equation}
\left\{
\begin{array}{crl}
\xi ^{\prime }\left( t\right) & = & 2\sgn(\xi \left( t\right) )\sqrt{%
\left\vert \xi \left( t\right) \right\vert },\text{\qquad }t\geq 0, \\
\xi \left( 0\right) & = & 0,%
\end{array}%
\right.  \label{1.3}
\end{equation}%
has infinitely many solutions. But among of them only both \textquotedblleft
extremal\textquotedblright\ solutions $t\rightarrow t^{2}$ and $t\rightarrow
-t^{2}$\ are limits of the corresponding SDEs. In fact, the limit of $%
X^{x,\varepsilon }\left( t\right) $ (solution of (\ref{1.2})) is a
continuous stochastic process $X\left( \cdot \right) $ as $\varepsilon
\rightarrow 0^{+}$ defined by
\begin{equation*}
\left\{
\begin{array}{lll}
P\left( X\left( t\right) =t^{2}\right) & = & \frac{1}{2}, \\
P\left( X\left( t\right) =-t^{2}\right) & = & \frac{1}{2}.%
\end{array}%
\right.
\end{equation*}
\end{example}

Consider a physical model
\begin{equation}
\left\{
\begin{array}{crl}
h^{\prime }\left( t\right) & = & -2\sqrt{h\left( t\right) },\text{\qquad }%
t\in \left[ 0,T\right] , \\
h\left( T\right) & = & 0,\text{ }%
\end{array}%
\right.
\end{equation}%
which admits infinitely many solutions by Peano existence theorem. This
counterexample is derived from a physical model: Suppose that there is a
leaking container, the relationship between the height (function $h$) and
the time (variable $t$) of the water surface is defined by the above
differential equation, then because the leaking process can actually be
observed, there must be a solution to the equation. However, if you only
know the state of the container at a certain point after the leakage of
water ($h\left( T\right) =0$), it is impossible to predict how high the
original water level was (that is, there is no unique solution).

For comparison with the one-dimensional exit-time method in \cite{BB} \footnote{%
In the section of the concluding remarks in \cite{BB}, the authors
emphasized that \emph{c) It might not be simple to apply the method
developed in this article to higher dimensions. In any case the results for
the one-dimensional case suggest that, in the homogeneous case, the limiting
laws ``prefer'' the solutions of (1.1) which leave the starting point as
fast as possible. }We will explain this approach a moment later. },
we record and verify the boundary-value formula, including a general
interval centre. Let $l=x_0-r$, $u=x_0+r$, with $r>0$, and suppose $b$
is bounded and Borel on a neighbourhood of $[l,u]$. Define
\begin{equation}
 \frac{\varepsilon^2}{2}\phi_\varepsilon''+b\phi_\varepsilon'=-1
       \quad\hbox{almost everywhere on }(l,u),\qquad
 \phi_\varepsilon(l)=\phi_\varepsilon(u)=0. \label{bvp}
\end{equation}
An arbitrary fixed reference point may be used in the following
integrals; take it to be zero, with $b$ extended boundedly if necessary:
\begin{align*}
 Q_\varepsilon(s)&=\frac{2}{\varepsilon^2}\int_0^s b(v)\,\mathrm{d}v,\\
 A_\varepsilon(x)&=\int_0^x e^{-Q_\varepsilon(s)}\,\mathrm{d}s,\\
 B_\varepsilon(x)&=\frac{2}{\varepsilon^2}
       \int_0^x e^{-Q_\varepsilon(s)}
                   \left(\int_0^s e^{Q_\varepsilon(v)}\,\mathrm{d}v\right)\mathrm{d}s,\\
 c_2^\varepsilon&=\frac{B_\varepsilon(u)-B_\varepsilon(l)}
                       {A_\varepsilon(u)-A_\varepsilon(l)},\\
 c_1^\varepsilon&=\frac{A_\varepsilon(u)B_\varepsilon(l)
                       -A_\varepsilon(l)B_\varepsilon(u)}
                       {A_\varepsilon(u)-A_\varepsilon(l)}.
\end{align*}
Then $\phi_\varepsilon=c_1^\varepsilon+c_2^\varepsilon A_\varepsilon
-B_\varepsilon$. The denominator is positive, since
$A_\varepsilon'>0$. Differentiation almost everywhere gives
$\mathcal L A_\varepsilon=0$ and $\mathcal L B_\varepsilon=1$ for
$\mathcal L=(\varepsilon^2/2)\partial_{xx}+b\partial_x$; substitution
of $l,u$ gives both boundary conditions. Conversely, the integrating
factor $e^{Q_\varepsilon}$ shows that every solution with absolutely
continuous first derivative has this form before the boundary constants
are imposed. This proves uniqueness in that class.
For $x\in(l,u)$, stop the diffusion at its first exit $\tau$ from
$(l,u)$. The one-dimensional Ito formula for a continuously
differentiable function with locally absolutely continuous derivative
applies here (the second derivative is bounded on $[l,u]$). It gives
\[
 \mathbb E_x\phi_\varepsilon(X_{t\wedge\tau})
      =\phi_\varepsilon(x)-\mathbb E_x(t\wedge\tau).
\]
The left-hand side is bounded in absolute value, so
$\mathbb E_x\tau<\infty$ by monotone convergence. Letting $t\to\infty$
and using continuity at the two boundary points yields
$\phi_\varepsilon(x)=\mathbb E_x\tau$.
For $b=0$ the formula reduces to
$\{r^2-(x-x_0)^2\}/\varepsilon^2$, as required for the centred interval.
Clearly, for multi-dimensional case, this method fails. But one is able to
analyze the limits of $A_{\varepsilon }\left( x\right) $, $B_{\varepsilon
}\left( x\right) $ under some suitable conditions and then obtaining the
following result.

\begin{proposition}
Suppose that for some $\delta >0,$ the functions
\begin{equation*}
h(x)=\min_{\left[ x,x+\delta \right] }b,k(x)=\max_{\left[ x-\delta ,x\right]
}b
\end{equation*}%
are such that%
\begin{equation*}
\int_{0}^{r}\frac{1}{h\left( x\right) }\mathrm{d}x<+\infty ,\text{ }%
\int_{0}^{-r}\frac{1}{k\left( x\right) }\mathrm{d}x<+\infty
\end{equation*}%
Then every limiting value $P$ of $\left\{ P^{\varepsilon }\right\}
_{\varepsilon }$ is, concentrated on the \emph{extremal} solutions $\psi
_{1} $ and $\psi _{2}$, for a small time interval. More precisely for every
cluster point $\alpha $ of $\frac{-A_{\varepsilon }\left( -r\right) }{%
A_{\varepsilon }\left( r\right) -A_{\varepsilon }\left( -r\right) }$ as $%
\varepsilon \rightarrow 0,$ for some $t>0$%
\begin{equation}
P_{\left\vert \mathcal{F}_{t}\right. }=\alpha \delta _{\psi _{1}}+\left(
1-\alpha \right) \delta _{\psi _{2}},  \label{ba}
\end{equation}%
where $P_{\left\vert \mathcal{F}_{t}\right. }$ denotes the restriction of $P$
to $\mathcal{F}_{t}$. Conversely, if $P$ is a limiting value of $\left\{
P^{\varepsilon }\right\} _{\varepsilon }$, then (\ref{ba}) holds.
\end{proposition}

In the literature, more one-dimensional cases were considered for instance
by \cite{BB, Ba, Fl}\footnote{ In \cite{Fl}, Flandoli
also presented the same question, say, \emph{%
One of the most interesting problems related to the subject of these
lectures is the zero-noise limit when the deterministic equation has more
solutions and the stochastic equation has only one solution (for the
fundamental classical case see Friedlin and Wentzell [110]). The hope is
that it produces an interesting selection criterion for the deterministic
equation. This program is very difficult and only very few examples are
known. The main one is the case of 1-dimensional equations treated by Bafico
and Baldi [16, 17].}}. They showed that the limits of the solutions of the
perturbed SDEs (\ref{1.2}) are processes which are supported by the
solutions of ODE (\ref{1.1}). In \cite{G, S}, the authors used the large
deviation technique to give a more precise description of the limit in the
case
\begin{equation}
\xi ^{\prime }\left( t\right) =2\sgn\left( \xi \left( t\right) \right)
\left\vert \xi \left( t\right) \right\vert ^{\gamma },\qquad \gamma \in
\left( 0,1\right) .  \label{1.4}
\end{equation}%
We point out that another method to Bafico and Baldi's original problem was
revisited by Delarue and Flandoli in \cite{df14}. They developed a delicate
argument based on exit times. The point of the their proof is based on its
dynamical character. Noteworthy it is also valid in multidimensional case
but with a very specific right-hand side, comparing with the prime
assumption of a general continuous function; see e.g. Trevisian \cite{T}.
Pilipenko and Proske \cite{pp18spl} study the limit behavior of differential
equations with non-Lipschitz coefficients that are perturbed by a small
self-similar noise for one-dimensional case. Fjordholm, Musch and Pilipenko
\cite{fmp23} study the zero-noise limit for autonomous, one-dimensional ODEs
with discontinuous right-hand sides.

For the multi-dimensional case, the article of \cite{BYQ} by Buckdahn,
Ouknine and Quincampoix shows that the limit has its support in the set of
solutions to ODE (\ref{1.1}) when $b$ is only measurable (see Proposition %
\ref{pro1} below). In \cite{dfv14}, Delarue, Flandoli and Vincenzi solve a
two-dimensional zero noise problem with discontinuous drift. We point out
the paper by Delarue and Maurelli \cite{md19}, in which the multidimensional
gradient dynamics with H\"{o}lder type coefficients was perturbed by a small
Wiener noise.

If the drift in (\ref{1.1}) has H\"{o}lder-type asymptotics in a
neighborhood of $x=0$ and the perturbation is a self-similar noise,
Pilipenko and Proske \cite{pp18} employ the space-time scaling and reduce a
solution of the small-noise problem to a study of long time behavior of a
SDE with a fixed noise, assuming that the drift coefficient has a jump
discontinuity along a hyperplane and is Lipschitz continuous in the upper
and lower half-spaces. Kulik and Pilipenko \cite{kp21} consider the drift
and diffusion are locally Lipschitz continuous outside a fixed hyperplane $H$%
. The drift $a(x)$ has a H\"{o}der asymptotic as $x$ approaches $H.$ The see
also \cite{pp20, pp23} for different extensions.

For other discussions, see \cite{Fl, MMP, BS} references therein. Besides,
we mention that the work by Attanasio and Flandoli \cite{AF} studied a zero
noise limit for some linear PDEs of transport type related to the family of
ODEs (\ref{1.4}).

We state the following bounded-drift version directly whose proof can be found in Appendix \ref{app}.

\begin{proposition}[Filippov cluster laws for bounded Borel drift]
\label{pro1} Let $b:\mathbb{R}^d\to\mathbb{R}^d$ be bounded and Borel
measurable, and let $X^{x,\varepsilon}$ solve \eqref{1.2} with initial
point $x$. For every $T<\infty$, its laws are relatively compact on
$C([0,T];\mathbb{R}^d)$ as $\varepsilon\downarrow0$. Every cluster law is
concentrated on absolutely continuous paths $z$ satisfying
$z(0)=x$ and $\dot z(t)\in F_b(z(t))$ for almost every $t$.
\end{proposition}

We will be interested in knowing which solutions of (\ref{1.1}) are the
limit solutions of (\ref{1.2}) under measurable of $b$. Our main aim is to
provide a more precise description to the limit (for instance in Example \ref%
{exa1}, the constant solution $\xi \left( t\right) \equiv 0$ is a solution
to ODE (\ref{1.3}) but it is not in the support of the limit process).

To our best knowledge, the existing techniques of one-dimensional case are
based on the explicit computation of the laws of $X^{\varepsilon }\left(
\cdot \right) $. Indeed, in \cite{BB, Ba}, the laws are computed by studying
the Boundary Value Problems. In \cite{G, S}, laws are explicitly derived
from the large deviation theory. However, in our multi-dimensional case,
explicit computations of laws are hardly possible. The works introduced
above \cite{df14, pp18, pp20, pp23, pp18spl,T} are basing on large and
complex assumptions and calculations which is quite often restricted to
handle the multidimensional case generally. In this paper, we develop a
simple method to find the limiting solution by means of the properties of
drift $b$. As claimed before, it is not easy to define the \textquotedblleft
fast\textquotedblright\ leaving solutions in certain sense. One of the
possible strategies relies on the notion of fastest leaving solutions from a
compact set $K$ for the ODE initialized at the singularity, for instance, $%
K=[-r,r],r>0,$ in \cite{BB}. In this framework, the notion of fastest
solutions is well-defined: it corresponds to the shortest time needed to
reach the boundary of $K$. The whole point is then to identify this shortest
exit time with the limit of the exit times of the solutions of the SDEs as
the viscosity tends to zero. However, this is rather difficult as known
examples show that the points of the boundary of $K$ that can be reached in
the fastest way are not the only ones to be reached by the limiting
solutions of the vanishing SDEs. One effective trick is then to normalize $K$
in such a way that all the points of the boundary are reached in the same
time. In this framework, it is indeed possible to give a relevant notion of
fastest solutions. Nevertheless, how to normalize $K$ is also a intractable
problem except for a few special concrete circumstances.

Our main radial result is stated under (A1). Every zero-noise path-law
cluster is concentrated on Filippov solutions which leave the origin
immediately: $|X_t^0|\ge\psi(t)>0$ for every $t>0$, with the explicit
positive function $\psi$ in Lemma \ref{ml1}. This is a no-waiting
property, not minimization of every fixed-radius exit time. Theorem
\ref{the1} does not assert uniqueness of the cluster law or stochastic
selection of every deterministic instantaneous-escape solution.
We work with a specified cluster law whenever no separate uniqueness
proof is available.

The angular question requires additional structure and a distinct
order of limits. Proposition \ref{prop:select} uses (A5), a separable
sphere-gradient drift with explicit linear unstable modes. In that
class the direction has a limit in the fixed maximum set for every
fixed positive noise. The verified examples include a single preferred
direction and, in any dimension, $2d$ symmetric maximal directions.
This infinite-time assertion is not used as a substitute for a
finite-time zero-noise path theorem. The unbounded cubic field and the
discontinuous square-root field $b_2^\delta$ are treated separately:
\eqref{eq:cubic-selection} and \eqref{eq:b2-selection} prove their
finite-time path-law limits without applying the bounded, smooth
angular theorem outside its hypotheses.

\paragraph{Main Contributions.}
\begin{itemize}
\item Under (A1), we prove instantaneous escape of every Filippov
path-law cluster by a same-noise radial comparison valid from the
origin, together with the one-sided scalar limit theorem in
\cite{fmp23}. A direct essential-value argument supplies the required
Filippov identification for bounded Borel drift.
\item At fixed time, the marginal support is contained in the closure
of the endpoints of instantaneous-escape Filippov solutions, and it
is compact. There is no converse inclusion in the theorem. Excluding
delayed paths from the path law does not by itself assign zero mass to
their endpoint set, which can overlap endpoints of immediate paths.
\item We derive consistent unit-sphere gradients, deterministic
reparameterization, and the complete extrinsic Ito equations. Under
(A5) we prove angular selection in arbitrary dimension. Independent
arguments establish the four-branch path laws of the two planar
examples, including the joint $\delta\to0$, $\varepsilon\to0$ limit
for $b_2^\delta$.
\item If a marginal support has the additional compact Lipschitz
radial-graph representation \eqref{eq:radial.graph}, its Hausdorff
dimension is at most $d-1$, so the marginal law is singular with
respect to $d$-dimensional Lebesgue measure. This geometric hypothesis
is not inferred from (A1); the radial unit-drift example verifies it
and attains dimension $d-1$.
\end{itemize}

\paragraph{Organization of the Paper.}

Section \ref{sect2} introduces the bounded drift and local radial
assumptions, Filippov regularization, and the path space notation. Section %
\ref{sect3} recalls the zero noise cluster point theorem, proves fixed noise
point polarity, derives radial-angular identities, and develops the
deterministic spherical geometry. It then discusses examples, marginal
support inclusion, a conditional Hausdorff dimension result, and an
exit-time continuous mapping criterion. Section \ref{sect4} summarizes the
conclusions and clearly separates proved statements from open selection
questions. The appendix contains the one-dimensional scale-function formula
under standard regular diffusion hypotheses and the Lipschitz
Hausdorff dimension lemma.

\section{Preliminaries and Problem Setup}

\label{sect2}

Let $(\Omega ,\mathcal{F},\{\mathcal{F}_{t}\}_{t\geq 0},\mathbb{P})$ be a
filtered probability space satisfying the usual conditions (right-continuity
and completeness). Let $W_{t}=(W_{t}^{1},\dots ,W_{t}^{d})^{T}$ denote a $d$%
-dimensional standard $\mathcal{F}_{t}$-Brownian motion on this space.

It is well-known that to reproduce the results in Bafico \& Baldi \cite{BB}
in higher dimension one has to consider repulsing non-Lipschitz vector
fields. For this aim, we impose the following assumption:

\begin{enumerate}
\item[\textbf{(A1)}] Let $b:\mathbb{R}^{d}\rightarrow \mathbb{R}^{d},$ $%
d\geq 2$ be bounded and measurable, satisfying $b(0)=0.$ Assume that $%
q\left( x\right) =\langle b(x),x\rangle $ is continuous on $\mathbb{R}^{d}.$
Define the worst-case lower bound
\begin{equation}
\varphi (r):=\inf_{\left\vert x\right\vert =r}\langle b(x),\frac{x}{r}%
\rangle ,\text{ for }r>0.  \label{B}
\end{equation}%
Suppose that,
\begin{equation}
\varphi (r)>0,\text{ for }r>0,  \label{via}
\end{equation}%
and there exists $\delta >0,$ such that \
\begin{equation}
\text{if }h\left( r\right) =\min_{\left[ r,r+\delta \right] }\varphi ,\text{
then }\int_{0}^{\delta }\frac{1}{h(r)}\mathrm{d}r<\infty .  \label{bi2}
\end{equation}

%Let the vector field $b:\mathbb{R}^{d}\rightarrow \mathbb{R}^{d}$ be locally
%bounded and measurable, satisfying $b(0)=0$. Assume that the inner product
%function $\langle b(x),x\rangle $ is continuous on $\mathbb{R}^{d}$. Define
%the minimal radial potential
%\begin{subequations}
%\label{0}
%\begin{equation}
%\varphi (r):=\inf_{|x|=r}\left\langle b(x),\frac{x}{r}\right\rangle ,\quad
%r>0.  \label{B}
%\end{equation}%
%We impose the following uniform origin-repelling conditions: $(i)$ Uniform
%strict repulsion. There exist constants $r_{0}>0$ and $\delta _{0}>0$ such
%that
%\end{subequations}
%\begin{equation}
%\varphi (r)\geq \delta _{0}>0,\quad \forall \,0<r\leq r_{0}.  \label{via}
%\end{equation}%
%In particular, $\langle b(x),x/r\rangle >0$ for all $0<|x|\leq r_{0}$, so
%the drift field points strictly outward in a neighborhood of the origin. $%
%(ii)$ Local uniform integrability. Define $h(r)=\min_{s\in \lbrack
%r,r+\delta ]}\varphi (s)$. There exists $\delta \in (0,r_{0})$ such that
%\begin{equation}
%\int_{0}^{\delta }\frac{1}{h(r)}\,dr<\infty .  \label{bi2}
%\end{equation}
\end{enumerate}

\begin{remark}
\label{elip}
%If $b$ is a Lipschitz continuous function, we claim that $b$ does not satisfy
%(A3). Indeed, we have%
%\begin{equation*}
%\varphi (r)\leq \langle b(r\omega ),\omega \rangle =\langle b(x),\frac{x}{%
%\left\vert x\right\vert }\rangle \leq \left\vert b(x)\right\vert =\left\vert
%b\left( x\right) -b\left( 0\right) \right\vert \leq L\left\vert x\right\vert
%=Lr.
%\end{equation*}%
%Immediately,
%\begin{equation*}
%\infty =\int_{0}^{\delta }\frac{1}{Lr}dr\leq \int_{0}^{\delta }\frac{1}{%
%\varphi (s)}\,ds.
%\end{equation*}%
%Therefore, Lipschitz continuous functions are excluded in our paper.

If $b$ is Lipschitz continuous, then it fails to satisfy condition (A1).
Indeed, for any $r>0$ and $\omega \in \mathbb{S}^{d-1}$, we have
\begin{equation*}
0<\varphi (r)\leq \langle b(r\omega ),\omega \rangle =\left\langle b(x),%
\frac{x}{|x|}\right\rangle \leq |b(x)|=|b(x)-b(0)|\leq L|x|=Lr,
\end{equation*}%
where $x=r\omega $ and $L$ denotes the Lipschitz constant of $b$. It follows
that
\begin{equation*}
\int_{0}^{\delta }\frac{1}{\varphi (r)}\mathrm{d}r\geq \int_{0}^{\delta }%
\frac{1}{Lr}\,\mathrm{d}r=\infty .
\end{equation*}%
Hence, Lipschitz continuous functions are incompatible with assumption (A1)
and are therefore excluded from consideration in this work.
\end{remark}

\begin{remark}
By definition $h(x)=\min_{[x,x+\delta ]}b(y)\leq b(x),\quad \forall x\in
\lbrack 0,r].$ We are in the one-directional case, and it is known that $%
b(y)>0$. Taking reciprocals reverses the inequality, we have $\frac{1}{b(x)}%
\leq \frac{1}{h(x)},\quad \forall x\in \lbrack 0,r].$ Integrate over the
interval $[0,r]$
\begin{equation}
0<\int_{0}^{r}\frac{1}{b(y)}\mathrm{d}y\leq \int_{0}^{r}\frac{1}{h(y)}%
\mathrm{d}y<+\infty .  \label{bi}
\end{equation}%
Conversely,
\begin{equation*}
\int_{0}^{r}\frac{1}{b(y)}\mathrm{d}y<\infty \quad \mathbf{\nRightarrow }%
\quad \int_{0}^{r}\frac{1}{h(x)}\mathrm{d}x<\infty
\end{equation*}%
where $h(x)$ denotes the local minimum of $b$, so $h(x)<b(x)$ in general.
Thus $\frac{1}{h(x)}$ is larger and its integral is more likely to diverge.
\end{remark}

\begin{remark}
From (\ref{via}), we have
\begin{equation}
\left\langle b\left( x\right) ,\vec{n}\left( x\right) \right\rangle >0,\text{%
\qquad }\forall x\in \partial B\left( 0,r\right) ,\text{ }\forall r>0,
\label{fly}
\end{equation}%
where $\vec{n}\left( x\right) $ denotes the unit outward normal. Apparently,
the hypothesis (\ref{fly}) means that as soon as a trajectory of (\ref{1.1})
reaches the boundary of $B\left( 0,r\right) ,$\ then it leaves immediately.
Furthermore, note that $\partial B\left( 0,r\right) $ is $C^{1}$, so that $%
\vec{n}\left( x\right) $ is continuous and well-defined.
\end{remark}

\begin{remark}
The condition (\ref{via}) means that $\langle b(x),x\rangle \geq 0,$ $x\in
\mathbb{R}^{d}$ its physical means that the drift field is non-negative in
the radial direction, the distance from the particle to the origin is
non-decreasing in time, the system exhibits outward repulsion or neutral
motion, and the origin is an unstable or neutral equilibrium point.
\end{remark}

The following lemma is due to Berge maximum Theorem.

\begin{lemma}
\label{bc} Let $b:\mathbb{R}^{d}\rightarrow \mathbb{R}^{d},$ $d\geq 2$ be
measurable, and suppose the scalar function $F(x)=\langle b(x),x\rangle $ is
continuous on $\mathbb{R}^{d}$. Then $\varphi \in C\big((0,\infty )\big)$.
\end{lemma}

\begin{proof}
\noindent Step 1. For any $r>0$, $\omega \in \mathbb{S}^{d-1}$, define
\begin{equation*}
b_{r}(r,\omega )=\langle b(r\omega ),\omega \rangle ,\text{ }T(r,\omega
)=b_{r}(r,\omega )-\bar{b}(r).
\end{equation*}%
Substitute the definition of $T(r,\omega )$ into the formula for $\varphi
(r) $ defined in (\ref{B}):
\begin{equation*}
\varphi (r)=\bar{b}(r)+\inf_{\omega \in \mathbb{S}^{d-1}}\big\{%
b_{r}(r,\omega )-\bar{b}(r)\big\}=\inf_{\omega \in \mathbb{S}%
^{d-1}}b_{r}(r,\omega ).
\end{equation*}%
From polar coordinate identity $F(r\omega )=\langle b(r\omega ),r\omega
\rangle =r\,b_{r}(r,\omega )$, we rearrange to get
\begin{equation*}
b_{r}(r,\omega )=\frac{F(r\omega )}{r},\quad \forall r>0,\;\omega \in
\mathbb{S}^{d-1}.
\end{equation*}%
Thus
\begin{equation*}
\varphi (r)=\inf_{\omega \in \mathbb{S}^{d-1}}\frac{F(r\omega )}{r}.
\end{equation*}

\noindent Step 2. Continuity of $\phi (r,\omega )=b_{r}(r,\omega )$ on $%
(0,\infty )\times S^{d-1}$. Equip $(0,\infty )\times \mathbb{S}^{d-1}$ with
the product metric topology. Take any convergent sequence $(r_{n},\omega
_{n})\rightarrow (r_{0},\omega _{0})$ in this product space. Then $%
r_{n}\omega _{n}\rightarrow r_{0}\omega _{0}\in \mathbb{R}^{d}\setminus
\{0\} $. Since $F\in C(\mathbb{R}^{d})$,
\begin{equation*}
F(r_{n}\omega _{n})\rightarrow F(r_{0}\omega _{0}).
\end{equation*}%
This shows $\Phi (r,\omega ):=F(r\omega )$ is continuous on $(0,\infty
)\times \mathbb{S}^{d-1})$. Because $r>0$ is a continuous nonvanishing
function, the quotient
\begin{equation*}
\phi (r,\omega ):=\frac{\Phi (r,\omega )}{r}=b_{r}(r,\omega )
\end{equation*}%
is continuous on $(0,\infty )\times \mathbb{S}^{d-1}$.

\noindent Step 3. Continuity of the infimum function over compact sets. The unit
sphere $\mathbb{S}^{d-1}$ is a compact metric space. Fix an arbitrary $r_0>0$
and arbitrary $\varepsilon>0$. Choose a closed bounded interval $%
[a,b]\subset(0,\infty)$ such that $(r_0-\delta,r_0+\delta)\subset [a,b]$ for
some small $\delta\in(0,1)$. The product set $K=[a,b]\times \mathbb{S}^{d-1}$
is compact, hence the continuous function $\phi$ is uniformly continuous on $%
K$. There exists $\delta\in(0,1)$ such that for all $r\in(r_0-\delta,r_0+%
\delta)$ and all $\omega\in\mathbb{S}^{d-1}$,
\begin{equation*}
|\phi(r,\omega)-\phi(r_0,\omega)|<\frac{\varepsilon}{2}.
\end{equation*}
By compactness of $\mathbb{S}^{d-1}$, the continuous function $\omega\mapsto
\phi(r_0,\omega)$ attains its global minimum at some point $\omega_0\in%
\mathbb{S}^{d-1}$, i.e., $\phi(r_0,\omega_0)=\varphi(r_0)$. Applying the
uniform continuity estimate at $\omega=\omega_0$, we have
\begin{equation*}
\varphi(r)\le \phi(r,\omega_0) < \phi(r_0,\omega_0)+\frac{\varepsilon}{2}%
=\varphi(r_0)+\frac{\varepsilon}{2}.
\end{equation*}
Similarly, let $\omega_r\in\mathbb{S}^{d-1}$ be a minimizer such that $%
\phi(r,\omega_r)=\varphi(r)$. Again by uniform continuity,
\begin{equation*}
\varphi(r_0)\le \phi(r_0,\omega_r) < \phi(r,\omega_r)+\frac{\varepsilon}{2}%
=\varphi(r)+\frac{\varepsilon}{2}.
\end{equation*}
Combining the two inequalities yields
\begin{equation*}
|\varphi(r)-\varphi(r_0)|<\frac{\varepsilon}{2}<\varepsilon
\end{equation*}
whenever $|r-r_0|<\delta$. Since $r_0>0$ is arbitrary, $\varphi$ is
continuous on $(0,\infty)$.
\end{proof}

\begin{corollary}
Let $d\geq 2$, $b\in C(\mathbb{R}^{d};\mathbb{R}^{d})$ be locally bounded.
Then $\varphi \in C\big((0,\infty )\big)$.
\end{corollary}

\begin{proof}
Since $b$ is continuous on $\mathbb{R}^{d}$, the scalar function $%
F(x)=\langle b(x),x\rangle $ is continuous as the inner product of two
continuous vector-valued functions. The boundedness and measurability of $b$
are trivial for continuous bounded maps. Applying the previous Lemma yields $%
\varphi \in C((0,\infty ))$.
\end{proof}

Indeed, if $b$ is measurable and locally bounded$,$ there is no existence
result for classical solutions of the ODE (\ref{1.1}). Therefore, a
generalized notion of solution can be employed due to Filippov \cite{fili88}
as follows:

\begin{definition}
\label{filir}Let us consider a function $f:\mathbb{R}^{d}\rightarrow \mathbb{%
R}^{d}$ to which we associate the following set-valued map---called
Filippov's regularization of $F_{f}$%
\begin{equation*}
F_{f}\left( x\right) :=\bigcap_{m\left( N\right) =0}\bigcap_{\delta >0}%
\overline{\mathrm{co}}\,f\left( \left( x+\delta B\right) \backslash N\right)
\end{equation*}%
the first intersection is taken over all sets of $\mathbb{R}^{d}$, being
negligible with respect to the Lebesgue measure $m$, $B$ is the closed unit
ball and $\overline{\mathrm{co}}\,$ denotes the closed convex hull.
\end{definition}

For almost every $x$, one has $f(x)\in F_f(x)$ (see Proposition 2
in \cite{BYQ}); this does not determine the value of $F_f$ at a
specified exceptional point. For example, $f(x)=e_1$ for $x\ne0$ and
$f(0)=0$ gives $F_f(0)=\{e_1\}$.
For the drift of this paper, however, (A1) does imply $0\in F_b(0)$.
Here is the required argument. Write $M=\sup|b|$. Fix a ball
$B(0,\eta)$ and a Lebesgue-null set $N$, and let
$C=\overline{\mathrm{co}}\,b(B(0,\eta)\setminus N)$.
This is a nonempty compact convex set. If $0\notin C$, strict
separation gives a unit vector $u$ and $c>0$ with
$\langle u,b(x)\rangle\ge c$ for every $x\in B(0,\eta)\setminus N$.
The cone of nonzero points in that ball satisfying
$|x/|x|+u|<c/(2M)$ has positive Lebesgue measure.
For a point in this cone outside $N$, (A1) instead gives
\[
 \langle u,b(x)\rangle
 =\langle u+x/|x|,b(x)\rangle
      -\langle x/|x|,b(x)\rangle<c/2,
\]
a contradiction. (The case $M=0$ is excluded by (A1).)
Thus $0\in C$ for every such ball and null set, and their intersection
is $F_b(0)$. The constant zero path is consequently a Filippov solution
under (A1); this conclusion uses the outward structure, not the assigned
single value $b(0)=0$ alone.

\begin{definition}
\label{filis}An absolutely continuous solution $\xi \left( t\right) \in
\mathbb{R}^{d}$ is a Filippov solution of (\ref{1.1}) \emph{if and only if}
it is a solution of the following differential inclusion %\begin{equation}
%\xi ^{\prime }\left( t\right) \in F_{f}\left( \xi \left( t\right) \right) ,%
%\text{ }\xi \left( 0\right) =x,\text{ }\forall t\geq 0.  \label{fili}
%\end{equation}

\begin{equation}
\xi ^{\prime }(t)\in F_{f}(\xi (t))\quad \text{for a.e. }t\geq 0,\qquad \xi
(0)=x.  \label{fili}
\end{equation}
\end{definition}

Note that the set-valued map $F_{f}$ is upper semi continuous with compact
convex values. This implies that the differential inclusion (\ref{fili}) has
a nonempty set of (local) solutions (cf. \cite{ac84}). The map $x\rightarrow
F_{f}(x)$ is single-valued if and only if there exists a continuous function
$g$ which coincides almost everywhere with $f$. In this case we have $%
F_{f}(x)=\{g(x)\}$ for almost all $x\in \mathbb{R}^{d}$.

\begin{lemma}
\label{lem:filippov.radial} Assume (A1). For every $x\neq0$ and every $v\in
F_b(x)$,
\begin{equation*}
\left\langle v,\frac{x}{|x|}\right\rangle \geq \varphi(|x|).
\end{equation*}
Consequently, if $x(\cdot)$ is a Filippov solution and $x(t)\neq0$, then
\begin{equation*}
\frac{\mathrm{d}}{\mathrm{d}t}|x(t)| \geq \varphi(|x(t)|) \quad\text{for
a.e. such }t.
\end{equation*}
\end{lemma}

\begin{proof}
Fix $x\neq0$ and write $\widehat x=x/|x|$. For $y$ sufficiently close to $x$%
, boundedness of $b$ gives
\begin{align*}
\langle b(y),\widehat x\rangle &=\left\langle b(y),\frac{y}{|y|}%
\right\rangle +\left\langle b(y),\widehat x-\frac{y}{|y|}\right\rangle \\
&\geq \varphi(|y|)-M_b\left|\widehat x-\frac{y}{|y|}\right|.
\end{align*}
From Lemma \ref{bc}, the continuity of $q(y)=\langle b(y),y\rangle$ implies
continuity of $\varphi $ on compact subintervals of $(0,\infty)$. Letting $%
y\to x$ shows that all essential nearby values of $b$ lie, up to an error
tending to zero, in the closed half-space $\{v:\langle v,\widehat
x\rangle\geq\varphi(|x|)\}$. Closed convex hulls and intersections preserve
this half-space, which proves the first claim. The second follows from the
chain rule for the norm at nonzero points.
\end{proof}

Lemma~\ref{lem:filippov.radial} reduces the high-dimensional Filippov
differential inclusion to a one-dimensional differential inequality for the
radial component $r(t)=|x(t)|$. No information on the angular part is
required; only the radial lower bound $\varphi $ from Assumption (A1) is
used. This is the key step to characterize the escape behaviour of the
deterministic zero-noise limit.

When the conditions (A1) are satisfied, the corresponding differential
inclusion (DI) reads:
\begin{equation}
\dot{x}(t)\in F_{b}(x(t)),\quad x(0)=0,  \label{2.2}
\end{equation}%
where the inclusion is understood for almost every time and $F_b$
is the Filippov regularization defined above. Under (A1), $b$ is
bounded and Borel. Proposition \ref{pro1} gives directly the
zero-noise identification used here, without a further external
Filippov-identification assumption.

%Recall also \cite{ac84} that the set of absolutely continuous solutions of (\ref{fili})
%is nonempty. Furthermore, when the solutions are restricted to a given
%interval $[0,T]$, the solution set is compact for the uniform convergence
%topology and it is sequentially compact for the weak-$\left(
%W^{1,1}([0,T])\right) $ topology.
Recall from \cite{ac84} that the set of absolutely continuous solutions to
DI (\ref{fili}) is nonempty. Moreover, when restricted to a fixed interval $%
[0,T]$, this solution set is compact under the topology of uniform
convergence and sequentially compact under the weak topology of the Sobolev
space $W^{1,1}([0,T])$.

Under (A1), the non-uniqueness arises: the DI (\ref{2.2}) admits multiple
distinct Filippov solutions from the origin, including:

\begin{itemize}
\item The trivial zero solution: $x(t)\equiv 0$ for all $t\geq 0$;

\item Instantaneous escape solutions: $x(t)\neq 0$ for all $t>0$;

\item Delayed escape solutions: There exists $T>0$ such that $x(t)=0$ for $%
t\in \lbrack 0,T]$ and $x(t)\neq 0$ for $t>T$.
\end{itemize}

We consider a family of SDEs parameterized by noise intensity $\varepsilon
>0 $:
\begin{equation}
\mathrm{d}X_{t}^{\varepsilon }=b(X_{t}^{\varepsilon })\mathrm{d}t+\varepsilon \mathrm{%
d}W_{t},\quad X_{0}^{\varepsilon }=x_{0}\in \mathbb{R}^{d}.  \label{2.1}
\end{equation}%
%
%
%
%
%
%
%
%
%
%
%
%
%
%
%
%
%
%
%
%
%
%
%
%
%
%
%
%
%
%
%
%
%
%
%
%
%
%
%
%
%
%
%
%
%
%
%
%
%
%
%
%
%
%
%
%
%
%
%
%
%
%
%
%
%
%
%
%
%
%
%
%
%
%
%
%
%
%
%
%
%
%
%
%
%
%
%
%
%
%
%
%
%
%
%
%
%
%
%
%
%
%
%
%
%
%
%
%
%
%
%
%
%
%
%
%
%
%
%
%
%
%
%
%
%
%
%
%
%
%
%
%
%
%
%
%
%
%
%
%
%
%
%
%
%
%
%
%
%
%
%
%
%
%
%
%
%
%
%
%
%
%
%
%
%
%
%
%
%
%
%
%
%
%
%
%
%
%
%
%
%
%
%
%
%
%
%
%
%
%
%
%
%
%SDE (\ref{2.1}) possess the very crucial property that a unique strong
%solution exists with the only assumption that $b$ is bounded and measurable
%(cf. \cite{Z}). It is possible to prove by a tightness argument of
%Prokhorov's type that the laws of solutions to (\ref{2.1}) are in a
%relatively (weakly sequentially) compact set of probabilities. Theorem 4 in
%\cite{BYQ} has proven that as $\varepsilon \rightarrow 0^{+},$ there exists
%a sub-sequence $\left\{ \varepsilon _{n}\right\} _{n\geq 1}$ such that $%
%X^{\varepsilon _{n}}\left( \cdot \right) $ converges in law, as $\varepsilon
%_{n}\rightarrow 0^{+}$, to some $X^{0}\left( \cdot \right) $ which is almost
%surely concentrated on the set of all Filippov solutions of ODE (\ref{1.1}).
For bounded Borel drift and each fixed nonzero constant diffusion,
strong existence and pathwise uniqueness hold. In dimension $d\ge2$
this is the constant-diffusion case of Theorem 1 in \cite{Ver1981}:
the diffusion matrix is continuous, uniformly elliptic, and constant,
so its Sobolev and Lipschitz requirements are satisfied. Equivalently,
rescale $X/\varepsilon$ to identity diffusion and bounded drift
$b(\varepsilon\,\cdot)/\varepsilon$ for that fixed $\varepsilon$.
The one-dimensional bounded-drift version is also recalled in
\cite{fmp23}. Proposition \ref{pro1} proves relative compactness as
the noise vanishes and identifies every cluster as a Filippov path
law. Thus the later uses of this fact do not require an unverified
version of Theorem 4 of \cite{BYQ}.

Let us recall the standard $d$-dimensional Bessel process (denoted $BES_{d}$%
) defined as the Euclidean norm of a $d$-dimensional standard Brownian
motion $W_{t}=(W_{t}^{1},\dots ,W_{t}^{d})$:
\begin{equation*}
R_{t}=\left\vert W_{t}\right\vert =\sqrt{(W_{t}^{1})^{2}+\cdots
+(W_{t}^{d})^{2}},\quad t\geq 0,
\end{equation*}%
where the state space $R_{t}\geq 0$, is a one-dimensional non-negative
diffusion process; The initial condition $R_{0}=0$ (starts from the origin).

\begin{definition}[Standard $d$-dimensional Bessel process ($BES_d$)]
\label{bes} In this paper $d$ is an integer with $d\ge2$.
The standard Bessel process of dimension $d$ is the nonnegative process
with the law of the norm of $d$-dimensional Brownian motion started at
zero. It satisfies
\begin{equation}
 \mathrm{d}R_t=\frac{d-1}{2R_t}\,\mathrm{d}t+\mathrm{d}B_t,\qquad R_0=0,
 \qquad \int_0^T R_t^{-1}\mathrm{d}t<\infty\quad(T<\infty).
 \label{besd}
\end{equation}
The solution class here consists of paths with $R_t>0$ for every
$t>0$, almost surely. Here $B$ is a one-dimensional Brownian motion.
The equation is understood
in integrated Ito form, and its nonnegative solution is strongly unique.
The restriction on dimension is part of this definition; no assertion
about the same unreflected equation for all real $d>0$ is made.
\end{definition}

Indeed, the Gaussian density in polar coordinates gives
$\mathbb E|W_t|^{-1}=C_d t^{-1/2}<\infty$ for $d\ge2$.
Point polarity and It\^{o}'s formula away from zero, followed by passage
of the lower time endpoint to zero, give \eqref{besd}, with
$B_t=\int_0^t\langle W_s/|W_s|,\mathrm{d}W_s\rangle$.
For two nonnegative solutions driven by the same Brownian motion,
on $[a,t]$ their difference $D$ satisfies
$(D^2)'=-(d-1)D^2/(R_1R_2)\le0$ whenever both radii are positive.
The same inequality follows by localization, and passage $a\downarrow0$
gives pathwise uniqueness for solutions which are positive at positive
times. This is the solution class in \eqref{besd}; weak existence above
and the pathwise-uniqueness implication for strong existence give the
stated strong solution.
In dimension one the Brownian norm instead has the reflected
representation $R_t=B_t+L_t$, with a nondecreasing regulator increasing
only at $R=0$. Without this term the equation would give ordinary
Brownian motion, which is not nonnegative.

%
%
%
%
%
%Before introducing the general condition, the following assumptions form the
%foundation of our analysis (cf. \cite{os1898}):confirming Osgood
%non-uniqueness.

%
%\subsection{Radial Process Analysis and Stochastic Comparison}
%
%\label{sect3}

We begin our analysis from the radial process $R_t^\varepsilon = \left|
X_t^\varepsilon \right|$, which captures the distance of the SDE solution
from the origin. We first derive the dynamics of this radial process and
then construct a one-dimensional comparison process that stochastically
dominates it from below.

%Since the function $\left\vert x\right\vert $ is not twice continuously
%differentiable at the origin (a critical technical obstacle), we use the
%standard technique of applying It\^{o}'s formula to the smooth function $%
%\left\vert x\right\vert ^{2}$ first, then differentiating to obtain the
%dynamics of $R_{t}^{\varepsilon }=\sqrt{\left\vert X_{t}^{\varepsilon
%}\right\vert ^{2}}$.

\begin{theorem}
\label{thm:origin-inaccessible} Let $d\geq 2$, let $b$ be bounded and Borel
measurable, and fix $\varepsilon >0$. Consider the unique strong solution $%
X^{\varepsilon }$ to the SDE
\begin{equation}
\mathrm{d}X_{t}^{\varepsilon }=b(X_{t}^{\varepsilon })\,\mathrm{d}%
t+\varepsilon \,\mathrm{d}W_{t},\qquad X_{0}^{\varepsilon }=x_{0}.
\label{sde0}
\end{equation}%
If $x_{0}=0$, define the first positive time return time
\begin{equation*}
\tau _{0}^{+}:=\inf \bigl\{t>0:R_{t}^{\varepsilon }=0\bigr\}.
\end{equation*}%
Then
\begin{equation}
\mathbb{P}_{0}\bigl(\tau _{0}^{+}=\infty \bigr)=1.  \label{inazero}
\end{equation}%
%
%
%
%
%
%
%
%
%
%
%
%
%
%
%
%
%
%
%
%
%
%\begin{enumerate}
%%\item If $x_0\neq 0$, then
%%\begin{equation*}
%%\mathbb{P}_{x_0}\bigl(\exists\, t\geq 0: X_t^{\varepsilon}=0\bigr)=0.
%%\end{equation*}
%
%\item
%\end{enumerate}
\end{theorem}

\begin{proof}
Fix $T<\infty$ and $\varepsilon>0$, and write $M=\sup_x|b(x)|$.
The stochastic exponential
\[
 Z_T=\exp\left\{-\varepsilon^{-1}\int_0^T
 \langle b(X_s^\varepsilon),\mathrm{d}W_s\rangle
 -\frac{1}{2\varepsilon^2}\int_0^T|b(X_s^\varepsilon)|^2\,\mathrm{d}s\right\}
\]
has mean one by Novikov's condition, since its quadratic-variation
integral is bounded by $M^2T/\varepsilon^2$. The measure
$d\mathbb Q_T=Z_T\,d\mathbb P$ is equivalent to $\mathbb P$ on
$\mathcal F_T$, and under it $X_t^\varepsilon=X_0^\varepsilon+
\varepsilon\widetilde W_t$ for $0\le t\le T$.
A point is polar for Brownian motion in dimension $d\ge2$ (see Proposition
2.7 in \cite{dm}); starting at
that point, Brownian motion has no return at a positive time. Thus the
corresponding non-hitting or non-return assertion holds up to $T$ under
$\mathbb Q_T$, and hence under $\mathbb P$. Applying this argument to
$T=1,2,\ldots$ proves \eqref{inazero} on the entire positive time axis.
No equivalence of measures on an infinite time interval is used.
\end{proof}

Immediately, we have

\begin{lemma}
\label{lem:radial-comparison-origin} Assume (A1), and extend $\varphi$ to
$\mathbb R$ by $\varphi(r)=0$ for $r\le0$. Let
$\mathrm{d}X_t^\varepsilon=b(X_t^\varepsilon)\,\mathrm{d}t+\varepsilon\,\mathrm{d}W_t$,
$X_0^\varepsilon=0$, and set $R_t^\varepsilon=|X_t^\varepsilon|$,
$\theta_t^\varepsilon=X_t^\varepsilon/R_t^\varepsilon$ for $t>0$.
Define $\theta_0^\varepsilon$ arbitrarily on the unit sphere and put
$B_t=\int_0^t\langle\theta_s^\varepsilon,\mathrm{d}W_s\rangle$.
Then $B$ is a Brownian motion. The scalar SDE
\[
 \mathrm{d}r_t^\varepsilon=\varphi(r_t^\varepsilon)\,\mathrm{d}t+\varepsilon\,\mathrm{d}B_t,
 \qquad r_0^\varepsilon=0,
\]
has a unique strong solution, and, for every fixed $\varepsilon>0$,
\begin{equation}
 \mathbb P\{R_t^\varepsilon\ge r_t^\varepsilon\text{ for all }t\ge0\}=1.
 \label{re1}
\end{equation}
\end{lemma}

\begin{proof}
The extended $\varphi$ is bounded and Borel measurable. It need not be
continuous at zero. Strong well-posedness for bounded Borel drift and
constant nonzero diffusion applies to the scalar equation.
Theorem \ref{thm:origin-inaccessible} gives $R_t^\varepsilon>0$ for every
$t>0$, on one probability-one event. The bounded progressive integrand
in $B$ has unit length for Lebesgue-almost every time; Levy's
characterization gives the Brownian property.

We first justify the radial equation all the way from time zero. Write
$M=\sup_x|b(x)|$ and $c=\varepsilon^2(d-1)/2>0$. For
$f_\eta(x)=(|x|^2+\eta^2)^{1/2}$ one has
\[
 \nabla f_\eta(x)\cdot b(x)=\frac{q(x)}{f_\eta(x)}\ge0,
 \qquad
 \Delta f_\eta(x)=\frac{d-1}{f_\eta(x)}+
                    \frac{\eta^2}{f_\eta(x)^3}.
\]
It\^{o}'s formula and $f_\eta(x)\le|x|+\eta$ imply
\[
 c\,\mathbb E\int_0^T\frac{\mathrm{d}t}{(R_t^{\varepsilon\,2}+\eta^2)^{1/2}}
 \le \mathbb E f_\eta(X_T^\varepsilon)-\eta
 \le MT+\varepsilon\sqrt{dT}.
\]
Monotone convergence proves $\int_0^T(R_t^\varepsilon)^{-1}\mathrm{d}t<\infty$
almost surely. Apply the ordinary radial It\^{o} formula on $[a,t]$, and
let $a\downarrow0$. The bounded radial drift, the integrable geometric
drift, and the continuous martingale then give
\begin{equation}
 R_t^\varepsilon=\int_0^t\left(v_s+\frac{c}{R_s^\varepsilon}\right)\mathrm{d}s
                       +\varepsilon B_t,
 \qquad v_s=\langle b(X_s^\varepsilon),\theta_s^\varepsilon\rangle
                    \ge\varphi(R_s^\varepsilon)>0.
 \label{4.1}
\end{equation}

The difference $D=R^\varepsilon-r^\varepsilon$ is absolutely continuous,
with
$D'=v+c/R^\varepsilon-\varphi(r^\varepsilon)$ almost everywhere.
For each path, continuity at zero gives a positive initial interval on
which $R^\varepsilon<c/(M+1)$. Thus $D'>1$ almost everywhere there and
$D(t)>0$ for all its positive times.
Suppose that $D$ subsequently becomes negative, and let
$\sigma=\inf\{t>0:D(t)<0\}$. Then $\sigma>0$ and
$D(\sigma)=0$, while
$R_\sigma^\varepsilon=r_\sigma^\varepsilon=\rho>0$.
The function
\[
 t\longmapsto \varphi(R_t^\varepsilon)+c/R_t^\varepsilon
                         -\varphi(r_t^\varepsilon)
\]
is continuous near $\sigma$, by Lemma \ref{bc}, and has value $c/\rho>0$
there. Consequently $D'>c/(2\rho)$ almost everywhere in a sufficiently
small interval about $\sigma$. This contradicts the definition of the
first entrance into $\{D<0\}$. Hence $D\ge0$ for all times. Intersecting
the probability-one events for integer time horizons completes the proof.
\end{proof}

\begin{remark}
\label{rem:role-of-A1} The reciprocal-integrability condition in (A1)
is not used in the fixed-noise comparison proof; it is used later to
construct the positive scalar zero-noise trajectory. Point polarity
alone follows from boundedness, $d\ge2$, and finite-time Girsanov.
In contrast, \eqref{re1} is a same-Brownian comparison and uses the
radial lower bound, its continuity at positive radii, and the positive
geometric drift. Point polarity does not by itself prove that ordering.
\end{remark}

\begin{remark}
\label{rem:geometric-drift} The term $\frac{\varepsilon ^{2}(d-1)}{2R_{t}}$
appearing in the radial SDE (\ref{4.1}) is known as the \emph{geometric
drift term}. It is not a physical force exerted by an external vector field,
but an equivalent deterministic effect arising purely from the nonlinear
coordinate transformation $x\mapsto |x|$ combined with It\^{o}'s stochastic
calculus. For a standard $d$-dimensional Brownian motion $W_{t}$ ($%
\varepsilon=1$, no deterministic drift), its radial component $R_{t}=|W_{t}|$
satisfies the Bessel SDE. This drift term originates exclusively from the
Laplacian of the radial function $f(x)=|x|$, namely $\Delta f(x)=\frac{d-1}{%
|x|}$, which reflects the curvature of spherical coordinates.

Intuitively, this outward drift is a statistical consequence of the growing
surface area of spheres in higher dimensions. A Brownian particle moves
isotropically in all directions; at radius $r$, the spherical surface area
scales as $r^{d-1}$, so there are more ``outward'' directions than
``inward'' directions. On average, this imbalance produces a net outward
tendency in the radial coordinate, even in the absence of any deterministic
force.

The strength of the geometric drift is entirely determined by the dimension $%
d$ (the following discussion refers to the standard Bessel process with $%
\varepsilon=1$):

\begin{itemize}
\item In dimension $d=1$, the radius of Brownian motion is reflected
Brownian motion and its semimartingale decomposition contains local time at
the origin; it is not an unconstrained one-dimensional Brownian motion,
which is recurrent and hits the origin infinitely often.

\item In dimension $d=2$, the geometric drift is $\frac{1}{2R_{t}}$. The
origin is a polar set (cannot be hit in finite time), but the process is
still neighborhood-recurrent.

\item In dimensions $d\geq 3$, the stronger geometric drift makes the
process transient: it never hits the origin and eventually drifts to
infinity (see Figure \ref{fig:W}).
\end{itemize}

\begin{figure}[]
\centering
\includegraphics[width=0.8\textwidth]{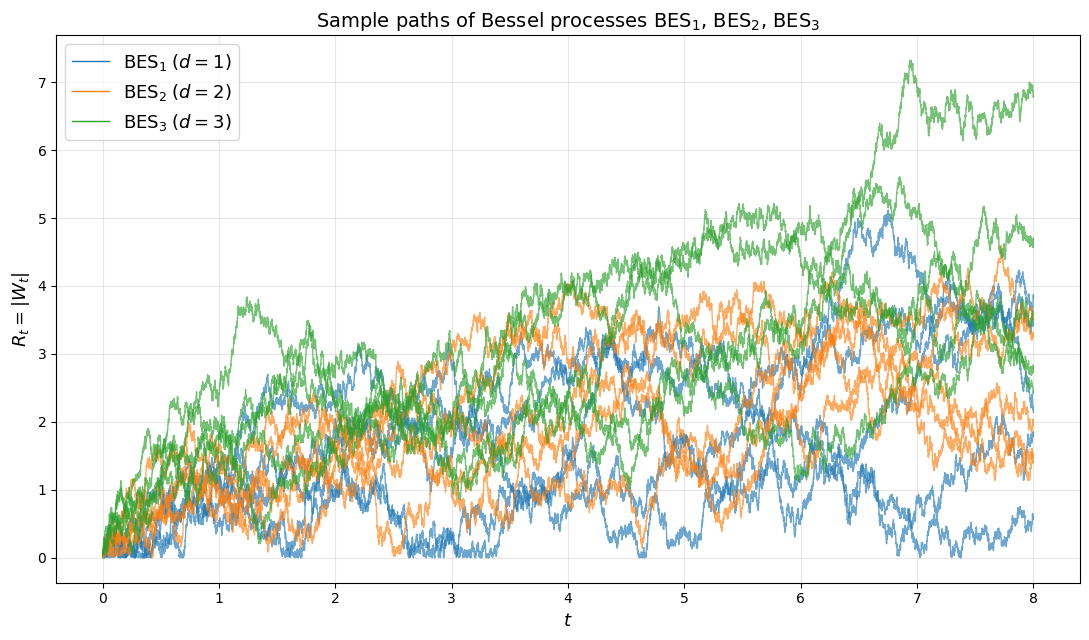}
\caption{Trajectory diagram of the $d$-dimensional Bessel process}
\label{fig:W}
\end{figure}

This geometric drift is fundamentally different from the radial component $%
\langle b(x),\theta \rangle$ of the drift field: the latter comes from the
external vector field $b(x)$ and can be either attractive or repulsive,
while for $d\ge 2$ the geometric drift is always positive (outward) and
diverges like $1/r$ near the origin (see Appendix \ref{app}).
\end{remark}

\section{Main results}
\label{sect3}
This section develops radial escape under (A1), followed by angular
identities and selection under the additional structure (A5).
Instantaneous escape means that every positive time has positive
radius, not that the process traverses a fixed positive radius in
zero time. The planar cubic and square-root/sign examples have
separate finite-time path-law selection proofs. Marginal supports are
compact and satisfy a one-sided endpoint inclusion; their Hausdorff
dimension is bounded by $d-1$ only under the independent Lipschitz
radial-graph hypothesis. The plots illustrate these examples and are
not inputs to the proofs.

\subsection{1D measurable drift}

\label{3.1}

Throughout this paper, all integrals involving the locally bounded
measurable drift coefficient $b \in L_{\mathrm{loc}}^{\infty}(\mathbb{R})$
are interpreted in the Lebesgue sense. Building upon insights gained from an
illustrative case, we extend our analysis to the higher-dimensional
framework.

\label{onefili} Assume $f:\mathbb{R}\rightarrow \mathbb{R}$ is locally
bounded, Lebesgue measurable with $f(0)=0.$ Suppose $f>0$ on $(0,r]$ (resp. $%
f<0$ on $[-r,0)$) almost everywhere for some $r>0$. Define for every $x\in
\mathbb{R}$
\begin{equation}
m(f)(x)=\sup_{\delta >0}\mathrm{essinf}_{[x-\delta ,x+\delta ]}f(s),\quad
M(f)(x)=\inf_{\delta >0}\mathrm{esssup}_{[x-\delta ,x+\delta ]}f(s),
\label{1dfili}
\end{equation}%
and the Filippov set-valued map
\begin{equation*}
F_{f}(x)=\big[m(f)(x),M(f)(x)\big].
\end{equation*}

\subsubsection{Example}

We consider the sign function
\begin{equation}
f(x)=\mathrm{sgn}(x)=%
\begin{cases}
1, & x>0, \\
0, & x=0, \\
-1, & x<0.%
\end{cases}
\label{sgf}
\end{equation}

\paragraph{Case 1: $x>0$}

Choose $\delta\in(0,x)$, then $[x-\delta,x+\delta]\subset(0,\infty)$ and $%
f\equiv 1$ on this neighborhood.
\begin{equation*}
m(f)(x)=1,\quad M(f)(x)=1,\quad F_f(x)=\{1\}.
\end{equation*}

\paragraph{Case 2: $x<0$}

Choose $\delta\in(0,-x)$, then $[x-\delta,x+\delta]\subset(-\infty,0)$ and $%
f\equiv -1$ on this neighborhood.
\begin{equation*}
m(f)(x)=-1,\quad M(f)(x)=-1,\quad F_f(x)=\{-1\}.
\end{equation*}

\paragraph{Case 3: $x=0$}

For any $\delta>0$, the interval $[-\delta,\delta]$ contains positive and
negative arguments: $essinf_{[-\delta,\delta]}f=-1$, $esssup_{[-\delta,%
\delta]}f=1$. Taking supremum/infimum over $\delta>0$ yields
\begin{equation*}
m(f)(0)=-1,\quad M(f)(0)=1,\quad F_f(0)=[-1,1].
\end{equation*}

\begin{enumerate}
\item Take any $\delta _{+}>0$. For $s\in (0,\delta _{+}]$, $m(f)(s)\equiv 1$%
, so
\begin{equation*}
\int_{0}^{\delta _{+}}\frac{1}{m(f)(s)}\,\mathrm{d}s=\int_{0}^{\delta
_{+}}1\,\mathrm{d}s=\delta _{+}<\infty .
\end{equation*}

\item Take any $\delta _{-}>0$. For $s\in \lbrack -\delta _{-},0)$, $%
M(f)(s)\equiv -1$, so $|M(f)(s)|=1$. Then
\begin{equation*}
\int_{-\delta _{-}}^{0}\frac{1}{\big|M(f)(s)\big|}\,\mathrm{d}%
s=\int_{-\delta _{-}}^{0}1\,\mathrm{d}s=\delta _{-}<\infty .
\end{equation*}
\end{enumerate}

\subparagraph{Family 1: Trivial stationary solution.}

Define $x_{0}(t)\equiv 0$ for all $t\geq 0$. $x_{0}$ is constant, hence
absolutely continuous, with $\dot{x}_{0}(t)\equiv 0$ a.e. $0\in \lbrack
-1,1]=F_{f}(0)$, so $\dot{x}_{0}(t)\in F_{f}(x_{0}(t))$ a.e., and $%
x_{0}(0)=0 $. Thus $x_{0}$ is a valid stationary Filippov solution.

\subparagraph{Family 2: Immediately departing solutions.}

For $x>0$, $F_{f}(x)=\{1\}$ gives exactly one admissible velocity $v_{+}=1$.
The ODE $\dot{x}=1$, $x(0)=0$ has unique solution $x^{+}(t)=t,\quad \forall
t\geq 0.$ We have $x^{+}(t)>0$ for all $t>0$, so $x^{+}$ is a right
immediately departing solution. For $x<0$, $F_{f}(x)=\{-1\}$ gives exactly
one admissible velocity $\dot{x}=-1$. The ODE $\dot{x}=-1$, $x(0)=0$ has
unique solution $x^{-}(t)=-t,\quad \forall t\geq 0.$ We have $x^{-}(t)<0$
for all $t>0$, so $x^{-}$ is a left immediately departing solution.

Fix any delay threshold $\tau>0$. We construct two types of delayed
trajectories.

\subparagraph{Rightward delay solution}

\begin{equation*}
x_{\tau }^{+}(t)=%
\begin{cases}
0, & 0\leq t\leq \tau , \\
t-\tau , & t>\tau .%
\end{cases}%
\end{equation*}%
Constant zero on $[0,\tau ]$, absolutely continuous shift $t\mapsto t-\tau $
on $[\tau ,\tau +T]$ for any $T>0$; continuity at $t=\tau $ holds ($x_{\tau
}^{+}(\tau )=0=x^{+}(0)$).

\subparagraph{Leftward delay solution}

\begin{equation*}
x_{\tau}^-(t)=
\begin{cases}
0, & 0\leq t\leq \tau, \\
-(t-\tau), & t>\tau.%
\end{cases}%
\end{equation*}
Identical absolute continuity and inclusion checks apply: $\dot{x}%
_\tau^-(t)=-1\in\{-1\}=F_f(x_\tau^-(t))$ for $t>\tau$, and $0\in F_f(0)$ on $%
(0,\tau)$. This is a left delay solution.

The set of delay thresholds $\tau\in(0,\infty)$ is uncountable, producing an
uncountable family $\{x_\tau^+,x_\tau^-\}_{\tau>0}$ of distinct delay
solutions.

\begin{enumerate}
\item The trivial solution $x_0(t)\equiv0$ never leaves the origin, so it
differs from all immediately departing and delay solutions.

\item Immediately departing solutions satisfy $x(t)\neq0$ for all $t>0$, so
they cannot coincide with delay solutions (which vanish on an initial
interval $[0,\tau]$).

\item The unique right immediately departing trajectory $x^+(t)=t$ satisfies
$x^+(t)>0$ for all $t>0$, while the unique left immediately departing
trajectory $x^-(t)=-t$ satisfies $x^-(t)<0$ for all $t>0$; hence $%
\{x^+\}\cap\{x^-\}=\emptyset$.

\item Any two distinct delay solutions (different $\tau$, or opposite
direction) yield distinct trajectories.
\end{enumerate}

Now consider the one-dimensional SDE with small diffusion noise parameter $%
\varepsilon >0$
\begin{equation}
\mathrm{d}X_{t}^{\varepsilon }=b_{1}(X_{t}^{\varepsilon })\mathrm{d}t+\varepsilon \mathrm{d}W_{t},\quad
X_{0}^{\varepsilon }=0,\quad t\geq 0,  \label{eq:sde_eps}
\end{equation}%
where the drift field $b_{1}$ is defined by \eqref{sgf} and $W_{t}$ is a
standard one-dimensional Brownian motion on a filtered probability space $%
(\Omega ,\mathcal{F},\{\mathcal{F}_{t}\}_{t\geq 0},\mathbb{P})$. For all $%
x\neq 0$, we verify the radial inner product condition:
\begin{equation}
x\cdot b_{1}(x)=x\cdot \mathrm{sgn}(x)=\left\vert x\right\vert >0.
\label{sc}
\end{equation}%
This is a \textit{necessary condition} for immediate exit from the origin,
but not sufficient for any fixed $\varepsilon >0$. Set $\varepsilon =0$ to
obtain the ordinary differential inclusion
\begin{equation*}
\dot{x}(t)\in F_{\mathrm{sgn}}(x(t)),\quad x(0)=0,\quad F_{\mathrm{sgn}%
}(x)=\left\{
\begin{array}{cc}
\left\{ 1\right\} , & \text{if }x>0; \\
\left[ -1,1\right] , & \text{if }x=0; \\
\left\{ -1\right\} , & \text{if }x<0.%
\end{array}%
\right.
\end{equation*}%
The right-hand side is discontinuous at $x=0$, yielding two distinct
deterministic forward trajectories emanating from the initial point: Right
branch: $x_{+}(t)=t$, valid for $t\geq 0$; here $x_{+}(t)>0$, so $\dot{x}%
_{+}=1$. Left branch: $x_{-}(t)=-t$, valid for $t\geq 0$; here $x_{-}(t)<0$,
so $\dot{x}_{-}=-1$. These are the two limiting orbits $t$ and $-t$.

Let $R_{t}^{\varepsilon }=|X_{t}^{\varepsilon }|$. Tanaka's formula for
continuous semimartingales gives
\begin{equation*}
R_{t}^{\varepsilon }=\int_{0}^{t}\mathrm{sgn}(X_{s}^{\varepsilon })\mathrm{d}%
X_{s}^{\varepsilon }+L_{t}^{0}(X^{\varepsilon }),
\end{equation*}%
where $L_{t}^{0}(X^{\varepsilon })$ denotes the local time of $%
X^{\varepsilon }$ at the origin $0$. Substitute \eqref{eq:sde_eps}
\begin{equation*}
\begin{aligned} R_t^\varepsilon &= \int_0^t
\operatorname{sgn}(X_s^\varepsilon)\big(\operatorname{sgn}(X_s^%
\varepsilon)\mathrm{d}s + \varepsilon \mathrm{d}W_s\big) + L_t^0(X^\varepsilon) \\ &= \int_0^t
\big(\operatorname{sgn}(X_s^\varepsilon)\big)^2 \mathrm{d}s + \varepsilon \int_0^t
\operatorname{sgn}(X_s^\varepsilon)\,\mathrm{d}W_s + L_t^0(X^\varepsilon).
\end{aligned}
\end{equation*}%
Since $\big(\mathrm{sgn}(x)\big)^{2}=1$ for all $x\neq 0$, the differential
form reads
\begin{equation}
\mathrm{d}R_{t}^{\varepsilon }=\mathrm{d}t+\varepsilon \mathrm{sgn}%
(X_{t}^{\varepsilon })\,\mathrm{d}W_{t}+\mathrm{d}L_{t}^{0}(X^{\varepsilon
}).  \label{eq:radial_eps}
\end{equation}%
The martingale perturbation tends to zero uniformly in probability
on every finite time interval. More precisely, the occupation-time
formula gives zero Lebesgue time on $\{X^\varepsilon=0\}$, so
$B_t=\int_0^t\operatorname{sgn}(X_s^\varepsilon)dW_s$ is Brownian and
Tanaka's equation is $R_t^\varepsilon=t+\varepsilon B_t+L_t$.
Its reflection regulator satisfies
\[
 L_t=\sup_{s\le t}[-s-\varepsilon B_s]_+
              \le\varepsilon\sup_{s\le t}|B_s|.
\]
Consequently $R^\varepsilon\to(t\mapsto t)$ uniformly in probability
on $[0,T]$. The bounded-drift integral equation gives tightness of
$X^\varepsilon$ on that path space. Every cluster path is continuous
and has $|x(t)|=t$ simultaneously for all $t$, and therefore has a
constant sign on $(0,T]$. Reflection invariance gives equal weights.
It follows that the entire family of path laws satisfies
\[
 \mathcal L(X^\varepsilon|_{[0,T]})\Longrightarrow
     \tfrac12\delta_{(t)_{0\le t\le T}}
       +\tfrac12\delta_{(-t)_{0\le t\le T}}.
\]
This is weak convergence of path laws, not convergence in probability
to a fair random branch on the original Brownian space. Indeed, such
a probability limit would have the form $tS$, with $S=\pm1$ with equal
probability. Adaptedness would make $S$ measurable with respect to
$\mathcal F_t$ for every $t>0$, hence to the Brownian initial germ
$\mathcal F_{0+}$. Blumenthal's zero-one law makes that germ trivial,
contradicting the two nonzero weights. Figure \ref{fig:b_1} illustrates
the two branches; it does not assert their pathwise selection in the
original coupling. In this example the dimension is one, so the higher-dimensional
point-polarity theorem is not being invoked.

\begin{figure}[]
\centering          % 图片居中
\includegraphics[width=0.8\textwidth]{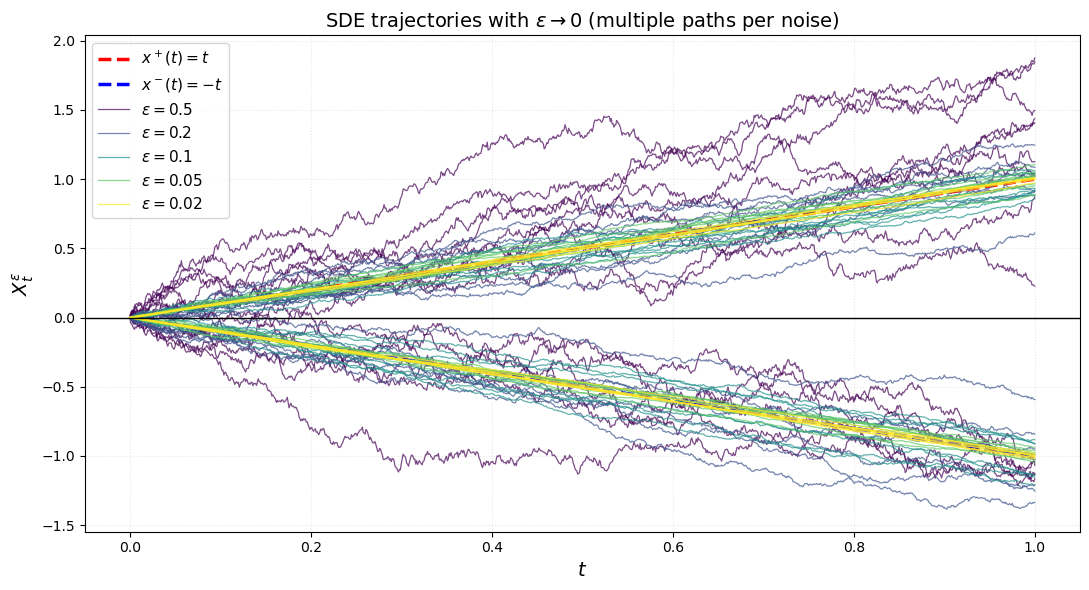} % 宽度 = 页面宽度的80%
\caption{Random perturbation of $1$-dim SDE with drift $b_{1}$}
\label{fig:b_1}
\end{figure}
The plotted noise intensities are $\varepsilon=0.5,0.2,0.1,0.05$
and $0.02$. The reference lines are $x^+(t)=t$ and $x^-(t)=-t$.
The proved limiting object is the equal-weight probability measure
on these two paths. The figure does not demonstrate almost sure or
in-probability convergence to a random branch under the original
common-Brownian coupling.

\subsection{Physical mechanism and dynamic behavior}

\label{3.2}

In this subsection, we investigate the limiting behavior of solutions to the
high-dimensional SDE (\ref{2.1}) as the intensity parameter approaches zero,
under specified regularity and structural conditions. To characterize the
asymptotic dynamics of solution trajectories, we decompose the solution into
radial and angular components. We demonstrate that, in the
vanishing-intensity limit, the limiting solutions concentrate on
instantaneous escape trajectories. Furthermore, leveraging tools from
differential geometry, we rigorously derive the limiting angular dynamics
and elucidate the underlying physical mechanism governing this behavior.

\begin{definition}
Let $d\geq 2$, denote the unit sphere $\mathbb{S}^{d-1}=\{\theta \in \mathbb{%
R}^{d}:|\theta |=1\}$. For any $\theta \in \mathbb{S}^{d-1}$ and vector $%
v\in \mathbb{R}^{d}$, the tangential projection onto the tangent space of $%
\mathbb{S}^{d-1}$ is given by
\begin{equation*}
\mathbb{P}_{\theta }^{\perp }v:=v-\bigl(\theta \cdot v\bigr)\theta .
\end{equation*}%
For $V\in C^{\infty }(\mathbb{S}^{d-1})$, $\nabla _{\mathbb{S}}V(\theta )$
stands for the spherical tangential gradient of $V$ at $\theta $.
\end{definition}

Next, we adopt the standard high-dimensional polar decomposition: $x=r\theta
$, where $r=\Vert x\Vert \geq 0$ denotes the radial coordinate and $\theta
\in \mathbb{S}^{d-1}$ represents the angular component on the unit sphere in
$\mathbb{R}^{d}$. Any vector can be uniquely decomposed as:
\begin{equation}
v=\underbrace{(v\cdot \theta )\theta }_{\text{radial component}}+\underbrace{%
v-(v\cdot \theta )\theta }_{\text{tangential/angular component}}.
\label{dric}
\end{equation}%
Now we consider the following
\begin{equation}
\mathrm{d}X_{t}^{\varepsilon }=b(X_{t}^{\varepsilon })\,\mathrm{d}%
t+\varepsilon \mathrm{d}W_{t},\quad X_{t}^{\varepsilon }=R_{t}^{\varepsilon
}\theta _{t}^{\varepsilon },\;R_{t}^{\varepsilon }>0,\;\theta
_{t}^{\varepsilon }\in \mathbb{S}^{d-1},X_{0}^{\varepsilon }=0.  \label{sde2}
\end{equation}%
The radial part has already been derived in (\ref{4.1}). We now give the
rigorous derivation of the SDE (\ref{sde2}) for the angular process $\theta
_{t}^{\varepsilon }$

The projection onto the unit sphere is given by: $\theta _{t}^{\varepsilon }=%
\frac{X_{t}^{\varepsilon }}{|X_{t}^{\varepsilon }|}$. Apply the
multi-dimensional It\^{o} formula to $\mathrm{d}\theta _{t}^{\varepsilon }=%
\mathrm{d}\left( \frac{X_{t}^{\varepsilon }}{R_{t}^{\varepsilon }}\right) $.
Using the quotient rule for It\^{o} processes:
\begin{equation*}
\mathrm{d}\left( \frac{X}{r}\right) =\frac{1}{r}\mathrm{d}X-\frac{X}{r^{2}}%
\mathrm{d}r+\frac{X}{r^{3}}\mathrm{d}\langle r,r\rangle _{t}-\frac{1}{r^{2}}%
\mathrm{d}\langle X,r\rangle _{t}.
\end{equation*}%
After simplifying the cross-variation terms on the sphere, the final form of
the standard spherical SDE is

\begin{equation}
\mathrm{d}\theta _{t}^{\varepsilon }=\frac{1}{R_{t}^{\varepsilon }}\mathbb{P}%
_{\theta _{t}^{\varepsilon }}^{\perp }\big(b(R_{t}^{\varepsilon }\theta
_{t}^{\varepsilon })\big)\mathrm{d}t+\frac{\varepsilon }{R_{t}^{\varepsilon }%
}\mathbb{P}_{\theta _{t}^{\varepsilon }}^{\perp }\mathrm{d}W_{t}-\frac{%
\varepsilon ^{2}(d-1)}{2\left( R_{t}^{\varepsilon }\right) ^{2}}\theta
_{t}^{\varepsilon }\mathrm{d}t,  \label{asde}
\end{equation}%
Strong coupling between radial and angular components:
\begin{equation}
\left\{
\begin{array}{lll}
\mathrm{d}R_{t}^{\varepsilon } & = & \left[ b(R_{t}^{\varepsilon }\theta
_{t}^{\varepsilon })\cdot \theta _{t}^{\varepsilon }+\frac{\varepsilon
^{2}(d-1)}{2R_{t}^{\varepsilon }}\right] \mathrm{d}t+\varepsilon \theta
_{t}^{\varepsilon }\cdot \mathrm{d}W_{t}, \\
\mathrm{d}\theta _{t}^{\varepsilon } & = & \left[ \frac{1}{%
R_{t}^{\varepsilon }}\mathbb{P}_{\theta _{t}^{\varepsilon }}^{\perp }\big(%
b(R_{t}^{\varepsilon }\theta _{t}^{\varepsilon })\big)-\frac{\varepsilon
^{2}(d-1)}{2\left( R_{t}^{\varepsilon }\right) ^{2}}\theta _{t}^{\varepsilon
}\right] \mathrm{d}t+\frac{\varepsilon }{R_{t}^{\varepsilon }}\Pi _{\theta
_{t}^{\perp }}\mathrm{d}W_{t}.%
\end{array}%
\right.  \label{fullsde}
\end{equation}%
They cannot be fully decoupled.

\begin{remark}
\label{rem:tangential-projection}
The extrinsic It\^{o} differential of a sphere-valued process is not in
general tangent. In the present notation the complete equation is
\[
 \mathrm{d}\theta_t^\varepsilon
 =\frac{1}{R_t^\varepsilon}\mathbb P_{\theta_t^\varepsilon}^\perp
                   b(X_t^\varepsilon)\,\mathrm{d}t
 +\frac{\varepsilon}{R_t^\varepsilon}
                   \mathbb P_{\theta_t^\varepsilon}^\perp \mathrm{d}W_t
 -\frac{\varepsilon^2(d-1)}{2(R_t^\varepsilon)^2}
                   \theta_t^\varepsilon\,\mathrm{d}t.
\]
Indeed its quadratic-variation trace is
$\varepsilon^2(d-1)(R_t^\varepsilon)^{-2}dt$. It\^{o}'s formula for
$|\theta_t^\varepsilon|^2=1$ therefore gives
\[
 \langle\theta_t^\varepsilon,\mathrm{d}\theta_t^\varepsilon\rangle
 =-\frac{\varepsilon^2(d-1)}{2(R_t^\varepsilon)^2}\,\mathrm{d}t.
\]
Projection removes this normal curvature term only from
$\mathbb P_{\theta_t^\varepsilon}^\perp d\theta_t^\varepsilon$;
that projected differential must not be substituted for the full
$d\theta_t^\varepsilon$. All stochastic angular equations below use
the extrinsic Ito convention and retain the curvature term.
\end{remark}

%Next, owing to the high complexity of Filippov solutions to the DI (\ref{2.2}%
%), we first examine the continuity properties of the system under the
%condition that $\langle b(x),x\rangle >0$, a condition that plays a crucial
%role in analyzing the radial behavior of trajectories, as previously noted.

Let $\mathbb{P}_{x}^{\varepsilon ,b}$ be the law of the diffusion process
associated with the differential generator

\begin{equation*}
L_{\varepsilon }=\frac{\varepsilon ^{2}}{2}\Delta +\langle b(x),\nabla
\rangle
\end{equation*}%
starting at $x\in\mathbb R^d$. In the main results we assume (A1),
so the drift is bounded and the solution is non-explosive. More
generally this notation on the full space
$\mathscr C(\mathbb R_+,\mathbb R^d)$ is used only after non-explosion
has been established; local boundedness alone is not that assertion.
The topology is uniform convergence on compact time intervals.

The following lemma can be found in \cite{Bi68}.

\begin{lemma}[Portmanteau]
\label{por} Let $(S,\rho )$ be a separable metric space, and let $\{\mu
_{n}\}_{n\geq 1},\mu $ be probability measures on $\mathcal{B}(S)$. 1) If $%
\mu _{n}\xrightarrow{w}\mu $, then for any closed subset $F\subset S$, $%
\varlimsup_{n\rightarrow \infty }\mu _{n}(F)\leq \mu (F).$ 2) If $\mu _{n}%
\xrightarrow{w}\mu $, then for any open subset $G\subset S$, $\mu \left(
G\right) \leq \varliminf_{n\rightarrow \infty }\mu _{n}(G).$
\end{lemma}

\begin{lemma}
\label{ml1} Assume (A1), and define
\[
 H(u)=\int_0^u\frac{\mathrm{d}s}{\varphi(s)},\qquad \psi=H^{-1}.
\]
Every radial zero-noise cluster law is concentrated on paths $R^0$ such
that $R^0(t)\ge\psi(t)>0$ for every $t>0$.
\end{lemma}

\begin{proof}
Since $0<h\le\varphi$, assumption \eqref{bi2} gives integrability of
$1/\varphi$ near zero. Lemma \ref{bc} and positivity give integrability
on every other compact interval. Moreover $\varphi\le M:=\sup_x|b(x)|$,
so $H(u)\ge u/M$ and $H(u)\to\infty$. Thus $H$ is continuous and
strictly increasing from $[0,\infty)$ onto $[0,\infty)$; its inverse
$\psi$ exists globally and is positive at every positive time.

Use the scalar equation, with the extension by zero on $(-\infty,0]$,
\begin{equation}
 \mathrm{d}r_t^\varepsilon=\varphi(r_t^\varepsilon)\,\mathrm{d}t+\varepsilon\,\mathrm{d}B_t,
 \qquad r_0^\varepsilon=0. \label{sder}
\end{equation}
Its drift is bounded, nonnegative on both sides of zero, and has
integrable reciprocal on the positive side. Theorem 1.1 of
\cite{fmp23} therefore gives convergence to $\psi$, uniformly in
probability on every finite time interval. This theorem allows a jump
at zero. Although the Brownian motion $B$ used in our radial coupling
may depend on $\varepsilon$, uniqueness in law of the scalar SDE shows
that the same conclusion holds in this coupling: convergence in law to
a deterministic path is convergence in probability to that path.
Lemma \ref{lem:radial-comparison-origin} gives
\begin{equation}
 R_t^\varepsilon\ge r_t^\varepsilon\quad(t\ge0),\quad\mathbb P\text{-a.s.}
 \label{comrr}
\end{equation}
For $T<\infty$ and $\eta>0$, the set
\[
 E_{\eta,T}=\{z\in C([0,T]):z(t)\ge\psi(t)-\eta
                         \text{ for every }t\in[0,T]\}
\]
is closed. Its probability under the law of $R^\varepsilon$ tends to
one by \eqref{comrr}. The closed-set form of the Portmanteau theorem
therefore gives probability one to $E_{\eta,T}$ under every cluster
law. Intersect over $\eta=1/n$ and, for locally uniform path laws, over
integer $T$. The resulting probability-one event is
$\{R^0(t)\ge\psi(t)>0\text{ for all }t>0\}$.
\end{proof}

%\begin{remark}
%\label{mb} In the one-dimensional setting with merely bounded measurable
%drift $b\in L^{\infty }((-r,r))$, the boundary-value elliptic ODE associated
%with the expected exit time $\frac{\varepsilon ^{2}}{2}a(x)\phi ^{\prime
%\prime }+b(x)\phi ^{\prime }=-1,\quad \phi (\pm r)=0$ admits no explicit
%classical $C^{2}$ smooth solution of the form $c_{1}+c_{2}A_{\varepsilon
%}(x)-B_{\varepsilon }(x)$ derived for continuous $b$ (see \cite{BB}). Only a
%unique weak solution exists in the Sobolev space $W_{0}^{1,2}((-r,r))$ via
%the Lax--Milgram theorem. The integral representations $A_{\varepsilon
%},B_{\varepsilon }$ lose their validity as classical solutions because b is
%only measurable: $A_{\varepsilon },B_{\varepsilon }$ are globally $C^{1}$
%but twice differentiable almost everywhere only, failing the pointwise ODE
%at the null discontinuity set of b. Consequently, explicit algebraic
%formulas for the constants $c_{1}(\varepsilon ),c_{2}(\varepsilon )$ (and
%hence the exit expectation $\mathbb{E}^{\varepsilon }[\tau
%]=c_{1}(\varepsilon )$ in the smooth case) no longer hold rigorously.
%Precise quantitative asymptotics, such as the explicit convex weight
%coefficients for the zero-noise limiting measure obtained by Bafico, Baldi (%
%\cite{BB}, 1982) under continuous drift, cannot be recovered directly for
%bounded measurable $b$.
%\end{remark}

Note that the condition (\ref{bi}) imposes a strong restriction on the drift
field $b(x)$. By definition, $\varphi (r):=\bar{b}(r)+\inf_{\omega \in
\mathbb{S}^{d-1}}T(r,\omega )$ represents the worst-case lower bound of the
radial drift. The requirement that the reciprocal of this quantity is
integrable near zero forces the radial drift to grow sufficiently fast as $%
r\rightarrow 0$, which may be too restrictive for many applications. For
instance, consider the case where $\langle b(x),x\rangle =0$ for all $x$.
While this satisfies the non-negativity condition, it forces $\varphi
(r)\equiv 0$, making the integral diverge. Even in cases where $\langle
b(x),x\rangle $ is non-negative but vanishes only on a set of measure zero,
the worst-case bound $\varphi (r)$ can still fail to be integrable due to
the possibility of arbitrarily small radial drift values along some
directions. For general case, we will study in our future work.

\subsubsection{A Counterexample}

Consider the following 2-dimensional drift field, which is non-negative in
the radial direction everywhere but fails the integrability condition (A1).
\begin{equation*}
b_{2}(x,y)=2%
\begin{pmatrix}
\mathrm{sgn}(x)\sqrt{|y|} \\
\mathrm{sgn}(y)\sqrt{|x|}%
\end{pmatrix}%
.
\end{equation*}%
A direct calculation shows that $\langle b_{2}(x,y),(x,y)\rangle =2(\mathrm{%
sgn}(x)x\sqrt{|y|}+\mathrm{sgn}(y)y\sqrt{|x|})=2(|x|\sqrt{|y|}+|y|\sqrt{|x|}%
)\geq 0$ for all $(x,y)\in \mathbb{R}^{2}$. However, along the coordinate
axes (e.g., the positive $x$-axis where $y=0$), the radial component $%
\langle b_2(x,0),(1,0)\rangle=0$. Thus $\varphi(r)=0$ for $r>0$,
so $b_2$ does not satisfy the positive radial condition in (A1).
Consider the perturbation
\[
 b_2^\delta=b_2+\delta(\operatorname{sgn}x,
                              \operatorname{sgn}y)^T,\qquad\delta>0.
\]
Its radial component is at least $\delta$, but this does not verify
all of (A1): $b_2^\delta$ is unbounded at infinity and discontinuous
along the coordinate axes. Theorem \ref{the1} is not directly invoked
for it. Instead, Example \ref{exa2d} proves by a localized Cartesian
and Tanaka argument that its path-law limit is the equal mixture of
four diagonal paths with absolute coordinates $q_\delta(t)$, where
$q_\delta'=2\sqrt{q_\delta}+\delta$, $q_\delta(0)=0$. In particular,
\begin{equation}
 R_t^0=\sqrt2q_\delta(t)\ge\sqrt2\delta t>0
              \quad(t>0),\quad\mathbb P\hbox{-a.s.} \label{er0}
\end{equation}
That example also proves the limit for $\delta=0$ and for
$\delta=\delta_\varepsilon\to0$. Pointwise convergence of the drift
alone is not the justification for these statements.

\begin{theorem}
\label{the1} Under (A1), consider
\begin{equation}
 \mathrm{d}X_t^\varepsilon=b(X_t^\varepsilon)\,\mathrm{d}t+\varepsilon\,\mathrm{d}W_t,
 \qquad X_0^\varepsilon=0. \label{sdethe}
\end{equation}
Every sequence $\varepsilon\downarrow0$ has a subsequence whose path
laws converge locally uniformly in distribution to a Filippov solution
$X^0$. Almost surely $R_t^0:=|X_t^0|\ge\psi(t)>0$ for all $t>0$, with
$\psi$ as in Lemma \ref{ml1}. Set
$\theta_t^0=X_t^0/R_t^0$ for $t>0$. Along that subsequence, for every
$T<\infty$, the polar processes converge in distribution in
\[
 C([0,T];[0,\infty))\ \times\ C_{\rm loc}((0,T];\mathbb S^{d-1}).
\]
On every $[\eta,T]$, $0<\eta<T$, the limiting polar paths are absolutely
continuous and, for almost every $t$,
\begin{equation}
 (\dot R_t^0,\dot\theta_t^0)\in
 \left\{\left(\langle v,\theta_t^0\rangle,
     \frac{1}{R_t^0}\mathbb P_{\theta_t^0}^{\perp}v\right):
                v\in F_b(R_t^0\theta_t^0)\right\}.
 \label{mr}
\end{equation}
The same vector $v$ is used in both components. No value or continuity
of $\theta^0$ at time zero is asserted, and uniqueness of the cluster
law is not asserted.
\end{theorem}

\begin{proof}
Proposition \ref{pro1}, followed by the usual diagonal extraction on
integer time horizons, gives the locally uniform path-law subsequence
and its Filippov identification. The norm map is continuous, so the
radial marginal of its limit is a radial cluster law. Lemma \ref{ml1}
gives the simultaneous lower bound by $\psi$.
On $[\eta,T]$ each limiting path has a positive minimum radius. The
maps $x\mapsto|x|$ and $x\mapsto x/|x|$ are continuous under uniform
convergence on such paths. The continuous mapping theorem, applied to
$\eta=1/n$ and finite $T$, proves the asserted polar convergence. These
maps are also continuously differentiable on every compact annulus.
The absolutely continuous chain rule applied to the very same velocity
$v(t)=\dot X^0(t)\in F_b(X^0(t))$ gives \eqref{mr} and absolute
continuity on $[\eta,T]$. This proves all the stated conclusions.
\end{proof}

Note that the support of the zero-noise limit is carried by immediate
departure solutions, but not every immediate departure solution is selected (see Remark \ref{rem:immediate-not-sufficient}).

\subsection{Limiting angular}

\label{3.3}

Theorem \ref{the1} gives instantaneous escape of the cluster paths
under (A1). We next derive angular identities for drifts smooth away
from the origin. General identities and monotonicity are distinguished
from a selection theorem: the latter is proved below only under the
explicit additional structure (A5). The two unbounded planar examples
are subsequently handled by independent arguments.

Let the unit sphere $\mathbb{S}^{d-1}\subset \mathbb{R}^{d}$. For a point $%
\omega \in \mathbb{S}^{d-1}$, the tangent space is defined as
\begin{equation*}
T_{\omega }\mathbb{S}^{d-1}=\big\{v\in \mathbb{R}^{d}\,\big|\,\langle
v,\omega \rangle =0\big\}.
\end{equation*}%
Consider a function $V:\mathbb{S}^{d-1}\rightarrow \mathbb{R}$. Take any
unit tangent vector $u\in T_{\omega }\mathbb{S}^{d-1}$ ($\Vert u\Vert =1$).
The directional derivative of $V$ along the direction $u$ is defined by
\begin{equation*}
D_{u}V(\omega ):=\lim_{t\rightarrow 0}\frac{V\big(\exp _{\omega }(tu)\big)%
-V(\omega )}{t},
\end{equation*}%
where $\exp _{\omega }$ denotes the spherical exponential map. In tangent
directions, this is equivalent to the Euclidean directional derivative.

%For a point $\omega \in \mathbb{S}^{d-1}$, the spherical gradient is defined
%as
%\begin{equation*}
%\nabla _{\mathbb{S}}V(\omega )=\nabla V(\omega )-\big\langle\nabla V(\omega
%),\omega \big\rangle\omega .
%\end{equation*}%
%That is, the spherical gradient is the orthogonal projection of the
%Euclidean gradient onto the tangent space $T_{\omega }\mathbb{S}^{d-1}$.

Extend $V$ smoothly to a neighborhood of $\mathbb{S}^{d-1}$ in $\mathbb{R}%
^{d}$, denoted by $\tilde{V}$. The spherical gradient is defined as the
orthogonal projection of the Euclidean gradient onto the tangent space
\begin{equation*}
\nabla _{\mathbb{S}}V(\omega )=\nabla \tilde{V}(\omega )-\big\langle\nabla
\tilde{V}(\omega ),\omega \big\rangle\omega .
\end{equation*}%
Clearly, $\nabla _{\mathbb{S}}V(\omega )\in T_{\omega }\mathbb{S}^{d-1}.$
For any tangent vector $u\in T_{\omega }\mathbb{S}^{d-1}$, we have $\langle
\omega ,u\rangle =0$.

By the asymptotic expansion of the exponential map,
\begin{equation*}
\exp_\omega(tu)=\omega+tu+o(t),\quad t\to 0.
\end{equation*}
The directional derivative satisfies
\begin{equation*}
D_u V(\omega) =\lim_{t\to0}\frac{\tilde V(\omega+tu+o(t))-\tilde V(\omega)}{t%
} =\big\langle\nabla\tilde V(\omega),u\big\rangle.
\end{equation*}
Taking inner product between spherical gradient and tangent vector $u$
yields
\begin{equation*}
\big\langle\nabla_{\mathbb{S}}V,u\big\rangle =\langle\nabla\tilde V,u\rangle-%
\big\langle\nabla\tilde V,\omega\big\rangle\langle\omega,u\rangle
=\langle\nabla\tilde V,u\rangle.
\end{equation*}
Consequently,
\begin{equation*}
D_u V(\omega)=\big\langle\nabla_{\mathbb{S}}V(\omega),u\big\rangle.
\end{equation*}
We next prove the gradient corresponds to the direction of fastest
increasing. By the Cauchy--Schwarz inequality
\begin{equation*}
\big|\langle \nabla _{\mathbb{S}}V,\,u\rangle \big|\leq \Vert \nabla _{%
\mathbb{S}}V\Vert \cdot \Vert u\Vert .
\end{equation*}%
Since $u$ is a unit tangent vector, $\Vert u\Vert =1$, hence
\begin{equation}
D_{u}V(\omega )\leq \big\|\nabla _{\mathbb{S}}V(\omega )\big\|.  \label{qm}
\end{equation}%
We can take the equality in (\ref{qm}) when $u=\frac{\nabla _{\mathbb{S}%
}V(\omega )}{\big\|\nabla _{\mathbb{S}}V(\omega )\big\|}.$ In other words,
the directional derivative attains its global maximum if and only if the
direction $u$ is parallel to the spherical gradient $\nabla _{\mathbb{S}}V$.

This proves that $\nabla _{\mathbb{S}}V$ is the direction of the fastest
rate of increase (steepest ascent) of $V$ on the sphere. If $\nabla _{%
\mathbb{S}}V(\omega )\neq 0$, the function strictly increases along this
direction. If $\nabla _{\mathbb{S}}V(\omega )=0$, the point is critical, and
all directional derivatives in tangent directions vanish.

We derive the relation between the spherical gradient of the angular
potential and the tangential component of the vector field $b$.

\begin{lemma}
\label{keyl} Let $b\in C^1(\mathbb R^d\setminus\{0\};\mathbb R^d)$ and
$V(r,\omega)=\langle b(r\omega),\omega\rangle$. Throughout, the notation
$\nabla_{\mathbb S}V(r,\omega)$ denotes the gradient for the standard
metric on the \emph{unit} sphere. Then
\begin{equation}
 \nabla_{\mathbb S}V(r,\omega)
 =r\mathbb P_\omega^\perp(Db(r\omega))^\top\omega
       +\mathbb P_\omega^\perp b(r\omega). \label{qc}
\end{equation}
If $\widetilde V(x)=V(|x|,x/|x|)$, then
$\mathbb P_\omega^\perp\nabla_x\widetilde V(r\omega)
=r^{-1}\nabla_{\mathbb S}V(r,\omega)$.
\end{lemma}

\begin{proof}
For $u\perp\omega$, differentiate along a smooth curve on the unit sphere
with initial velocity $u$. The chain rule gives
\[
 D_\omega V(r,\omega)[u]
 =r\langle Db(r\omega)u,\omega\rangle+\langle b(r\omega),u\rangle.
\]
The unique tangent vector representing this functional is the right-hand
side of \eqref{qc}. For the last assertion, the derivative of
$x\mapsto x/|x|$ at $r\omega$ is $r^{-1}\mathbb P_\omega^\perp$;
tangential differentiation leaves $r$ fixed and gives the factor $r^{-1}$.
The function $\widetilde V$ need not be homogeneous.

The Jacobian term measures part of the angular derivative; its vanishing
does not characterize radial symmetry. For example, with
$f(r)=\sqrt r/(1+\sqrt r)$ and a fixed unit vector $e$, the field
$b(r\omega)=f(r)(2\omega+e)$ is not radial, but
$\nabla_{\mathbb S}V=f(r)\mathbb P_\omega^\perp e
=\mathbb P_\omega^\perp b$; hence
$\mathbb P_\omega^\perp(Db)^\top\omega=0$. A tangent vector has no
unspecified scalar sign such as ``strictly positive''.
\end{proof}

\begin{proposition}
Assume (A1) and $b\in C^1(\mathbb R^d\setminus\{0\};\mathbb R^d)$.
For every limiting path in Theorem \ref{the1}, on each $[\eta,T]$ with
$\eta>0$, its angular component satisfies, almost everywhere,
\[
 \dot\theta_t^0
 =\frac{1}{R_t^0}\nabla_{\mathbb S}V(R_t^0,\theta_t^0)
 -\mathbb P_{\theta_t^0}^\perp(Db(R_t^0\theta_t^0))^\top\theta_t^0
 =\frac{1}{R_t^0}\mathbb P_{\theta_t^0}^\perp b(R_t^0\theta_t^0).
\]
\end{proposition}

\begin{proof}
At every nonzero point, continuity of $b$ gives $F_b(x)=\{b(x)\}$.
Use \eqref{mr} and then \eqref{qc}. The assertion concerns the law of
the limiting path; it does not assume almost sure convergence on the
original Brownian probability space.
\end{proof}

A pointwise necessary and sufficient condition for $\nabla _{\mathbb{S}%
}V(r,\omega )=0$ is established as follows.

\begin{lemma}
Fix $r>0$ and $\omega \in \mathbb{S}^{d-1}$. Let $b\in C^{\infty }\left(
\mathbb{R}^{d}\setminus \{0\}\right) .$ The spherical gradient of the radial
potential vanishes at $(r,\omega )$, i.e., $\nabla _{\mathbb{S}}V(r,\omega
)=0$, if and only if
\begin{equation*}
\mathbb{P}_{\omega }^{\perp }\left( (\nabla b(r\omega ))^{\top }\omega +%
\frac{1}{r}b(r\omega )\right) =0.
\end{equation*}%
Equivalently, the vector field combination $(\nabla b(r\omega ))^{\top
}\omega +r^{-1}b(r\omega )$ is purely radial and parallel to the unit
direction vector $\omega $, possessing no tangential component on the sphere.
\end{lemma}

\begin{remark}
For a radial field $b(r\omega)=g(r)\omega$, the potential $V(r,\cdot)$
is constant. Non-radiality alone implies neither isolated critical
points nor correspondence between potential critical points and
angular equilibria. In dimension two, let $J$ be rotation through
$\pi/2$ and let $b(r\omega)=f(r)(\omega+J\omega)$ with
$f(r)=\sqrt r/(1+\sqrt r)$ and $b(0)=0$. This bounded field satisfies
(A1), but $V(r,\omega)=f(r)$ and
$\dot\omega=f(r)J\omega/r\ne0$ at every direction. All directions
are potential critical points. The Morse hypothesis (A2) below gives
isolated critical points (including possible saddles when $d\ge3$).
The alignment hypothesis (A4), not non-radiality, will make vanishing
of the potential gradient equivalent to vanishing angular velocity.
\end{remark}

We impose the following conditions:

\begin{enumerate}
\item[\textbf{(A2)}] Let the vector field $b:\mathbb{R}^{d}\rightarrow
\mathbb{R}^{d}$ satisfy $b\in C\left( \mathbb{R}^{d}\right) \cap C^{\infty
}\left( \mathbb{R}^{d}\setminus \{0\}\right) $, and define the radial
potential for $r>0$ and $\omega \in \mathbb{S}^{d-1}$ as $V(r,\omega
)=\langle b(r\omega ),\omega \rangle $. Suppose the following hold for each
fixed $r>0$:

\begin{enumerate}
\item[(i)] $V(r,\cdot )$ is nonconstant on $\mathbb{S}^{d-1}$;

\item[(ii)] Every critical point $\omega^{*}\in\mathbb{S}^{d-1}$ of $%
V(r,\cdot)$ is non-degenerate, i.e., $\mathrm{Hess}_{\mathbb{S}%
^{d-1}}V(r,\omega^{*})$ has no zero eigenvalue on $T_{\omega^{*}}\mathbb{S}%
^{d-1}$.

\item[(iii)] Every local maximum is global maximum.
\end{enumerate}

\item[\textbf{(A3)}] Let the vector field $b:\mathbb{R}^{d}\rightarrow
\mathbb{R}^{d}$ satisfy $b\in C\left( \mathbb{R}^{d}\right) \cap C^{\infty
}\left( \mathbb{R}^{d}\setminus \{0\}\right) .$ Assume that $\langle \left(
\nabla b(r\omega )\right) \omega ,\,\omega \rangle \geq 0,\forall r>0,\
\forall \omega \in \mathbb{S}^{d-1}.$
\end{enumerate}

\begin{remark}
For a fixed $r$, critical points of $V(r,\cdot )$ may be local maxima, local
minima, or saddle points when $\dim \mathbb{S}^{d-1}\geq 2$. Under a Morse
nondegeneracy hypothesis these types are determined by the signature of the
spherical Hessian.
\end{remark}

\begin{remark}
Differentiate the radial potential $V(r,\omega )=\langle b(r\omega ),\omega
\rangle $ with respect to $r$ via the chain rule:
\begin{equation*}
\frac{\partial V}{\partial r}(r,\omega )=\left\langle \nabla b(r\omega )%
\frac{\partial (r\omega )}{\partial r},\omega \right\rangle =\big\langle%
\nabla b(r\omega )\,\omega ,\omega \big\rangle,
\end{equation*}%
which denotes the Jacobi's quadratic form in the radial direction $\omega $.
This quadratic form characterizes the monotonicity of the radial velocity $%
\dot{r}=V(r,\omega )$ with respect to the radius $r$. Assumption (A3) is
equivalent to the monotonicity condition of radial escape speed:
\begin{equation*}
\frac{\partial V}{\partial r}(r,\omega )\geq 0,\quad \forall r>0,\ \forall
\omega \in \mathbb{S}^{d-1}.
\end{equation*}
\end{remark}

Let $M$ be a closed smooth manifold of dimension $d$, and $f\colon M\to%
\mathbb{R}$ a Morse function. Denote by $M_k$ the number of critical points
of $f$ with Morse index $k$, for $0\le k\le d$. Then
\begin{equation*}
\sum_{k=0}^{d}(-1)^{k} M_{k}=\chi(M),
\end{equation*}
where $\chi(M)$ denotes the Euler characteristic of $M$. This identity
follows as an immediate consequence of the strong Morse inequalities; see
Milnor~\cite{M63}, Theorem~5.2. The sphere $\mathbb{S}^{d-1}$ is a closed
smooth manifold of dimension $n=d-1$, whose Euler characteristic admits an
explicit classification formula:
\begin{equation*}
\chi(\mathbb{S}^{d-1})=
\begin{cases}
2, & \text{if }d-1\text{ is even}, \\
0, & \text{if }d-1\text{ is odd}.%
\end{cases}%
\end{equation*}
For any Morse function defined on $\mathbb{S}^{d-1}$, the alternating sum of
the number of critical points of each Morse index equals the Euler
characteristic (see Milnor~\cite{M63}, Theorem 5.2):
\begin{equation*}
\sum_{k=0}^{d-1} (-1)^k M_k = \chi(\mathbb{S}^{d-1}),
\end{equation*}
where $M_k$ denotes the number of critical points of Morse index $k$.
Under (A2), maxima are strict local maxima and also global maxima;
different global maximizers may have the same value. Minima are
strict local minima, and the other critical points have intermediate
Morse index $1\le k\le d-2$. The Euler characteristic constrains the
numbers of these points; it does not imply uniqueness of a global
maximizer or strict comparison with every other point of the sphere.

\begin{lemma}
\label{key2} Assume (A2) and fix $r>0$. The critical points of
$V(r,\cdot)$ are finite in number. A critical point with negative
definite spherical Hessian is a strict local maximum and, by
(A2)(iii), a global maximum. It need not be the unique global maximum.
A positive definite Hessian gives a strict local minimum. An
indefinite Hessian gives a saddle point. In particular, the set of
global maximizers is finite.
\end{lemma}

\begin{proof}
In normal coordinates centred at a critical point $m$, Taylor's formula
reads
\[
 V(r,\exp_m u)=V(r,m)+\tfrac12\langle H_m u,u\rangle+o(|u|^2).
\]
Definiteness gives the stated strict local extrema. If $H_m$ is
indefinite, choose vectors on which the quadratic form has opposite
signs; sufficiently small nonzero multiples give values on both sides
of $V(r,m)$, proving the saddle classification. Non-degeneracy implies
that these cases exhaust the critical points. The inverse function
theorem applied to the gradient in a chart makes every critical point
isolated. The critical set is closed in the compact sphere. An infinite
critical set would have an accumulation point which is itself critical,
contradicting isolation. Thus it is finite. A global maximum is a
critical point, and (A2)(iii) identifies the strict local maxima among
these points as global maxima. Equal values at distinct global maxima
are allowed.
\end{proof}

%\begin{remark}
%Under the hypotheses of Lemma \ref{key2}, several important topological
%consequences follow from Morse inequalities on the compact sphere $\mathbb{S}%
%^{d-1}$. If $V(r,\cdot)$ possesses two or more distinct strict
%global-maximizer points, then at least one saddle-point critical point must
%exist on $\mathbb{S}^{d-1}$. The saddle points and their stable manifolds
%form the separating boundaries between the basins of attraction of different
%global-maximizer directions for the spherical gradient flow
%\begin{equation*}
%\dot{\omega}(s)=\nabla_{\mathbb{S}}V(r,\omega(s)).
%\end{equation*}
%This multi-basin structure permits the gradient flow to converge to any one
%of the global-maximizer directions, which provides the geometric mechanism
%for multi-branch zero-noise selection phenomena of the stochastic
%differential equation $dX_t^\varepsilon =
%b(X_t^\varepsilon)\,dt+\varepsilon\,dW_t$. We also emphasize that no such
%saddle-point constraint appears in the one-dimensional case $d=2$, where $%
%\mathbb{S}^{1}$ is a circle and saddle points cannot exist.
%\end{remark}

For $r>0$ and $\omega\in\mathbb S^{d-1}$, define
\begin{align}
 A(r,\omega)&=\mathbb P_\omega^\perp(Db(r\omega))^\top\omega,
                                                       \label{Ab}\\
 B(r,\omega)&=r^{-1}\mathbb P_\omega^\perp b(r\omega),    \label{Bb}\\
 G(r,\omega)&=r^{-1}\nabla_{\mathbb S}V(r,\omega).        \label{Gb}
\end{align}
Thus $G$ is a scaled gradient, whereas $\nabla_{\mathbb S}V$ always
refers to the unit-sphere metric. Lemma \ref{keyl} gives
\begin{equation}
 G=A+B. \label{ABG}
\end{equation}

%For $r>0$ and $\omega\in\mathbb{S}^{d.1}$, define tangent vectors
%\begin{align}
%A(r,\omega) &:=\mathbb{P}_{\omega}^{\perp} \bigl(Db(r\omega)\bigr)%
%^{\top}\omega,  \label{Ab} \\
%B(r,\omega) &:=\frac1r\mathbb{P}_{\omega}^{\perp}b(r\omega),  \label{Bb} \\
%G(r,\omega) &:=\frac1r\nabla_{\mathbb{S}}V(r,\omega).  \label{Gb}
%\end{align}
%By Lemma \ref{keyl},
%\begin{equation}
%G(r,\omega)=A(r,\omega)+B(r,\omega).  \label{ABG}
%\end{equation}
%Thus $G$ is a scaled spherical gradient; the standard spherical gradient is $%
%rG$.

\begin{enumerate}
\item[\textbf{(A4)}] On the radial-angular region under consideration there
exists a Borel function $\gamma (r,\omega )$ and a constant $c_{\gamma }>0$
such that
\begin{equation}
A(r,\omega )=\gamma (r,\omega )B(r,\omega ),\qquad 1+\gamma (r,\omega )\geq
c_{\gamma }>0.  \label{rp}
\end{equation}%
If classical ODE uniqueness is invoked, $\gamma $ and the resulting vector
field are additionally assumed locally Lipschitz in $\omega $.
\end{enumerate}

%
%\begin{enumerate}
%\item[\textbf{(A4)}] Suppose that for every $r>0$ and every $\omega \in
%\mathbb{S}^{d-1}$, there exists a scalar $\gamma =\gamma (r,\omega )>-1$
%such that
%\begin{equation}
%A(r,\omega )=\gamma (r,\omega )\,B(r,\omega ).  \label{rp}
%\end{equation}
%\end{enumerate}

Define the set of spherical critical points %\begin{equation*}
%\mathcal{C}\triangleq \big\{\omega \in S^{d-1}\mid \nabla _{\mathbb{S}%
%}V(r,\omega )=0\big\}.
%\end{equation*}%
\begin{equation*}
\mathcal{C}(r):=\{\omega \in \mathbb{S}^{d-1}:G(r,\omega )=0\}.
\end{equation*}

\begin{proposition}
\label{prop:polar.chain} Let $x\in AC([a,b];\mathbb{R}^{d})$ and assume $%
\left\vert x(t)\right\vert >0$ on $[a,b]$. Set $r(t)=\left\vert
x(t)\right\vert $ and $\omega (t)=x(t)/\left\vert x(t)\right\vert $. Then $%
r\in AC([a,b])$, $\omega \in AC([a,b];\mathbb{S}^{d-1})$, and for almost
every $t$,
\begin{equation*}
\dot{r}(t)=\langle \dot{x}(t),\omega (t)\rangle ,\qquad \dot{\omega}(t)=%
\frac{1}{r(t)}\mathbb{P}_{\omega (t)}^{\perp }\dot{x}(t).
\end{equation*}%
If in addition $\dot{x}(t)=b(x(t))$ almost everywhere, then
\begin{equation*}
\dot{r}(t)=V(r(t),\omega (t)),\qquad \dot{\omega}(t)=\frac{1}{r(t)}\mathbb{P}%
_{\omega (t)}^{\perp }b(r(t)\omega (t)).
\end{equation*}%
The assertion is made only on intervals on which the path stays away from
the origin; no initial angular value at $r(0)=0$ is required.
\end{proposition}

\begin{proof}
Apply the chain rule for absolutely continuous paths to the smooth maps $%
x\mapsto \left\vert x(t)\right\vert $ and $x\mapsto x/\left\vert
x(t)\right\vert $ on $\mathbb{R}^{d}\setminus \{0\}$.
\end{proof}

\begin{proposition}
\label{pro2} Let $x(t)=r(t)\omega(t)$ solve $\dot x=b(x)$ on an interval
with $r(t)>0$. Assume $b\in C^1(\mathbb R^d\setminus\{0\})$ and (A4)
along this trajectory. Then
\begin{equation}
 \dot\omega=B=\frac{G}{1+\gamma}. \label{om2}
\end{equation}
Moreover, with all coefficients evaluated at $(r(t),\omega(t))$,
\begin{equation}
 \frac{\mathrm{d}}{\mathrm{d}t}V(r(t),\omega(t))
 =\partial_rV\,V+\frac{r}{1+\gamma}|G|^2
 =\partial_rV\,V+\frac{|\nabla_{\mathbb S}V|^2}{r(1+\gamma)}.
 \label{eq:V.monotone}
\end{equation}
This derivative is nonnegative when $V\ge0$ and $\partial_rV\ge0$.
\end{proposition}

\begin{proof}
Proposition \ref{prop:polar.chain} gives $\dot r=V$ and $\dot\omega=B$.
Equations \eqref{ABG} and \eqref{rp} give $G=(1+\gamma)B$.
The chain rule uses the unit-sphere gradient $\nabla_{\mathbb S}V=rG$:
\[
 \frac{\mathrm{d}}{\mathrm{d}t}V=\partial_rV\,\dot r+
                    \langle\nabla_{\mathbb S}V,\dot\omega\rangle
 =\partial_rV\,V+\left\langle rG,\frac{G}{1+\gamma}\right\rangle.
\]
The positive denominator in (A4) proves the final assertion.
\end{proof}

We distinguish a Hessian lower bound from regularity upper bounds.
\begin{itemize}
\item[\textbf{(A2$^{\prime}$)}] For every $r_0>0$ there is
$\kappa_{r_0}>0$ such that, at every critical point $m$ of $V(r,\cdot)$
with $r\ge r_0$, the unit-sphere Hessian satisfies
\[
 |\operatorname{Hess}_{\mathbb S}V(r,m)u|\ge\kappa_{r_0}|u|
 \quad(u\in T_m\mathbb S^{d-1}).
\]
\end{itemize}
This condition is not equivalent to a uniform $C^1$ upper bound for
the gradient. For example, on the circle, with
$f(r)=\sqrt r/(1+\sqrt r)$, the profile
$f(r)[3-e^{-r}(1+\cos\theta)]$ has uniformly bounded first and second
angular derivatives, but its critical Hessians have magnitude
$f(r)e^{-r}\to0$. Conversely, finite-radius critical non-degeneracy
does not exclude new critical points at infinity. The profile
\[
 V(r,\theta)=f(r)[3-\tfrac13\cos^3\theta-e^{-r}(1+\cos\theta)]
\]
has derivative $f(r)\sin\theta(\cos^2\theta+e^{-r})$. Its only
finite-radius critical points are $0,\pi$, with Hessians
$\pm f(r)(1+e^{-r})$, uniformly separated from zero on $r\ge r_0$.
Nevertheless its limiting profile has an additional degenerate
critical point at $\pi/2$, since its derivative there tends to zero.
When the maximizer set is independent of $r$, denote it by
$\mathcal M=\operatorname*{arg\,max}_{\omega\in\mathbb S^{d-1}}V(r,\omega)$.
No general angular-selection conclusion from (A1)--(A4) and
(A2$^{\prime}$) is claimed below; the selection results will use the
explicit additional structure (A5) below.

For the original diffusion, under (A4) and at positive radii, the
\emph{physical-time} equations are
\begin{align}
 \mathrm{d}R_t^\varepsilon&=\left[V(R_t^\varepsilon,\theta_t^\varepsilon)
       +\frac{\varepsilon^2(d-1)}{2R_t^\varepsilon}\right]\mathrm{d}t
                         +\varepsilon\,\mathrm{d}B_t, \label{eq:rad-t}\\
 \mathrm{d}\theta_t^\varepsilon&=\frac{G(R_t^\varepsilon,\theta_t^\varepsilon)}
                  {1+\gamma(R_t^\varepsilon,\theta_t^\varepsilon)}\mathrm{d}t
       +\frac{\varepsilon}{R_t^\varepsilon}
                      \mathbb P_{\theta_t^\varepsilon}^\perp \mathrm{d}W_t
       -\frac{\varepsilon^2(d-1)}{2(R_t^\varepsilon)^2}
                         \theta_t^\varepsilon\,\mathrm{d}t. \label{eq:ang-t}
\end{align}
Start the angular clock at a positive time $t_*$ (or at a stopping
time with positive radius), and put
$s(t)=\int_{t_*}^t(R_u^\varepsilon)^{-2}du$.
On its range, with $r_s=R_{t(s)}^\varepsilon$ and
$\vartheta_s=\theta_{t(s)}^\varepsilon$, the legitimate time change gives
\begin{align}
 \mathrm{d}\vartheta_s&=\frac{r_s^2G(r_s,\vartheta_s)}{1+\gamma(r_s,\vartheta_s)}\mathrm{d}s
      +\varepsilon\mathbb P_{\vartheta_s}^\perp \mathrm{d}\widehat W_s
      -\tfrac12\varepsilon^2(d-1)\vartheta_s\,\mathrm{d}s, \label{eq:ang-s}\\
 \mathrm{d}r_s&=\left[r_s^2V(r_s,\vartheta_s)
                   +\tfrac12\varepsilon^2(d-1)r_s\right]\mathrm{d}s
                   +\varepsilon r_s\,\mathrm{d}\widehat B_s. \label{eq:rad-s}
\end{align}
Here $\mathrm{d}\widehat B_s=\langle\vartheta_s,d\widehat W_s\rangle$.
The covariance computation for the time-changed stochastic integrals
is the identity $dt=r_s^2ds$; in particular $\widehat W$ is Brownian
on the clock interval, with an independent extension if necessary.
The clock is not started at the singular origin.
The state $R_t^\varepsilon$ itself is not a time parameter:
$[R^\varepsilon]_t-[R^\varepsilon]_a=\varepsilon^2(t-a)>0$ on any
positive-time interval, whereas continuous monotone paths have zero
quadratic variation. The Dambis--Dubins--Schwarz theorem does not
replace $\mathrm{d}t$ by $\mathrm{d}r/V$ in the stochastic equation.

\begin{lemma}
\label{lem:reparam} Assume (A1), (A3), and (A4). An ODE trajectory
starting at positive radius $r_0$ has increasing radius tending to
infinity and can be parameterized by $r$. Its angular component solves
\begin{equation}
 \frac{\mathrm{d}\omega}{\mathrm{d}r}=F(r,\omega):=
 \frac{G(r,\omega)}{(1+\gamma(r,\omega))V(r,\omega)}
 =\frac{\nabla_{\mathbb S}V(r,\omega)}
             {r(1+\gamma(r,\omega))V(r,\omega)}. \label{eq:ode-r}
\end{equation}
Its energy satisfies
\begin{equation}
 \frac{\mathrm{d}}{\mathrm{d}r}V(r,\omega(r))=\partial_rV+
 \frac{|\nabla_{\mathbb S}V|^2}{r(1+\gamma)V}\ge0. \label{eq:energy-r}
\end{equation}
In particular,
\begin{equation}
 V(r,\omega)\ge\varphi(r_0)>0\quad(r\ge r_0), \label{eq:V-lower}
\end{equation}
and, writing $M=\sup_x|b(x)|$,
\begin{equation}
 \int_{r_0}^\infty
 \frac{|\nabla_{\mathbb S}V(r,\omega(r))|^2}
 {r(1+\gamma(r,\omega(r)))V(r,\omega(r))}\,\mathrm{d}r
 \le M-V(r_0,\omega(r_0)). \label{eq:integ}
\end{equation}
\end{lemma}

\begin{proof}
By (A3), $V(r,\omega)$ and hence $\varphi(r)$ are nondecreasing in $r$.
Along an outward trajectory, $\dot r=V\ge\varphi(r_0)>0$ and
$\dot r\le M$, proving global continuation and $r\to\infty$.
Divide \eqref{om2} by $\dot r=V$ to get \eqref{eq:ode-r}.
Apply the unit-sphere chain rule to obtain \eqref{eq:energy-r}.
Integrating it and using $\partial_rV\ge0$ and $V\le M$ gives
\eqref{eq:integ}. No unweighted gradient-separation assertion follows
from this weighted bound.
\end{proof}

\begin{itemize}
\item[\textbf{(A5)}] \emph{Separable gradient structure with linear
unstable modes.} Let $f\in C([0,\infty))\cap C^\infty((0,\infty))$
be bounded and nondecreasing, with $f(0)=0$, $f(r)>0$ for $r>0$, and
$\int_0^1f(r)^{-1}dr<\infty$. Let $\Phi$ be a positive smooth Morse
function on $\mathbb S^{d-1}$ whose local maxima are all global, and
let $a_\star:\mathbb S^{d-1}\to(0,\infty)$ be smooth. Assume
\begin{equation}
 b(r\omega)=f(r)\{\Phi(\omega)\omega+
                    a_\star(\omega)\nabla_{\mathbb S}\Phi(\omega)\},
 \qquad b(0)=0. \label{eq:separable-class}
\end{equation}
For each critical point $m$ of $\Phi$ which is not a maximum, require
a unit vector $u_m\perp m$, a neighbourhood $U_m$ of $m$ on the sphere,
and a smooth scalar function $K_m$ on $U_m$ such that
\begin{equation}
 \left\langle u_m,\Phi(\omega)\omega+
                 a_\star(\omega)\nabla_{\mathbb S}\Phi(\omega)\right\rangle
 =K_m(\omega)\langle u_m,\omega\rangle\quad(\omega\in U_m),
 \qquad K_m(m)>\Phi(m). \label{eq:linear-unstable-mode}
\end{equation}
\end{itemize}
The last requirement is a local algebraic identity for the drift and a
strict inequality at a specified critical point. It does not assume
angular convergence, a stochastic exclusion property, or the existence
of an entry distribution. It is stronger than the existence of a
positive Hessian eigenvalue alone.

These hypotheses imply (A1)--(A4) and (A2$^{\prime}$). Indeed
$V=f\Phi$, $\varphi=f\min\Phi$, $h=\varphi$ in (A1), and
\[
 \nabla_{\mathbb S}V=f\nabla_{\mathbb S}\Phi,\quad
 B=\frac{fa_\star}{r}\nabla_{\mathbb S}\Phi,\quad
 A=\frac{f(1-a_\star)}r\nabla_{\mathbb S}\Phi,
 \quad \gamma=a_\star^{-1}-1.
\]
Boundedness, continuity at zero, and smoothness away from zero follow
from \eqref{eq:separable-class}. Positivity and radial monotonicity
follow from $\Phi>0$ and $f'\ge0$. At a critical point the spherical
Hessian is $f(r)\operatorname{Hess}_{\mathbb S}\Phi$. There are finitely
many such points, all non-degenerate, and $f(r)\ge f(r_0)>0$ for
$r\ge r_0$. This proves the uniform Hessian lower bound. Also
$1+\gamma\ge1/\max a_\star>0$, and
$\mathcal M=\operatorname*{arg\,max}\Phi$ is independent of $r$.
\begin{example}
\label{A5}
Here are two explicit families satisfying the entire hypothesis group.
In both, one may take $f(r)=r^\alpha/(1+r^\alpha)$, $0<\alpha<1$.

\begin{itemize}
\item First, let $\Phi(\omega)=c+a\langle e,\omega\rangle$ with $c>a>0$
and $a_\star=\lambda/a$, $\lambda>0$. The only critical points are
$e$ and $-e$, with Hessians $-aI$ and $aI$. For $m=-e$, choose any
unit $u_m\perp e$ and
$K_m(\omega)=c+(a-\lambda)\langle e,\omega\rangle$.
Then \eqref{eq:linear-unstable-mode} holds and
$K_m(-e)=c-a+\lambda>\Phi(-e)$; the maximum set is $\{e\}$.

\item Second, in any orthonormal coordinates take an integer $\ell\ge2$,
$c\ge0$, $a>0$, $\lambda>0$, and
\[
 Q(\omega)=\sum_{i=1}^d\omega_i^{2\ell},\qquad
 \Phi=c+aQ,\qquad a_\star=\lambda/a.
\]
The critical points have a nonempty support $S$ of size $k$ and
$|\omega_i|=1/\sqrt k$ for $i\in S$. This follows from the Lagrange
multiplier equations
$\omega_i(\omega_i^{2\ell-2}-Q)=0$ and $|\omega|=1$.
Put $q_k=k^{1-\ell}$. On the tangent subspace supported in $S$ the
Hessian of $Q$ has eigenvalue $4\ell(\ell-1)q_k$ (multiplicity $k-1$),
and on its complementary tangent subspace it has eigenvalue
$-2\ell q_k$ (multiplicity $d-k$). These formulas follow by subtracting
$2\ell Q I$ from the ambient Hessian restricted to the tangent space.
Thus all critical points are non-degenerate and precisely the $2d$
points $\{\pm e_i\}$ are local, hence global, maxima. Positivity holds
because $Q\ge d^{1-\ell}>0$.
For a nonmaximum critical point choose distinct $i,j\in S$, set
$s=\operatorname{sgn}(m_i m_j)$, and $u_m=(e_i-se_j)/\sqrt2$.
With
\[
 D_n(x,y)=\sum_{j=0}^{n-1}x^{n-1-j}y^j,
 \qquad
 K_m(\omega)=c+(a-2\ell\lambda)Q(\omega)
            +2\ell\lambda D_{2\ell-1}(\omega_i,s\omega_j),
\]
the polynomial identity $x^n-y^n=(x-y)D_n(x,y)$ proves
\eqref{eq:linear-unstable-mode} on the whole sphere. Moreover
\[
 K_m(m)=c+\{a+4\ell\lambda(\ell-1)\}q_k
                     >c+aq_k=\Phi(m).
\]
This family includes the bounded version of the cubic angular field:
$\ell=2$, $c=0$, $a=1$, $\lambda=1/4$ gives
$b(r\omega)=f(r)(\omega_1^3,\ldots,\omega_d^3)$.
\end{itemize}

\end{example}

\begin{lemma}
Under \emph{(A5)}, let $m$ be a non-maximum critical point and let
$u_m,U_m,K_m$ be as in \eqref{eq:linear-unstable-mode}. Put
\[
\lambda_m:=\frac{K_m(m)-\Phi(m)}{\Phi(m)}>0,
\qquad \delta(r):=\langle u_m,\theta(r)\rangle .
\]
For every $\eta\in(0,\lambda_m)$ there is a neighbourhood
$U_{m,\eta}\subset U_m$ such that, on any deterministic angular orbit
segment with $0<r_1\le r_2$ and
$\theta([r_1,r_2])\subset U_{m,\eta}$, either $\delta$ vanishes
throughout the segment or it has constant nonzero sign and
\[
\left(\frac{r_2}{r_1}\right)^{\lambda_m-\eta}
\le \left|\frac{\delta(r_2)}{\delta(r_1)}\right|
\le \left(\frac{r_2}{r_1}\right)^{\lambda_m+\eta}.
\]
If, in addition, the orbit is defined for $0<r\le r_0$ and
$\theta(r)\to m$ as $r\downarrow0$, then, for all sufficiently small
$r$, either $\delta(r)\equiv0$ or there are constants $C\ne0$ and
$\nu>0$ such that
\begin{equation}
\delta(r)=C r^{\lambda_m}\bigl(1+O(r^\nu)\bigr),
\qquad r\downarrow0. \label{eq:powerlaw}
\end{equation}
The constant $C$ depends on the orbit and is not prescribed by its
value at one fixed radius through the linearized equation alone.
\end{lemma}

\begin{proof}
Write $V(r,\omega)=\langle b(r\omega),\omega\rangle$. Since $\nabla_{\mathbb{%
S}}\Phi\perp\omega$,
\begin{equation}
V(r,\omega)=f(r)\,\Phi(\omega).  \label{eq:Vsep}
\end{equation}

\noindent Step 1. Deterministic radial and angular ODEs. The tangential part
of the drift is
\begin{equation}
\mathbb{P}_{\omega }^{\perp }b(r\omega )=f(r)\,a_{\star }(\omega )\,\nabla _{%
\mathbb{S}}\Phi (\omega ),  \label{eq:tang}
\end{equation}%
because $\mathbb{P}_{\omega }^{\perp }(\Phi (\omega )\omega )=0$. On
positive radii the deterministic ($\varepsilon =0$) equations are
\begin{equation}
\frac{\mathrm{d}\theta _{t}}{\mathrm{d}t}=\frac{1}{R_{t}}f(R_{t})\,a_{\star
}(\theta _{t})\,\nabla _{\mathbb{S}}\Phi (\theta _{t}),\qquad \frac{\mathrm{d%
}R_{t}}{\mathrm{d}t}=f(R_{t})\,\Phi (\theta _{t}).  \label{eq:odes}
\end{equation}

\noindent Step 2. Divide and remove $f(r)$; set $\tau =\log r$. Using $r$ as
the independent variable,
\begin{equation}
\frac{\mathrm{d}\theta }{\mathrm{d}r}=\frac{\mathrm{d}\theta /\mathrm{d}t}{%
\mathrm{d}r/\mathrm{d}t}=\frac{a_{\star }(\theta )\,\nabla _{\mathbb{S}}\Phi
(\theta )}{r\,\Phi (\theta )}.  \label{eq:dthdr}
\end{equation}%
Introduce the logarithmic clock $\tau :=\log r, d\tau =\frac{dr}{r}.$
Since $\mathrm{d}\theta /\mathrm{d}\tau =r\,\mathrm{d}\theta /\mathrm{d}r$, %
\eqref{eq:dthdr} becomes
\begin{equation}
\frac{\mathrm{d}\theta }{\mathrm{d}\tau }=\frac{a_{\star }(\theta )\,\nabla
_{\mathbb{S}}\Phi (\theta )}{\Phi (\theta )}.  \label{eq:auton}
\end{equation}%
The radial envelope $f(r)$ cancels: the angular flow is autonomous on the $%
\tau $-clock.

\noindent Step 3. Project onto the unstable axis $u_{m}$. Define
\begin{equation}
\delta (\tau ):=\langle u_{m},\theta (\tau )\rangle ,
\end{equation}%
which vanishes at $\theta =m$. Projecting \eqref{eq:auton},
\begin{equation}
\frac{\mathrm{d}\delta }{\mathrm{d}\tau }=\left\langle u_{m},\frac{a_{\star
}\nabla _{\mathbb{S}}\Phi }{\Phi }\right\rangle =\frac{1}{\Phi (\theta )}%
\,\langle u_{m},a_{\star }\nabla _{\mathbb{S}}\Phi \rangle .  \label{eq:proj}
\end{equation}

\noindent Step 4. Use the algebraic identity \eqref{eq:linear-unstable-mode}. Split %
\eqref{eq:linear-unstable-mode} as
\begin{equation}
\langle u_{m},a_{\star }\nabla _{\mathbb{S}}\Phi \rangle =K_{m}(\omega
)\,\langle u_{m},\omega \rangle -\Phi (\omega )\,\langle u_{m},\omega
\rangle =\big(K_{m}(\omega )-\Phi (\omega )\big)\,\delta .  \label{eq:split}
\end{equation}%
Substituting into \eqref{eq:proj} gives the exact equation
\begin{equation}
\frac{d\delta }{d\tau }=\frac{K_{m}(\theta )-\Phi (\theta )}{\Phi (\theta )}%
\,\delta .  \label{eq:exact}
\end{equation}

\noindent Step 5. Control the full angular displacement. Set
\[
\kappa_m(\omega):=\frac{K_m(\omega)-\Phi(\omega)}{\Phi(\omega)}.
\]
This is smooth near $m$ and $\kappa_m(m)=\lambda_m$. Consequently
\begin{equation}
\frac{\mathrm{d}\delta}{\mathrm{d}\tau}
=\lambda_m\delta+O\bigl(|\theta-m|\,|\delta|\bigr).
\label{eq:lin}
\end{equation}
No bound of $|\theta-m|$ by $|\delta|$ is used. On every connected
interval on which $\theta\in U_m$, the exact equation
\eqref{eq:exact} gives
\[
\delta(r_2)=\delta(r_1)
\exp\left\{\int_{r_1}^{r_2}
\kappa_m(\theta(\rho))\frac{\mathrm{d}\rho}{\rho}\right\}.
\]
This proves the zero/nonzero alternative and constant sign on that
interval. Shrink $U_m$ so that $|\kappa_m-\lambda_m|\le\eta$;
integrating this bound proves both local power inequalities in the
statement.
\noindent Step 6. An integrable correction along a backward orbit.
Suppose now that $\theta(e^\tau)\to m$ as $\tau\to-\infty$, and
write $\theta(\tau)=\theta(e^\tau)$ in this step. Choose $\tau_0$
so that the whole tail $(-\infty,\tau_0]$ lies in a sufficiently
small neighbourhood of $m$. Put $g=a_\star/\Phi$. In this
neighbourhood $g\ge g_0>0$. Non-degeneracy of the spherical Hessian
and Taylor's formula in normal coordinates give constants $c,C_0>0$
such that
\[
|\nabla_{\mathbb S}\Phi(\omega)|\ge c|\omega-m|,
\qquad |\Phi(\omega)-\Phi(m)|\le C_0|\omega-m|^2.
\]
Indeed, the coordinate gradient has an invertible linear part at
$m$, and the chart metric and chord distance are uniformly
comparable in a small neighbourhood. Thus, with $k=c/\sqrt{C_0}$,
\[
|\nabla_{\mathbb S}\Phi(\omega)|
\ge k\sqrt{|\Phi(\omega)-\Phi(m)|}.
\]
Let $E(\tau)=\Phi(\theta(\tau))-\Phi(m)$. Equation
\eqref{eq:auton} gives
\[
E'=g(\theta)|\nabla_{\mathbb S}\Phi(\theta)|^2,
\qquad \lim_{\tau\to-\infty}E(\tau)=0.
\]
Hence $E\ge0$. If $E$ vanishes at any finite time, its
nonnegative derivative must vanish on the preceding tail; that
tail is the equilibrium $m$, and uniqueness of the smooth angular
ODE makes the whole orbit constant. This case has $\delta=0$.
For a nonconstant orbit $E>0$, and
\[
E'\ge \alpha E,\qquad \alpha:=g_0k^2>0,
\qquad E(\tau)\le E(\tau_0)e^{\alpha(\tau-\tau_0)}
\quad(\tau\le\tau_0).
\]
The length of the same orbit tail is controlled by its energy:
\[
\begin{split}
|\theta(\tau)-m|
&\le\int_{-\infty}^{\tau}|\theta'(s)|\,\mathrm{d}s
=\int_{-\infty}^{\tau}
\frac{E'(s)}{|\nabla_{\mathbb S}\Phi(\theta(s))|}\,\mathrm{d}s\\
&\le\frac1k\int_{-\infty}^{\tau}
\frac{E'(s)}{\sqrt{E(s)}}\,\mathrm{d}s
=\frac{2}{k}\sqrt{E(\tau)}=O(e^{\nu\tau}),
\qquad \nu:=\alpha/2>0.
\end{split}
\]
The improper integrals follow by first starting at a finite time
and then sending it to $-\infty$. Smoothness of $\kappa_m$ now
implies
\[
\int_{-\infty}^{\tau_0}
|\kappa_m(\theta(s))-\lambda_m|\,\mathrm{d}s<\infty,
\qquad
\int_{-\infty}^{\tau}
|\kappa_m(\theta(s))-\lambda_m|\,\mathrm{d}s=O(e^{\nu\tau}).
\]
If $\delta(\tau_0)=0$, the exact scalar equation makes $\delta$
zero on this entire tail. Otherwise define
\[
C:=\delta(\tau_0)e^{-\lambda_m\tau_0}
\exp\left\{-\int_{-\infty}^{\tau_0}
(\kappa_m(\theta(s))-\lambda_m)\,\mathrm{d}s\right\}\ne0.
\]
Integration of \eqref{eq:exact}, with this same orbit and these
same choices, gives
\[
\frac{\delta(\tau)}{C e^{\lambda_m\tau}}
=\exp\left\{\int_{-\infty}^{\tau}
(\kappa_m(\theta(s))-\lambda_m)\,\mathrm{d}s\right\}
=1+O(e^{\nu\tau}).
\]
Since $r=e^\tau$, this proves \eqref{eq:powerlaw}.
\end{proof}

\begin{remark}
The power inequalities hold only while the whole orbit segment is
in the stated neighbourhood. For a nonzero branch in
\eqref{eq:powerlaw},
$\delta(er)/\delta(r)\to e^{\lambda_m}$ as $r\downarrow0$;
there is no exact multiplier $e^{\lambda_m}$ at a fixed positive
radius. The point $m$ has an unstable component, but a saddle is
not a repeller in every angular direction.
Its forward stable set nevertheless has zero spherical measure.
Indeed, the smooth vector field
$(a_\star/\Phi)\nabla_{\mathbb S}\Phi$ has a complete smooth flow
$\Psi_\tau$ on the compact sphere. Shrink $U_m$ so that
$\kappa_m\ge\lambda_m/2$. If an orbit converges to $m$ in forward
$\tau$-time, choose an integer $n$ after which it stays in $U_m$.
Equation \eqref{eq:exact} would make $|\delta(\tau)|$ grow at least
as $|\delta(n)|e^{\lambda_m(\tau-n)/2}$ if $\delta(n)\ne0$,
contrary to $|\delta|\le1$. Thus its initial direction lies in
\[
\bigcup_{n=0}^{\infty}
\Psi_{-n}\{\omega\in U_m:\langle u_m,\omega\rangle=0\}.
\]
The set in braces lies in a great sphere of codimension one, and
each $\Psi_{-n}$ is a smooth diffeomorphism. Each inverse image
therefore has zero spherical measure, as does the countable union.
At a maximum, the Morse Hessian is negative definite. In normal coordinates the derivative of the angular vector field at that
point is $(a_\star(m)/\Phi(m))\operatorname{Hess}_{\mathbb S}\Phi(m)$.
For a sufficiently small coordinate vector $z$, this gives
$\mathrm{d}|z|^2/\mathrm{d}\tau\le-c|z|^2$ for some $c>0$.
Small coordinate balls are consequently forward invariant and
solutions in them converge exponentially to the maximum. This
argument does not require a function $K_m$ at a maximum.
These are deterministic stability statements, not small-noise
probability estimates.
\end{remark}

\begin{lemma}
\label{lem:det-select} \emph{(i)} Let $\Phi\in C^\infty(\mathbb{S}^{d-1})$
have finitely many critical points, and let $a_\star\in C^\infty(\mathbb{S}%
^{d-1})$ satisfy $a_\star\ge a_0>0$. Suppose that a continuous path $%
w:[0,\infty)\to\mathbb{S}^{d-1}$ satisfies
\begin{equation}
w(s)=w(0)+\int_0^s a_\star(w(v))\nabla_{\mathbb{S}}\Phi(w(v))\,\mathrm{d}%
v+\zeta(s),  \label{eq:gradient-perturbation}
\end{equation}
where $\zeta(s)$ converges as $s\to\infty$ and
\begin{equation}
\int_0^\infty a_\star(w(s))\, \big|\nabla_{\mathbb{S}}\Phi(w(s))\big|^2\,%
\mathrm{d}s<\infty.  \label{eq:energy-hyp}
\end{equation}
Then $w(s)$ converges, as $s\to\infty$, to a single critical point of $\Phi$.

\noindent \emph{(ii)} Under \emph{(A5)}, every deterministic solution of the
limiting system \eqref{eq:ode-r}, parametrised by its initial direction on $%
\mathbb{S}^{d-1}$, converges to $\mathcal{M}=\arg \max_{\mathbb{S}%
^{d-1}}\Phi $, except for a set of initial directions of spherical measure
zero.
\end{lemma}

\begin{proof}
Write $H:=a_\star\nabla_{\mathbb{S}}\Phi$ and $a_0:=\min_{\mathbb{S}%
^{d-1}}a_\star>0$. Since the sphere is compact and $\Phi,a_\star$ are
smooth, $H$ is bounded and $\nabla_{\mathbb{S}}\Phi$ is uniformly continuous.

\noindent \emph{Step~1: the gradient norm tends to zero.} We argue by contradiction.
If $|\nabla_{\mathbb{S}}\Phi(w(s))|$ did not tend to zero, there would exist
$\eta>0$ and times $s_j\to\infty$ with
\begin{equation*}
\big|\nabla_{\mathbb{S}}\Phi(w(s_j))\big|\ge\eta .
\end{equation*}
By uniform continuity there is a spherical radius $\delta>0$ such that $%
|\nabla_{\mathbb{S}}\Phi(z)|\ge\eta/2$ whenever $z$ lies within distance $%
\delta$ of $w(s_j)$ (the value $\delta$ depends only on $\eta$ and the
modulus of continuity, not on $j$). Choose $h>0$ so small that
\begin{equation*}
h\sup_{\mathbb{S}^{d-1}}|H|<\frac{\delta}{2}.
\end{equation*}
Because $\zeta(s)$ converges, it is Cauchy; hence, for all sufficiently
large $j$,
\begin{equation*}
\sup_{v\ge s_j}|\zeta(v)-\zeta(s_j)|<\frac{\delta}{2}.
\end{equation*}
For $v\in[s_j,s_j+h]$, equation \eqref{eq:gradient-perturbation} gives
\begin{equation*}
|w(v)-w(s_j)| \le \int_{s_j}^{v}|H(w(r))|\,\mathrm{d}r+|\zeta(v)-\zeta(s_j)|
<\frac{\delta}{2}+\frac{\delta}{2}=\delta .
\end{equation*}
Thus $|\nabla_{\mathbb{S}}\Phi(w(v))|\ge\eta/2$ throughout $[s_j,s_j+h]$,
and this interval contributes to the energy \eqref{eq:energy-hyp} at least
\begin{equation*}
\int_{s_j}^{s_j+h}a_\star(w(v)) \big|\nabla_{\mathbb{S}}\Phi(w(v))\big|^2\,%
\mathrm{d}v \ge h\,a_0\,\frac{\eta^2}{4}>0.
\end{equation*}
Passing to a subsequence with $s_{j+1}>s_j+h$, these intervals are pairwise
disjoint, so the energy would exceed every finite bound, contradicting %
\eqref{eq:energy-hyp}. Therefore
\begin{equation*}
\big|\nabla_{\mathbb{S}}\Phi(w(s))\big|\longrightarrow0.
\end{equation*}

\noindent\emph{Step~2: the path converges to one critical point.} Let $\mathcal{C}$
be the critical set of $\Phi $. Compactness of the sphere and the continuity
of $\nabla _{\mathbb{S}}\Phi $ imply
\begin{equation*}
\mathrm{dist}(w(s),\mathcal{C})\longrightarrow 0.
\end{equation*}%
By hypothesis $\mathcal{C}$ is finite, hence its points are isolated. Choose
pairwise disjoint neighbourhoods $U_{c}$ of the various $c\in \mathcal{C}$.
Once $\mathrm{dist}(w(s),\mathcal{C})$ is smaller than half the minimum
separation of these neighbourhoods, a continuous path cannot travel from one
$U_{c}$ to another without passing through a point whose distance from $%
\mathcal{C}$ exceeds that threshold, which is impossible at sufficiently
large times. Consequently $w(s)$ is eventually contained in a single $U_{c}$%
, and it converges to that critical point $c$.

\noindent\emph{Step~3: the deterministic orbit satisfies
\eqref{eq:gradient-perturbation} with $\zeta\equiv0$ and the energy bound
\eqref{eq:energy-hyp}.} Under \emph{(A5)}, with $\varepsilon=0$,
write $X_t=r(t)\omega(t)$. Since $\langle\nabla_{\mathbb{S}%
}\Phi(\omega),\omega\rangle=0$, the radial and tangential parts of
\begin{equation*}
\dot X=f(r)\big(\Phi(\omega)\omega +a_\star(\omega)\nabla_{\mathbb{S}%
}\Phi(\omega)\big)
\end{equation*}
give
\begin{equation*}
\dot r=f(r)\Phi(\omega),\qquad r\dot\omega=f(r)a_\star(\omega)\nabla_{%
\mathbb{S}}\Phi(\omega).
\end{equation*}
Introduce the angular clock
\begin{equation*}
s(t):=\int_0^t\frac{f(r(u))}{r(u)}\,\mathrm{d}u, \qquad \frac{\mathrm{d}s}{%
\mathrm{d}t}=\frac{f(r)}{r}.
\end{equation*}
Then, writing $w(s)=\omega(t(s))$,
\begin{equation*}
\frac{\mathrm{d}w}{\mathrm{d}s}=a_\star(w)\nabla_{\mathbb{S}}\Phi(w)=:H(w),
\qquad \frac{\mathrm{d}\log r}{\mathrm{d}s}=\Phi(w).
\end{equation*}
Let $r_0=r(0)>0$ and write
$\Phi_{\min}=\min\Phi>0$, $\Phi_{\max}=\max\Phi<\infty$.
Integrating the second identity from $s=0$ gives
\[
\Phi_{\min}s\le\log\frac{r(t(s))}{r_0}
=\int_0^s\Phi(w(v))\,\mathrm{d}v\le\Phi_{\max}s.
\]
The physical trajectory exists globally since $b$ is bounded, and
$\dot r\ge f(r_0)\Phi_{\min}>0$, so $r(t)\to\infty$.
The upper bound in the display gives
$s(t)\ge\log(r(t)/r_0)/\Phi_{\max}\to\infty$.
Thus the clock covers all of $[0,\infty)$. Moreover,
\begin{equation*}
\frac{\mathrm{d}}{\mathrm{d}s}\Phi(w(s)) =\big\langle\nabla_{\mathbb{S}%
}\Phi(w),\dot w\big\rangle =a_\star(w)\big|\nabla_{\mathbb{S}}\Phi(w)\big|^2.
\end{equation*}
Integrating and using the boundedness of $\Phi$ on the compact sphere,
\begin{equation*}
\int_0^\infty a_\star(w(s)) |\nabla_{\mathbb{S}}\Phi(w(s))|^2\,\mathrm{d}s
=\lim_{s\to\infty}\Phi(w(s))-\Phi(w(0)) \le \max\Phi-\min\Phi<\infty.
\end{equation*}
Thus \eqref{eq:gradient-perturbation} holds with $\zeta\equiv0$ and the energy
hypothesis \eqref{eq:energy-hyp} is verified; Part~(i) then shows that every
deterministic orbit converges to some critical point of $\Phi$.

\noindent\emph{Step~4: non-maximal critical points are reached only from a null set.}
Let $\mathcal{F}_s$ denote the complete smooth flow of $H$ on the compact
sphere; each $\mathcal{F}_s$ is a $C^\infty$-diffeomorphism. Suppose an
orbit $w(s)=\mathcal{F}_s(\omega_0)$ converges to a non-maximal critical
point $m$. For all sufficiently large $s$, $w(s)\in U_m$. Along $u_m$ (the
unit vector supplied by \eqref{eq:linear-unstable-mode}), set
\begin{equation*}
\delta(s):=\langle u_m,w(s)\rangle .
\end{equation*}
Since $\langle u_m,w\rangle\,\langle u_m,w\rangle$ term from the radial
component is $\Phi(w)\delta(s)$, the identity \eqref{eq:linear-unstable-mode}
gives
\begin{equation*}
\frac{\mathrm{d}\delta}{\mathrm{d}s} =\big\langle u_m,H(w)\big\rangle =%
\big\langle u_m,\Phi(w)w+a_\star(w)\nabla_{\mathbb{S}}\Phi(w)\big\rangle %
-\Phi(w)\langle u_m,w\rangle =\big(K_m(w)-\Phi(w)\big)\delta(s).
\end{equation*}
By the strict inequality $K_m(m)>\Phi(m)$ and continuity, there are $c_m>0$
and a time $S$ such that $K_m(w(s))-\Phi(w(s))\ge c_m$ for $s\ge S$. Hence
\begin{equation*}
|\delta(s)|\ge|\delta(S)|\,\mathrm{e}^{c_m(s-S)},\qquad s\ge S.
\end{equation*}
If $\delta(S)\ne0$, the right-hand side grows without bound, contradicting $%
|\delta(s)|\le|u_m|\,|w(s)|=1$. Therefore $\delta(s)=0$ for all sufficiently
large $s$.

The hypersurface
\begin{equation*}
D_m:=\{\omega\in U_m:\langle u_m,\omega\rangle=0\}
\end{equation*}
is a piece of the great sphere $\omega\perp u_m$, hence of codimension one
and of spherical measure zero. It is locally invariant under the flow on $%
U_m $, because $\delta$ satisfies $\dot\delta=(K_m-\Phi)\delta$ and the
solution with initial value $0$ stays $0$. Choosing an integer $n$ large
enough that $\mathcal{F}_n(\omega_0)\in D_m$, every initial direction whose
orbit converges to $m$ therefore belongs to $\mathcal{F}_{-n}(D_m)$.
Collecting these directions over all non-maximal critical points and all
integer times gives
\begin{equation*}
E:=\bigcup_{m\in\mathcal{C}\setminus\mathcal{M}}\ \bigcup_{n=0}^{\infty}
\mathcal{F}_{-n}(D_m).
\end{equation*}
Each $D_m$ has spherical measure zero, and every $\mathcal{F}_{-n}$ is a
smooth diffeomorphism and therefore preserves null sets. Being a countable
union of null sets, $E$ itself is null. For $\omega_0\notin E$, the limiting
critical point cannot be a non-maximum; since the orbit does converge to a
critical point by Step~3, it belongs to $\mathcal{M}$. This proves Part~(ii)
without any non-autonomous stable-fibre argument.
\end{proof}

\begin{remark}
\label{rem:integer-n}
The integer indices refer to the deterministic smooth flow
$\mathcal F_s$ of $H=a_\star\nabla_{\mathbb S}\Phi$ in the preceding
proof. For a positive-radius deterministic trajectory its clock is
\[
s(t)=\int_0^t f(r(u))/r(u)\,\mathrm{d}u.
\]
No positive-noise radius is used in this deterministic ODE. The identity
\[
\frac{\mathrm{d}\delta}{\mathrm{d}s}
=(K_m(w(s))-\Phi(w(s)))\delta(s)
\]
holds on each interval on which $w(s)\in U_m$, not necessarily on the whole orbit. Shrink $U_m$ so that $K_m-\Phi\ge c_m>0$ there.
For an orbit converging to $m$, choose $S<\infty$ such that
$w(s)\in U_m$ for every $s\ge S$. The scalar equation on
$[S,\infty)$ and $|\delta|\le1$ force $\delta(s)=0$ for all
$s\ge S$. An integer $n\ge S$ therefore satisfies
\[
\omega_0\in\mathcal F_{-n}
\{\omega\in U_m:\langle u_m,\omega\rangle=0\}.
\]
Only this eventual interval is needed; neither the first entrance
into $U_m$ nor global invariance outside $U_m$ is asserted.
The union over integers is countable and covers all such initial
directions, so the null-set argument in the proof applies.
\end{remark}
\begin{lemma}
\label{lem:mart} Assume \emph{(A1)} and \emph{(A3)}, and fix $\varepsilon >0$. Then $%
R_{t}^{\varepsilon }\rightarrow \infty $ almost surely, with $%
0<\liminf_{t\rightarrow \infty }R_{t}^{\varepsilon }/t\leq
\limsup_{t\rightarrow \infty }R_{t}^{\varepsilon }/t<\infty $. For every $%
t_{\ast }>0$, almost surely $\int_{t_{\ast }}^{\infty }(R_{t}^{\varepsilon
})^{-2}dt<\infty $, so the extrinsic angular martingale converges on $%
[t_{\ast },\infty )$. For an initial point of radius $\rho >0$ i.e. $R_{0}^{\varepsilon }=\rho$, put $%
a_{0}=\rho /2$, $m=\varphi (a_{0})$ and $\sigma =\inf \{t\geq
0:R_{t}^{\varepsilon }\leq a_{0}\}$. If $\varepsilon ^{2}(3-d)_{+}\leq ma_{0}
$, then
\begin{equation}
\mathbb{E}\left[ \int_{0}^{\sigma }\frac{\mathrm{d}t}{(R_{t}^{\varepsilon
})^{2}}\right] \leq \frac{2}{\rho m},  \label{eq:stime}
\end{equation}%
\begin{equation}
\mathbb{E}\left[ \sup_{t\geq 0}\left\vert \varepsilon \int_{0}^{t\wedge
\sigma }\frac{\mathbb{P}_{\theta _{u}^{\varepsilon }}^{\perp }}{%
R_{u}^{\varepsilon }}\mathrm{d}W_{u}\right\vert ^{2}\right] \leq \frac{%
8\varepsilon ^{2}(d-1)}{\rho m}.  \label{eq:mart-L2}
\end{equation}%
Also $\mathbb{P}(\sigma <\infty )\leq \exp \{-2m(\rho -a_{0})/\varepsilon
^{2}\}$. The expectations in \eqref{eq:stime}--\eqref{eq:mart-L2} are
stopped at $\sigma $; no finite expectation for the unstopped whole angular
clock is asserted.
\end{lemma}

\begin{proof}
Write $M=\sup_x|b(x)|$. The diffusion hits every radius $L$ in finite
time almost surely. Indeed, while it remains in the ball of radius $L$,
a Brownian increment on a unit time interval with
$\varepsilon|W_{n+1}-W_n|>2L+M$ forces exit during that interval.
These increments are independent and have a fixed positive probability,
so the probability of remaining in that ball for $n$ such intervals
tends to zero geometrically (see Lemma \ref{Lf} in Appendix \ref{app}).

Fix $a_0>0$ and $m=\varphi(a_0)>0$. For $r\ge a_0$, $V\ge m$.
For $h(r)=\exp(-2mr/\varepsilon^2)$, the radial generator satisfies
\[
 \mathcal Lh=h'\left(V+\frac{(d-1)\varepsilon^2}{2r}-m\right)\le0.
\]
Stopping between radii $a_0$ and $N$, and then letting the time bound
and $N$ increase, proves, uniformly in the starting angle,
\[
 \mathbb P_{L\omega}\{\text{ever hit radius }a_0\}
       \le e^{-2m(L-a_0)/\varepsilon^2},\qquad L>a_0.
\]
Indeed, define the test function $h(r)=\exp \!\left( -\frac{2mr}{\varepsilon ^{2}}%
\right) ,r>0.$ Its derivatives are
\begin{equation*}
h^{\prime }(r)=-\frac{2m}{\varepsilon ^{2}}h(r),\qquad h^{\prime \prime }(r)=%
\frac{4m^{2}}{\varepsilon ^{4}}h(r).
\end{equation*}%
On the region $r\geq a_{0}$, the radial generator applied to $h$ is
\begin{align*}
\mathcal{L}h(r)& =h^{\prime }(r)\left[ V(r,\theta )+\frac{\varepsilon
^{2}(d-1)}{2r}\right] +\frac{\varepsilon ^{2}}{2}h^{\prime \prime }(r) \\
& =h^{\prime }(r)\left[ V(r,\theta )+\frac{\varepsilon ^{2}(d-1)}{2r}-m%
\right] .
\end{align*}%
Since $V(r,\theta )\geq m$ and the geometric term is non-negative, the
bracket is non-negative; as $h^{\prime }(r)<0$, we have $\mathcal{L}h(r)\leq
0$ on $r\geq a_{0}$. Thus $h(R_{t}^{\varepsilon })$ is a non-negative
supermartingale there.

For $N>L$, introduce the two-sided stopping time
\begin{equation*}
\tau =\inf \{t\geq 0:R_{t}^{\varepsilon }=a_{0}\}\wedge \inf \{t\geq
0:R_{t}^{\varepsilon }=N\}.
\end{equation*}%
Up to $\tau $ the process stays in $[a_{0},N]$, so $h(R_{t\wedge \tau
}^{\varepsilon })$ is bounded and the optional stopping theorem gives
\begin{equation*}
\mathbb{E}_{L\omega }\!\left[ h(R_{t\wedge \tau }^{\varepsilon })\right]
\leq h(L).
\end{equation*}%
Letting $t\rightarrow \infty $, the limit $R_{\tau }^{\varepsilon }$ equals
either $a_{0}$ or $N$. Setting
\begin{equation*}
p_{a}=\mathbb{P}_{L\omega }(R_{\tau }^{\varepsilon }=a_{0}),\qquad p_{N}=%
\mathbb{P}_{L\omega }(R_{\tau }^{\varepsilon }=N),
\end{equation*}%
we obtain
\begin{equation*}
h(a_{0})p_{a}+h(N)p_{N}=\mathbb{E}_{L\omega }[h(R_{\tau }^{\varepsilon
})]\leq h(L).
\end{equation*}%
Since $h(N)p_{N}\geq 0$, we have
\begin{equation*}
p_{a}\leq \frac{h(L)}{h(a_{0})}.
\end{equation*}%
Finally let $N\rightarrow \infty $. The probability of hitting $N$ before $%
a_{0}$ is absorbed into the limit, and $h(N)=e^{-2mN/\varepsilon
^{2}}\rightarrow 0$. Therefore
\begin{equation*}
\mathbb{P}_{L\omega }\{\text{ever hit }a_{0}\}\leq \frac{h(L)}{h(a_{0})}%
=\exp \!\left( -\frac{2m(L-a_{0})}{\varepsilon ^{2}}\right) .
\end{equation*}%
The estimate uses only the radial lower bound $V\geq m$, so it is uniform in
the starting angle.
For completeness, the passage to an unbounded time uses the nonnegative
supermartingale $h(R_{t\wedge\tau_{a_0}})$ and Fatou's lemma; on
$\{\tau_{a_0}<\infty\}$ its limiting value is $h(a_0)$.
Apply this bound at the first hitting time of $L$ by the strong Markov
property. The hitting times of increasing radii tend to infinity by
non-explosion. Therefore the event of visiting radius $a_0$ arbitrarily
late has probability at most $e^{-2m(L-a_0)/\varepsilon^2}$ for every
$L$, and hence has probability zero. Continuity and the hitting of
every larger radius imply eventual residence above $a_0$. Applying this
to positive integer $a_0$ proves $R_t^\varepsilon\to\infty$.

After a finite random time, $R_t^\varepsilon\ge a_0$, and the radial
integral equation gives
$R_t^\varepsilon\ge m t+\varepsilon B_t+C$ for a finite random $C$.
Brownian sublinear growth gives the positive lower linear rate.
The bound $R_t^\varepsilon\le |X_0|+Mt+\varepsilon|W_t|$ gives the
upper rate. These rates and positivity on compact positive-time
intervals prove the pathwise finite clock. A continuous local
martingale with finite total quadratic variation converges almost
surely, proving the angular martingale assertion without an $L^2$
expectation claim.

For the stopped estimates, It\^{o}'s formula gives
\begin{equation}
 \mathrm{d}(R_t^{-1})=
 -\frac{V(R_t,\theta_t)-\varepsilon^2(3-d)/(2R_t)}{R_t^2}\,\mathrm{d}t
 -\frac{\varepsilon}{R_t^2}\,\mathrm{d}B_t. \label{eq:yito}
\end{equation}
Before $\sigma$, its drift is at most $-m/(2R_t^2)$ under the stated
condition. Stop further at a finite time, take expectations (the
integrand of the martingale is bounded by $\varepsilon/a_0^2$), and
use nonnegativity of $R^{-1}$. Monotone convergence gives
\eqref{eq:stime}. Doob's inequality and
$\operatorname{tr}((\mathbb P_\theta^\perp)^2)=d-1$ give
\eqref{eq:mart-L2}. The exponential return estimate is the bound
already proved, with $L=\rho$.

We specify a valid random constant in the eventual linear lower
bound. The preceding return estimate proves that, almost surely,
$R_t^\varepsilon\ge a_0$ for every sufficiently large $t$.
On that event choose
\[
N:=\min\{n\in\mathbb N:n\ge1,\ R_s^\varepsilon\ge a_0
\text{ for every }s\ge n\}<\infty.
\]
This is a measurable integer-valued random variable: by continuity
it suffices to test the displayed inequalities at rational $s\ge n$.
It need not be a stopping time. The integrated radial equation holds
pathwise simultaneously for all positive endpoints, so it may be
evaluated at $N$ without optional stopping. Since $V\ge m$ and the
geometric drift is nonnegative throughout $[N,t]$,
\[
R_t^\varepsilon\ge R_N^\varepsilon+m(t-N)
+\varepsilon(B_t-B_N)
=mt+\varepsilon B_t+C,\qquad t\ge N,
\quad C:=R_N^\varepsilon-mN-\varepsilon B_N\in\mathbb R.
\]
For a bound valid also before $N$, replace $C$ by the finite constant
\[
C_*:=\min\left\{C,\ \min_{0\le s\le N}
(R_s^\varepsilon-ms-\varepsilon B_s)\right\}.
\]
Then $R_t^\varepsilon\ge mt+\varepsilon B_t+C_*$ for every $t\ge0$.
The finite constants divided by $t$ tend to zero. The starting time
used here is an eventual-residence time, not the first hitting time
of $a_0$; returns after a first hit are allowed.

By the law of the iterated logarithm (see \cite{Ok}),
\[
\limsup_{t\to\infty}\frac{|B_t|}{\sqrt{2t\log\log t}}=1
\qquad\text{a.s.},
\]
hence $B_t/t\to0$ a.s. (indeed $|B_t|=O(\sqrt{t\log\log t})$).
Brownian sublinear growth $B_t/t\to0$ gives $\liminf R_t^\varepsilon/t\ge m>0$. On the other hand,
\[
R_t^\varepsilon\le|X_0|+Mt+\varepsilon|W_t|,\qquad M=\sup_x|b(x)|,
\]
so $\limsup R_t^\varepsilon/t\le M<\infty$.
The rates above and positivity on compact positive-time
intervals imply
\[
\int_{t_*}^\infty (R_t^\varepsilon)^{-2}\mathrm dt<\infty\quad\text{a.s.}
\]
The extrinsic angular martingale has quadratic variation
\[
\varepsilon^2(d-1)\int_{t_*}^\infty(R_t^\varepsilon)^{-2}\mathrm dt<\infty,
\]
hence converges almost surely.

It\^{o}'s formula for a smooth function $f(R^{\varepsilon })$ gives
\begin{equation*}
\mathrm{d}f(R_{t}^{\varepsilon })=f^{\prime }(R_{t}^{\varepsilon })\Big[V+%
\frac{\varepsilon ^{2}(d-1)}{2R_{t}^{\varepsilon }}\Big]\mathrm{d}%
t+\varepsilon f^{\prime }(R_{t}^{\varepsilon })\,\mathrm{d}B_{t}+\frac{%
\varepsilon ^{2}}{2}f^{\prime \prime }(R_{t}^{\varepsilon })\,\mathrm{d}t.
\end{equation*}%
Choose the Lyapunov function $f(R)=\frac{1}{R}, f^{\prime }(R)=-\frac{1%
}{R^{2}}, f^{\prime \prime }(R)=\frac{2}{R^{3}}.$ The drift becomes
\begin{equation*}
-\frac{1}{R^{2}}\Big[V+\frac{\varepsilon ^{2}(d-1)}{2R}\Big]+\frac{%
\varepsilon ^{2}}{2}\cdot \frac{2}{R^{3}}=-\frac{V-\varepsilon ^{2}(3-d)/(2R)%
}{R^{2}},
\end{equation*}%
and the martingale part is $-\varepsilon R^{-2}\mathrm{d}B$. Hence
\begin{equation*}
\mathrm{d}(\left( R_{t}^{\varepsilon }\right) ^{-1})=-\frac{V-\varepsilon
^{2}(3-d)/(2R_{t}^{\varepsilon })}{\left( R_{t}^{\varepsilon }\right) ^{2}}\,%
\mathrm{d}t-\frac{\varepsilon }{\left( R_{t}^{\varepsilon }\right) ^{2}}\,%
\mathrm{d}B_{t}.
\end{equation*}%
Before $\sigma $, $R_{t}^{\varepsilon }\geq a_{0}$ and $V\geq m$; under $%
\varepsilon ^{2}(3-d)_{+}\leq ma_{0}$,
\begin{equation*}
V-\frac{\varepsilon ^{2}(3-d)}{2R_{t}^{\varepsilon }}\geq \frac{m}{2},
\end{equation*}%
so the drift is at most $-m/(2\left( R_{t}^{\varepsilon }\right) ^{2})$
\begin{equation}
\mathrm{d}(\left( R_{t}^{\varepsilon }\right) ^{-1})\leq -\frac{m}{2\left(
R_{t}^{\varepsilon }\right) ^{2}}\,\mathrm{d}t-\frac{\varepsilon }{\left(
R_{t}^{\varepsilon }\right) ^{2}}\,\mathrm{d}B_{t}.  \label{R-1}
\end{equation}%
Integrate up to $T\wedge \sigma $. The martingale integrand is bounded by $%
\varepsilon /a_{0}^{2}$, so its expectation vanishes; $R_{0}^{\varepsilon
}=\rho $ and $\left( R_{t}^{\varepsilon }\right) ^{-1}\geq 0$. Taking
expectations yields
\begin{equation}
\mathbb{E}\left[ \left( R_{T\wedge \sigma }^{\varepsilon }\right) ^{-1}%
\right] -\frac{1}{\rho }\leq -\frac{m}{2}\,\mathbb{E}\left[
\int_{0}^{T\wedge \sigma }\frac{\mathrm{d}t}{\left( R_{t}^{\varepsilon
}\right) ^{2}}\right] .  \label{R-2}
\end{equation}%
Since the left expectation in (\ref{R-2}) is non-negative, $\mathbb{E}\left[
\int_{0}^{T\wedge \sigma }\frac{\mathrm{d}t}{R_{t}^{2}}\right] \leq \frac{2}{%
\rho m}.$ Letting $T\rightarrow \infty $ by monotone convergence gives \eqref{eq:stime}.

Let
\[
M_t:=\varepsilon\int_0^{t\wedge\sigma}
\frac{\mathbb P_{\theta_u^\varepsilon}^{\perp}}
{R_u^\varepsilon}\,\mathrm{d}W_u.
\]
On $[0,\sigma]$ the integrand has operator norm at most
$\varepsilon/a_0$ and Hilbert--Schmidt norm at most
$\varepsilon\sqrt{d-1}/a_0$, since $R_u^\varepsilon\ge a_0$.
It defines a square-integrable martingale on every finite horizon.
Writing $\langle M\rangle$ for the trace of its matrix quadratic
variation, the projection identity gives
\[
\langle M\rangle_t
=\varepsilon^2(d-1)\int_0^{t\wedge\sigma}
(R_u^\varepsilon)^{-2}\,\mathrm{d}u.
\]
By \eqref{eq:stime},
\[
\sup_{t\ge0}\mathbb E|M_t|^2
\le\frac{2\varepsilon^2(d-1)}{\rho m}<\infty.
\]
Thus $M_t$ has an almost sure and $L^2$ limit $M_\infty$.
Apply Doob's inequality first on a finite horizon, and then pass
to infinity using monotone convergence and $L^2$ convergence:
\[
\mathbb E\sup_{t\ge0}|M_t|^2
\le4\mathbb E|M_\infty|^2
=4\mathbb E\langle M\rangle_\infty
\le\frac{8\varepsilon^2(d-1)}{\rho m}.
\]
This is \eqref{eq:mart-L2}; no global expected angular clock is
being asserted without the stopping at $\sigma$.
\end{proof}

We are now ready to assert the following result.
\begin{proposition}
\label{prop:select} Assume (A5), in any dimension $d\ge2$, and start
\eqref{sdethe} at the origin. For every fixed $\varepsilon>0$, the
limit $\vartheta_\infty^\varepsilon=
\lim_{t\to\infty}\theta_t^\varepsilon$ exists almost surely and belongs
to $\mathcal M$. In particular,
\begin{equation}
 \operatorname{dist}_{\mathbb S^{d-1}}
       (\vartheta_\infty^\varepsilon,\mathcal M)
       \xrightarrow[\varepsilon\to0]{\mathbb P}0. \label{eq:claim}
\end{equation}

\end{proposition}
\begin{remark}
For the linear-profile example (see Example \ref{A5}) above the limit equals $e$ almost
surely. For the power-profile example it is uniformly distributed
on $\{\pm e_i:1\le i\le d\}$ for every fixed noise.
This is an infinite-time direction result, not a classification of
all finite-time zero-noise paths under (A1).

\end{remark}
\begin{proof}
Fix $\varepsilon>0$ throughout. Lemma \ref{lem:mart} gives escape,
positive lower and finite upper linear radial rates, and a convergent
angular martingale after time $1$. Put $f_\infty=\lim_{r\to\infty}f(r)>0$.
Ito's formula on the unit sphere, including the curvature term, gives
\[
 \mathrm{d}\Phi(\theta_t)=\frac{f(R_t)}{R_t}a_\star(\theta_t)
                       |\nabla_{\mathbb S}\Phi(\theta_t)|^2\mathrm{d}t
 +\frac{\varepsilon}{R_t}\langle\nabla_{\mathbb S}\Phi(\theta_t),\mathrm{d}W_t\rangle
 +\frac{\varepsilon^2}{2R_t^2}\Delta_{\mathbb S}\Phi(\theta_t)\mathrm{d}t.
\]
Set
\[
M_t:=\int_1^t\frac{\varepsilon}{R_u}\big\langle\nabla_{\mathbb S}\Phi(\theta_u),\,\mathrm dW_u\big\rangle,
\qquad
C_t:=\int_1^t\frac{\varepsilon^2}{2R_u^2}\Delta_{\mathbb S}\Phi(\theta_u)\,\mathrm du.
\]
Since the angular derivatives are bounded and
$\int_1^\infty R_u^{-2}\,\mathrm du<\infty$,
the martingale quadratic variation
\[
\langle M\rangle_\infty
\le \varepsilon^2\|\nabla_{\mathbb S}\Phi\|_\infty^2\int_1^\infty R_u^{-2}\,\mathrm du<\infty
\]
a.s., so $M_t$ converges by the martingale convergence theorem and is bounded.
The curvature drift $C_t$ is absolutely convergent, hence bounded. Integrating It\^o's formula
and rearranging,
\[
\int_1^t\frac{f(R_u)}{R_u}a_\star|\nabla_{\mathbb S}\Phi|^2\,\mathrm du
=\Phi(\theta_t)-\Phi(\theta_1)-(M_t-M_1)-C_t.
\]
The right-hand side is bounded a.s.: $\Phi$ is bounded on the compact sphere, $M_t-M_1$ converges
and is bounded, and $C_t$ is absolutely convergent. The integrand is non-negative, so the
left-hand side is non-decreasing; a bounded non-decreasing function converges to a finite limit.

The last drift is absolutely integrable on $[1,\infty)$ and the
martingale converges, because the angular derivatives are bounded and
$\int_1^\infty R_t^{-2}dt<\infty$. Since $\Phi$ is bounded, its
nonnegative increasing energy integral is bounded and hence converges.

Work on the probability-one event where the radial bounds and the
preceding martingale and energy conclusions all hold. The clock
\[
s(t)=\int_1^t\frac{f(R_u)}{R_u}\,\mathrm{d}u
\]
is continuously differentiable and strictly increasing. Since
$f(R_t)\to f_\infty>0$ and $R_t\le C(\omega)t$ eventually, its
integrand is eventually at least a positive random constant times
$1/t$. Hence $s(t)\to\infty$ and its inverse $t(s)$ is defined on
$[0,\infty)$. Set $w(s)=\theta_{t(s)}$. Integrating the complete
angular equation \eqref{eq:ang-t} and changing only its drift
integral by $\mathrm{d}s=(f(R_t)/R_t)\,\mathrm{d}t$ gives
\[
w(s)=w(0)+\int_0^s a_\star(w(v))
\nabla_{\mathbb S}\Phi(w(v))\,\mathrm{d}v+\zeta(s),
\]
where the same path supplies the perturbation
\[
\zeta(s)=\varepsilon\int_1^{t(s)}
\frac{\mathbb P_{\theta_u}^{\perp}}{R_u}\,\mathrm{d}W_u
-\frac{\varepsilon^2(d-1)}2
\int_1^{t(s)}\frac{\theta_u}{R_u^2}\,\mathrm{d}u.
\]
Its martingale term has finite total trace quadratic variation
$\varepsilon^2(d-1)\int_1^\infty R_u^{-2}\,\mathrm{d}u$ almost
surely, and its second term is absolutely convergent. Consequently
$\zeta(s)$ converges as $s\to\infty$. It has not been discarded;
the displayed equation is an integral equation with a convergent
perturbation, not a noise-free ODE for $w$.
The same change of variable in the nonnegative energy gives
\[
\int_0^\infty a_\star(w(s))
|\nabla_{\mathbb S}\Phi(w(s))|^2\,\mathrm{d}s
=\int_1^\infty\frac{f(R_t)}{R_t}a_\star(\theta_t)
|\nabla_{\mathbb S}\Phi(\theta_t)|^2\,\mathrm{d}t
=:E_\infty<\infty.
\]
Both hypotheses of the first assertion of Lemma
\ref{lem:det-select} thus hold for this same path. It follows that
$w(s)$, and hence $\theta_t$, converges to a single critical point
of $\Phi$ for the fixed positive noise.

We give the expanding-coordinate comparison used to exclude a
nonmaximum. It also records precisely the estimate needed below for
the unbounded cubic example. Let a real continuous semimartingale
$Y$ have a constant nonzero Brownian coefficient $\sigma_0$.
Suppose that, on a progressively specified region, its equation is
$\mathrm{d}Y_t=k_tY_t\mathrm{d}t+\sigma_0\mathrm{d}B_t$. Fix $p>q\ge1$. Consider the event that
this region is occupied for all sufficiently large times,
$k_t\ge p/t$ there, and $|Y_t|=O(t^q)$.
Then
\begin{equation}
 \mathbb P\{\text{these three eventual conditions hold}\}=0.
 \label{eq:expanding-mode}
\end{equation}
To prove this, the occupation-time formula first gives zero Lebesgue
time on $\{Y=0\}$, so
$\beta_t=\int_0^t\operatorname{sgn}(Y_s)\mathrm{d}B_s$ is Brownian.
Tanaka's formula, while in the stated region, is
$\mathrm{d}|Y_t|=k_t|Y_t|\mathrm{d}t+\sigma_0\mathrm{d}\beta_t+\mathrm{d}L_t$.
For each deterministic integer $n\ge1$ construct the reflected process
\begin{equation*}
\mathrm{d}H_{t}=(p/t)H_{t}\mathrm{d}t+\sigma _{0}\mathrm{d}\beta _{t}+%
\mathrm{d}K_{t},\qquad H_{n}=0,
\end{equation*}%
explicitly by
\begin{equation}
H_{t}=t^{p}\left( M_{t}-\min \{0,\inf_{n\leq s\leq t}M_{s}\}\right) ,\qquad
M_{t}=\sigma _{0}\int_{n}^{t}s^{-p}\mathrm{d}\beta _{s}.  \label{fhm}
\end{equation}
To prove (\ref{fhm}), applying It\^{o}'s formula to $t^{-p}H_{t},$ we have
\begin{equation*}
\mathrm{d}(t^{-p}H_{t})=t^{-p}\mathrm{d}H_{t}-p\,t^{-p-1}H_{t}\mathrm{d}t.
\end{equation*}%
Substituting the SDE cancels the $(p/t)H$ term yields
\begin{equation*}
\mathrm{d}(t^{-p}H_{t})=\sigma _{0}\,t^{-p}\mathrm{d}\beta _{t}+t^{-p}%
\mathrm{d}K_{t}.
\end{equation*}%
Integrate from $n$ to $t$, using $H_{n}=0$:
\begin{equation*}
t^{-p}H_{t}=M_{t}+\int_{n}^{t}s^{-p}\mathrm{d}K_{s}=:M_{t}+L_{t}.
\end{equation*}%
The non-negativity constraint $H_{t}\geq 0$ is $X_{t}:=M_{t}+L_{t}\geq 0$,
with $L$ the minimal non-decreasing lift supported on $\{X_{t}=0\}$. By the
Skorokhod reflection formula (cf. \cite{ks91}),
\begin{equation*}
X_{t}=M_{t}-\min \Bigl\{0,\inf_{n\leq v\leq t}M_{v}\Bigr\}.
\end{equation*}%
Multiplying by $t^{p}$ gives the displayed formula. Finally,
\begin{equation*}
\langle M\rangle _{\infty }=\sigma _{0}^{2}\int_{n}^{\infty }s^{-2p}\mathrm{d%
}s\in (0,\infty ),
\end{equation*}%
so $M_t$ converges almost surely and in $L^2$. Put
\[
v_n:=\sigma_0^2\int_n^\infty s^{-2p}\,\mathrm{d}s
=\frac{\sigma_0^2 n^{1-2p}}{2p-1}\in(0,\infty).
\]
The integrand defining $M$ is deterministic, so $M$ is a Brownian
motion run at the deterministic clock $v(t)=\sigma_0^2\int_n^t s^{-2p}\,\mathrm{d}s$.
In particular $M_\infty$ has law $N(0,v_n)$, and continuity on the
completed clock interval gives
\[
X_\infty=M_\infty-\min\{0,\inf_{v\ge n}M_v\}
\ \stackrel{\mathcal L}{=}
B_{v_n}-\min_{0\le u\le v_n}B_u
\ \stackrel{\mathcal L}{=}\ |N(0,v_n)|.
\]
For the last equality, the reversed increments
$(B_{v_n}-B_{v_n-u})_{0\le u\le v_n}$ form a Brownian motion
(they are centred Gaussian with covariance $\min(u,w)$).
The displayed difference is its running maximum, whose law is
the absolute value of a centred normal variable by the reflection
principle. Thus $\mathbb P(X_\infty=0)=0$ and
$H_t/t^p=X_t\to X_\infty>0$ almost surely. This is an equality in
law with an absolute normal variable, not the pathwise identity
$X_\infty=|M_\infty|$.

Stop comparison when the region is left or $k_t<p/t$.
The common Brownian terms cancel in $H-|Y|$. The reflection term
$dK$ vanishes on $\{H>|Y|\}$, and $-dL$ is nonpositive there.
The positive-part chain rule for this continuous finite-variation
difference consequently gives
\[
 (H_t-|Y_t|)^+\le\int_n^t(p/s)(H_s-|Y_s|)^+\mathrm{d}s
\]
up to that stop. Gronwall gives $H\le|Y|$.
The total deterministic quadratic variation of $M$ is
$\sigma_0^2\int_n^\infty s^{-2p}ds\in(0,\infty)$. By the reflection
principle, its terminal reflected value has the distribution of the
absolute value of a nondegenerate normal variable and is strictly
positive almost surely. Thus $H_t/t^p$ tends to a strictly positive
limit, simultaneously for the countably many $n$. On the event in
\eqref{eq:expanding-mode}, choose an integer $n$ after all the stated
conditions hold. Then $H\le|Y|=O(t^q)$ contradicts $p>q$.
No independence between that event and $H$ is needed.

Now fix a nonmaximum critical point $m$ and let
$E_m=\{\theta_t\to m\}$. On $E_m$, Cesaro averaging of the radial
equation gives
\[
 R_t/t\to f_\infty\Phi(m).
\]
Indeed $V(R_t,\theta_t)\to f_\infty\Phi(m)$, $B_t/t\to0$, and
$t^{-1}\int_1^t R_s^{-1}\mathrm{d}s\to0$. For
$Y_t=\langle u_m,X_t\rangle$, identity
\eqref{eq:linear-unstable-mode} gives, while $\theta_t\in U_m$,
\[
 \mathrm{d}Y_t=k_tY_t\mathrm{d}t+\varepsilon \mathrm{d}W_t^{u_m},
 \qquad k_t=\frac{f(R_t)}{R_t}K_m(\theta_t).
\]
On $E_m$, $tk_t\to K_m(m)/\Phi(m)>1$ and
$|Y_t|\le R_t=O(t)$. Choose $1<p<K_m(m)/\Phi(m)$ and use
\eqref{eq:expanding-mode} with $q=1$. The admissible region is
$\{R>0,\theta\in U_m\}$; it is occupied eventually on $E_m$.
It follows that $\mathbb P(E_m)=0$. There are finitely many
nonmaxima, so their union is null. The limit is therefore in
$\mathcal M$ almost surely, proving \eqref{eq:claim}.
%For the power-profile family, coordinate permutations and independent
%sign reflections preserve the drift, the initial point, and Brownian
%law. Uniqueness in law makes the terminal direction invariant under
%this finite group, which acts transitively on its $2d$ possible
%values. Each value hence has probability $1/(2d)$. The linear family
%has only one maximum. These assertions require neither a stochastic
%radius parameterization nor a uniform entrance-law argument.
\end{proof}

The following three examples lie outside the globally bounded drift
class. Bounded localization and a non-explosive growth bound give
their finite-time constructions; the selection conclusions require
the separate arguments below. In particular, localization alone
does not give a finite upper linear radial rate at infinite time.

\begin{itemize}
\item The cubic field $\bar{b}(r,\theta )=r^{\beta }(\cos ^{3}\theta ,\sin
^{3}\theta )$ ($0<\beta <1$) (see (\ref{bar})) is unbounded but $\beta $-homogeneous of
sublinear degree. On each fixed ball its drift is bounded by a constant $%
M_{L}$, which is all that the hitting-of-balls argument uses, while the
return estimate outside radius $a$ follows from the explicit lower bound $%
V(r,\theta )\geq a^{\beta }/2$ for $r\geq a$; non-explosion follows from the
growth bound $|\bar{b}(x)|\leq 1+|x|$ and Gronwall's inequality.

\item The discontinuous field $b_{2}^{\delta }$ (see (\ref{bd-2})) is likewise unbounded, of
growth $|b_{2}^{\delta }(z)|\leq C_{D}(1+|z|)$ uniformly for $0\leq \delta
\leq D$: bounded Borel cutoffs have unique strong solutions that agree up to
their common exit from the ball on which they coincide, and the Gronwall
estimate on the stopped path supremum, which is independent of the cutoff
radius, yields a unique non-explosive global solution after the cutoff
radius tends to infinity; tightness as $\varepsilon \downarrow 0$ then
follows from the uniform Lipschitz modulus of the drift integrals on bounded
ranges.

\item The one-dimensional field $b_{5}$ (see (\ref{b_5})) is continuous, locally bounded, and
of at most linear growth: one restricts it to $[-a,a]$, extends it to a
bounded drift, applies the bounded-drift theorem, and couples the original
and extended solutions up to the exit time, whose probability tends to zero
on time intervals on which the deterministic limit stays a positive distance
from the boundary.
\end{itemize}

Thus the common device in these finite-time constructions is bounded
localization together with a sublinear or linear growth bound.
The discontinuity of $b_2^\delta$ is handled by bounded Borel drifts
and pathwise uniqueness for nondegenerate additive noise, not by
smoothness. Remark \ref{rem:unbounded-sublinear} specifies the
finite-time conclusions that this method preserves and explains
why it does not preserve every infinite-time conclusion.

\begin{example}
Consider%
\begin{equation}
\bar{b}(r,\theta )=%
\begin{cases}
r^{\beta }\left( \cos ^{3}\theta ,\sin ^{3}\theta \right) , & r>0,\ 1>\beta
>0, \\
0, & r=0.%
\end{cases}
\label{bar}
\end{equation}%
This field is unbounded as $r\to\infty$, so it is not an example of
(A1). Its smooth angular calculations and the independent stochastic
selection proof below use its own sublinear growth, not a direct
application of Theorem \ref{the1} or Proposition \ref{prop:select}.
For this field, the radial potential is $V(r,\theta )=r^{\beta }\left( \cos
^{4}\theta +\sin ^{4}\theta \right) ,$and the angular part $\Phi (\theta
)=\cos ^{4}\theta +\sin ^{4}\theta $ is strictly positive for all $\theta $,
satisfying the non-negativity condition $V(r,\theta )>0$. The radial
derivative is $\frac{\partial V}{\partial r}=\beta r^{\beta -1}\Phi (\theta
)>0,$which confirms the monotonicity of V in the radial direction. On $%
\mathbb{S}^{1}$, the function $\Phi (\theta )$ has four local maxima at the
cardinal directions $\theta =0,\frac{\pi }{2},\pi ,\frac{3\pi }{2}$, where $%
\Phi =1$, and four local minima at the diagonals $\theta =\frac{\pi }{4},%
\frac{3\pi }{4},\frac{5\pi }{4},\frac{7\pi }{4}$, where $\Phi =\frac{1}{2}$.
The alignment can be computed without confusing $Db$ with its transpose.
Put $c=\cos\theta$, $s=\sin\theta$, $\tau=(-s,c)$ and
$K=cs(s^2-c^2)=-\tfrac14\sin(4\theta)$. Direct projection gives
$B=r^{\beta-1}K\tau$. On the unit circle,
\[
 \nabla_{\mathbb S}V=\partial_\theta V\,\tau
                    =4r^\beta K\tau.
\]
Equation \eqref{qc} consequently gives
\begin{equation}
 A=r^{-1}\nabla_{\mathbb S}V-B=3B,
 \qquad r>0,\quad\omega\in\mathbb S^1. \label{bbAB}
\end{equation}
Thus (A4) holds with $\gamma=3$, independently of $\beta$;
$G=r^{-1}\nabla_{\mathbb S}V=4B$.
The identity $Db(r\omega)\omega=\partial_r b(r\omega)$ does not
identify $(Db(r\omega))^\top\omega$.
\begin{figure}[]
\centering          % 图片居中
\includegraphics[width=0.8\textwidth]{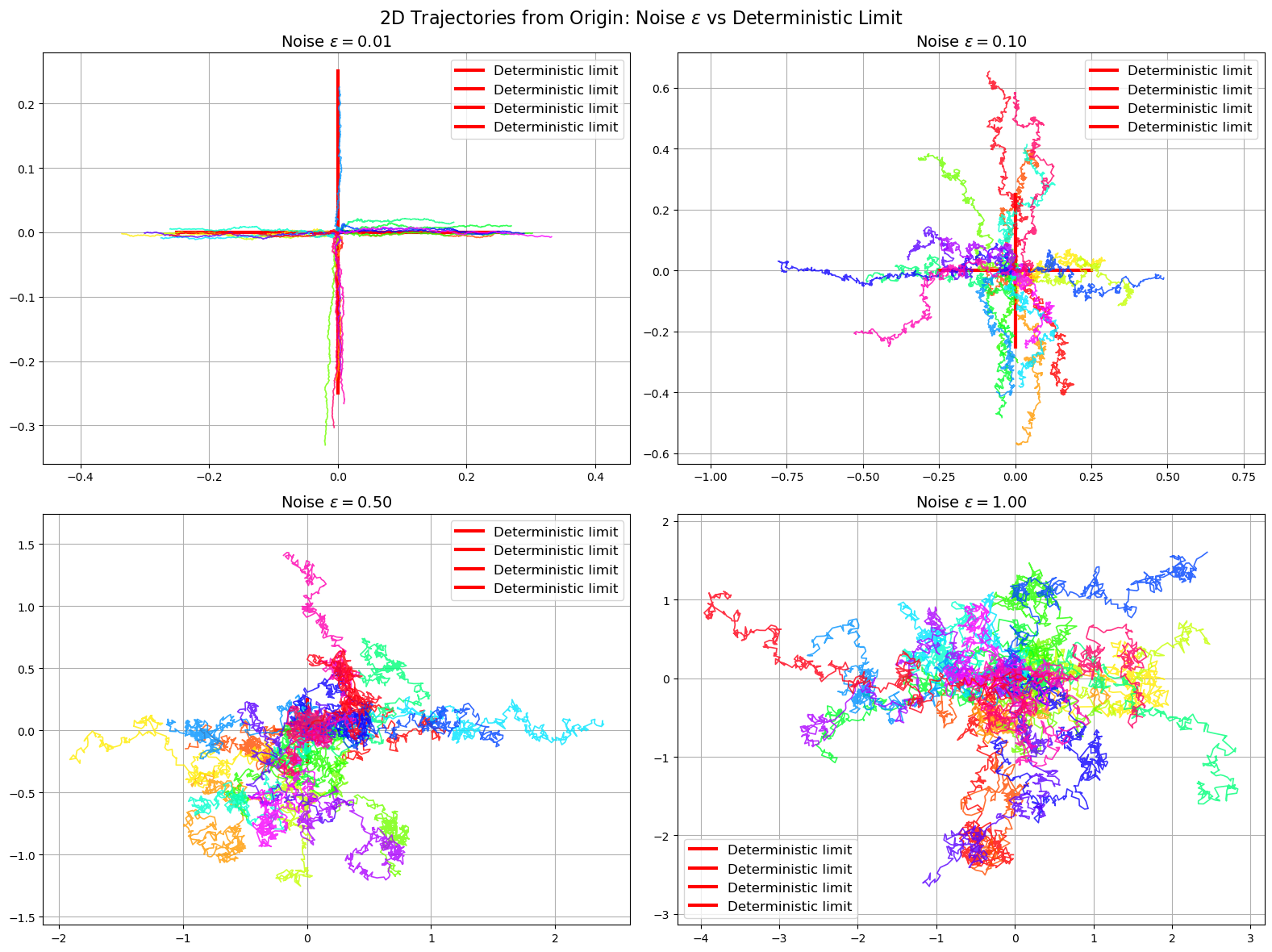} % 宽度 = 页面宽度的80%
\caption{Random perturbation of $2$-dim SDE with drift $\bar{b}$ with $%
\protect\beta =0.5$}
\label{fig:b_r}
\end{figure}
The deterministic angular equation, with the radius as parameter, is
\[
 \frac{\mathrm{d}\theta}{\mathrm{d}\log r}
 =-\frac{\sin4\theta}{3+\cos4\theta}.
\]
Here $\dot r=r^\beta\Phi(\theta)$ lies between
$r^\beta/2$ and $r^\beta$, so every positive-radius trajectory exists
for all future times and $r\to\infty$. In each open quadrant, an
initial angle below its diagonal decreases to the lower boundary,
and one above its diagonal increases to the upper boundary. To see
the stated limits, use the displayed strict sign and monotonicity;
a limit in the interior away from the diagonal would have nonzero
limiting derivative in $\log r$, a contradiction. Uniqueness prevents
crossing an angular equilibrium. The diagonal and boundary directions
themselves are stationary angular trajectories. In particular,
\[
 x(t)=\left(\frac{1-\beta}{2}t\right)^{1/(1-\beta)}
                       \frac{(1,1)}{\sqrt2}
\]
is an instantaneous-escape solution at a potential minimum, not a
maximum. The maximal directions therefore do not exhaust deterministic
instantaneous-escape directions. The following independent argument identifies the stochastic selection;
it does not apply a bounded-drift theorem to the unbounded field.

\noindent\emph{Finite-time zero-noise selection for the cubic field.}
Put $\alpha=1-\beta$ and $q=1/\alpha$. On every finite interval,
\begin{equation}
 \mathcal L(X^\varepsilon|_{[0,T]})\Longrightarrow
 \frac14\sum_{e\in\{e_1,-e_1,e_2,-e_2\}}
       \delta_{(\alpha^q t^q e)_{0\le t\le T}}.
 \label{eq:cubic-selection}
\end{equation}
Here the original, unbounded drift $\bar b$ is used without truncation
in the conclusion.

To prove this, bounded Borel cutoffs first give strong solutions up to
exit from any ball. They agree up to their common exit times by pathwise
uniqueness. The growth bound $|\bar b(x)|\le 1+|x|$ and Gronwall's
inequality give non-explosion and hence a unique global strong solution.
Work first at noise strength one and denote it by $Z$. Its radius and
angle will be written $R$ and $\theta$. Local Girsanov equivalence and
point polarity in dimension two give $R_t>0$ for all $t>0$ almost surely.
The escape argument in the proof of Lemma \ref{lem:mart} remains valid
with the drift bound on each fixed ball in place of the global bound:
hitting every larger ball only uses that local bound, while the return
estimate outside radius $a$ uses $V(r,\theta)\ge a^\beta/2$ for $r\ge a$.
Non-explosion ensures the hitting times of increasing radii tend to
infinity. Consequently $R_t\to\infty$ and $R_t\ge c t$ eventually,
for a positive constant $c$ on a probability-one event. In particular
$\int_1^\infty R_t^{-2}\mathrm{d}t<\infty$.

It\^{o}'s formula on positive-time intervals gives
\[
 \mathrm{d}R_t^\alpha=\alpha\Phi(\theta_t)\mathrm{d}t
     +\frac{\alpha^2}{2}R_t^{-1-\beta}\mathrm{d}t
     +\alpha R_t^{-\beta}\mathrm{d}B_t,
 \qquad \Phi(\theta)=\cos^4\theta+\sin^4\theta.
\]
The second term is integrable at infinity, since $R_t\ge ct$ and
$\beta>0$. The martingale divided by $t$ tends to zero almost surely:
its quadratic variation on $[1,t]$ is eventually at most a constant
times $t$, so the Brownian time-change strong law proves this assertion
(the case of finite total quadratic variation is immediate).
Since $1/2\le\Phi\le1$, it follows that
\[
 \limsup_{t\to\infty}R_t^\alpha/t\le\alpha,
 \qquad R_t=O(t^q).
\]
The angular clock $s(t)=\int_1^tR_u^{\beta-1}\mathrm{d}u$ therefore diverges.
The complete angular It\^{o} equation in this clock has drift
$\tfrac14\nabla_{\mathbb S}\Phi$ and a convergent perturbation consisting
of its angular martingale and its absolutely integrable curvature drift.
It\^{o}'s formula for $\Phi$ yields the nonnegative energy
$\tfrac14\int_1^\infty R_t^{\beta-1}|\nabla_{\mathbb S}\Phi|^2\mathrm{d}t$;
it is finite, since $\Phi$ is bounded, its martingale has finite total
quadratic variation, and its Laplacian correction is integrable.
The first assertion of Lemma \ref{lem:det-select} now implies that
$\theta_t$ converges to an axis or a diagonal.

If the limit is a diagonal, Cesaro averaging of the preceding formula
for $R^\alpha$ gives $R_t^\alpha/t\to\alpha/2$. Let
$\sigma\in\{-1,1\}$ be the sign of the product of the coordinates of
that diagonal, and put $Y_t=Z_{1,t}-\sigma Z_{2,t}$. Direct factorization
of the cubic drift gives
\[
 \mathrm{d}Y_t=k_tY_t\mathrm{d}t+\sqrt2\,\mathrm{d}\widetilde B_t,
 \qquad
 k_t=R_t^{\beta-1}
       (\theta_{1,t}^2+\sigma\theta_{1,t}\theta_{2,t}+\theta_{2,t}^2).
\]
On the event of convergence to this diagonal, $tk_t\to3/\alpha=3q$
and $|Y_t|\le\sqrt2R_t=O(t^q)$. The expanding-coordinate comparison
\eqref{eq:expanding-mode}, with $p=2q>q\ge1$, shows that this event
has probability zero. The four diagonal events form a finite union.
Thus $\theta_t\to\Theta_\infty\in\{\pm e_1,\pm e_2\}$ almost surely.
Applying the same Cesaro argument at an axis gives
\[
 Z_t/t^q\longrightarrow \alpha^q\Theta_\infty
                  \quad\hbox{almost surely}.
\]
Reflections and interchange of the coordinates preserve the drift and
Brownian law. Uniqueness in law therefore makes $\Theta_\infty$ uniform
on these four axes.

Finally set
$a_\varepsilon=\varepsilon^{2/(1+\beta)}$ and
$b_\varepsilon=a_\varepsilon^{1-\beta}$, so that
$a_\varepsilon=b_\varepsilon^q$.
Homogeneity of the drift and Brownian scaling give
$X_t^\varepsilon\stackrel{\mathcal L}=a_\varepsilon Z_{t/b_\varepsilon}$.
The preceding almost sure asymptotic implies uniform convergence of
these rescaled paths on $[0,T]$ to $\alpha^q t^q\Theta_\infty$.
Indeed, for any $\eta>0$, choose a finite random $S$ such that
$|Z_s-\alpha^q s^q\Theta_\infty|\le\eta s^q$ for all $s\ge S$.
On $t\le b_\varepsilon S$ the rescaled error tends to zero by continuity
on $[0,S]$; on $b_\varepsilon S\le t\le T$ it is at most $\eta T^q$.
Let $\varepsilon\downarrow0$ and then $\eta\downarrow0$. This proves
\eqref{eq:cubic-selection} as a statement about path laws, not random
branch convergence on the original Brownian space. The deterministic
diagonal solutions described above receive zero mass in this limit.
\end{example}

\begin{example}
\label{exa2d} Consider the perturbed vector field
\begin{equation}
b_{2}^{\delta }=2%
\begin{pmatrix}
\mathrm{sgn}(x)\sqrt{|y|} \\
\mathrm{sgn}(y)\sqrt{|x|}%
\end{pmatrix}%
+\delta (\mathrm{sgn}x,\mathrm{sgn}y)^{T},\text{ }\delta >0.  \label{bd-2}
\end{equation}%
The function $b_{2}^{\delta }$ is discontinuous at the origin $(0,0)$ and
along the entire coordinate axes, i.e., at all points where $x=0$ or $y=0$,
but is continuous everywhere else, specifically on the open set $\mathbb{R}%
^{2}\setminus \{(x,y)\mid x=0\text{ or }y=0\}$, i.e., the punctured plane
excluding both axes.

\noindent\emph{\textbf{Fillipov set-valued mapping.}} Given $\delta >0$, we derive the corresponding Filippov set-valued mapping $%
F[\cdot ]$ on $\mathbb{R}^{2}$ by case analysis: 1. Origin $z=(0,0)$
\begin{equation}
F_{b_{\delta }^{2}}(0,0)=\big\{(a,b)\in \mathbb{R}^{2}\;\big|\;|a|\leq
\delta ,\;|b|\leq \delta \big\}.  \label{F0}
\end{equation}%
This is a closed square convex set with Hausdorff dimension $2$. 2.
Non-origin points on vertical axis: $x=0,\;y\neq 0.$ Let $C_{y}=2\sqrt{|y|}%
+\delta >0$.
\begin{equation}
F_{b_{\delta }^{2}}(0,y)=\big\{\big(sC_{y},\;\delta \sgn(y)\big)\;\big|%
\;s\in \lbrack -1,1]\big\}.  \label{F1}
\end{equation}%
This set is a horizontal line segment of Hausdorff dimension $1$. 3.
Non-origin points on horizontal axis: $y=0,\;x\neq 0.$ Let $C_{x}=2\sqrt{|x|}%
+\delta >0$.
\begin{equation}
F_{b_{\delta }^{2}}(x,0)=\big\{\big(\delta \sgn(x),\;sC_{x}\big)\;\big|%
\;s\in \lbrack -1,1]\big\}.  \label{F2}
\end{equation}%
This set is a vertical line segment of Hausdorff dimension $1$. 4. Interior
of open quadrants: $x\neq 0,\;y\neq 0.$ The vector field $b_{\delta }$ is
continuous in this region, so the Filippov set collapses to a singleton
\begin{equation}
F_{b_{\delta }^{2}}(x,y)=\big\{b_{\delta }(x,y)\big\}.  \label{F3}
\end{equation}

\noindent\emph{\textbf{Filippov solutions.}} We now consider the Filippov DI
associated with drift $b_{\delta }$, namely
\begin{equation}
\dot{z}(t)\in F_{b_{\delta }^{2}}\big(z(t)\big),\quad z(0)=(0,0),\quad
z=(x,y),  \label{DIb2}
\end{equation}%
where $F_{b_{\delta }}$ denotes the Filippov set-valued mapping. %

We give some concrete immediate leaving solutions as follows:

\begin{itemize}
\item Clearly, $X(t)\equiv (\delta t,0)$ (see (\ref{F2})) is a Fillipov
solution to DI (\ref{DIb2}), etc.

\item We will choose the admissible instantaneous direction $(\delta ,\delta
),(-\delta ,-\delta ),(\delta ,-\delta ),(-\delta ,\delta )\in
F_{b_{2}^{\delta }}(0,0)$ (see (\ref{F0})). We first consider the DI in the
first quadrant after leaving the origin. In the smooth region $x>0,y>0$, the
vector field $b_{2}^{\delta }$ is continuous, so the DI (\ref{DIb2}) reduces
to the singleton $F_{b_{2}^{\delta }}(x,y)=\{b_{2}^{\delta }(x,y)\}.$ The DI
therefore becomes the classical ODE system
\begin{equation}
\begin{cases}
\left( x^{\delta }\right) ^{\prime }(t)=2\sqrt{y^{\delta }(t)}+\delta , \\%
[3pt]
\left( y^{\delta }\right) ^{\prime }(t)=2\sqrt{x^{\delta }(t)}+\delta , \\%
[3pt]
x^{\delta }(0)=y^{\delta }(0)=0.%
\end{cases}
\label{d1}
\end{equation}%
The symmetric solution $x_{+}^{\delta }(t)=y_{+}^{\delta }(t)$ satisfies,
with $u(t)=\sqrt{x_{+}^{\delta }(t)}$, the implicit form
\begin{equation*}
\sqrt{x_{+}^{\delta }(t)}-\frac{\delta }{2}\ln \left( \frac{2\sqrt{%
x_{+}^{\delta }(t)}+\delta }{\delta }\right) =t.
\end{equation*}%
For all $t>0$, $x_{+}^{\delta }(t)=y_{+}^{\delta }(t)>0$, so the solution
leaves the origin immediately. Second Quadrant: The symmetric solution $%
x_{-}^{\delta }(t)=-y_{+}^{\delta }(t)$ satisfies, with $u(t)=\sqrt{%
y_{+}^{\delta }(t)}$, the implicit form
\begin{equation*}
\sqrt{y_{+}^{\delta }(t)}-\frac{\delta }{2}\ln \left( \frac{2\sqrt{%
y_{+}^{\delta }(t)}+\delta }{\delta }\right) =t.
\end{equation*}

Third Quadrant: The symmetric solution $x_{-}^{\delta }(t)=y_{-}^{\delta
}(t)<0$ satisfies, with $u(t)=\sqrt{-x_{-}^{\delta }(t)}$, the implicit form
\begin{equation*}
\sqrt{-x_{-}^{\delta }(t)}-\frac{\delta }{2}\ln \left( \frac{2\sqrt{%
-x_{-}^{\delta }(t)}+\delta }{\delta }\right) =t.
\end{equation*}
Fourth Quadrant: The symmetric solution $y_{-}^{\delta }(t)=-x_{+}^{\delta
}(t)$ satisfies, with $u(t)=\sqrt{x_{+}^{\delta }(t)}$, the implicit form
\begin{equation*}
\sqrt{x_{+}^{\delta }(t)}-\frac{\delta }{2}\ln \left( \frac{2\sqrt{%
x_{+}^{\delta }(t)}+\delta }{\delta }\right) =t.
\end{equation*}
\end{itemize}

\noindent \emph{\textbf{Angular analysis.}} For the drift field $%
b_{2}^{\delta }$, at any point inside the four quadrants with $x\neq 0,y\neq
0$, the tangential vectors $A(r,\omega )=\mathbb{P}_{\omega }^{\perp
}(\nabla b_{\delta }^{2}(r\omega ))^{\top }\omega $ and $B(r,\omega )=\frac{1%
}{r}\mathbb{P}_{\omega }^{\perp }b_{\delta }^{2}(r\omega )$ share the same
tangential direction on $\mathbb{S}^{1}$. Indeed, let $\omega =(\cos \theta
,\sin \theta )^{\top }$, and let the unit counterclockwise tangent vector be
$\tau =(-\sin \theta ,\cos \theta )^{\top }$. Since orthogonal projection
onto the tangent space preserves tangential inner products, we have
\begin{equation*}
\langle A,\tau \rangle =\big\langle(\nabla b_{\delta }^{2})^{\top }\omega
,\tau \big\rangle,\qquad \langle B,\tau \rangle =\frac{1}{r}\langle
b_{\delta }^{2},\tau \rangle .
\end{equation*}%
Using the identity $\cos \theta =\emph{sgn}(\cos \theta )|\cos \theta |$, $%
\sin \theta =\emph{sgn}(\sin \theta )|\sin \theta |$, we compute
\begin{equation*}
\begin{aligned} \langle b,\tau\rangle
&=\operatorname{sgn}(\cos\theta)\operatorname{sgn}(\sin\theta)
\Big[|\cos\theta|\big(2\sqrt{r|\cos\theta|}+\delta\big)-|\sin\theta|\big(2%
\sqrt{r|\sin\theta|}+\delta\big)\Big]. \end{aligned}
\end{equation*}%
The map $t\mapsto 2\sqrt{r}t^{3/2}+\delta t$ is strictly increasing for $t>0$%
, hence
\begin{equation*}
\emph{sgn}\langle B,\tau \rangle =\emph{sgn}\big(|\cos \theta |-|\sin \theta
|\big)\cdot \emph{sgn}(\cos \theta )\emph{sgn}(\sin \theta ).
\end{equation*}%
Next, from the Jacobian matrix
\begin{equation*}
\nabla b(r\omega )=%
\begin{pmatrix}
0 & \displaystyle\frac{\emph{sgn}(\cos \theta )\emph{sgn}(\sin \theta )}{%
\sqrt{r|\sin \theta |}} \\
\displaystyle\frac{\emph{sgn}(\cos \theta )\emph{sgn}(\sin \theta )}{\sqrt{%
r|\cos \theta |}} & 0%
\end{pmatrix}%
,
\end{equation*}%
we derive
\begin{equation*}
\langle A,\tau \rangle =\frac{\emph{sgn}(\cos \theta )\emph{sgn}(\sin \theta
)}{\sqrt{r}\sqrt{|\cos \theta ||\sin \theta |}}\big(|\cos \theta
|^{5/2}-|\sin \theta |^{5/2}\big).
\end{equation*}%
Since $t\mapsto t^{5/2}$ is strictly increasing on $t>0$,
\begin{equation*}
\emph{sgn}\big(|\cos \theta |^{5/2}-|\sin \theta |^{5/2}\big)=\emph{sgn}\big(%
|\cos \theta |-|\sin \theta |\big),
\end{equation*}%
which yields
\begin{equation*}
\emph{sgn}\langle A,\tau \rangle =\emph{sgn}\big(|\cos \theta |-|\sin \theta
|\big)\cdot \emph{sgn}(\cos \theta )\emph{sgn}(\sin \theta )=\emph{sgn}%
\langle B,\tau \rangle .
\end{equation*}%
Therefore $A$ and $B$ have the same direction inside each open
quadrant, and both vanish on its diagonal. This is only a local
calculation: it does not supply the global smoothness or the global
(A4) hypothesis across the discontinuous coordinate axes.

The bounded-drift and global smoothness assumptions of Theorem
\ref{the1} and Proposition \ref{prop:select} do not apply to this field.
We now prove its finite-time stochastic selection directly; the proof
also covers $\delta=0$.

\noindent\emph{Radial thrust and restoring angular drift.}
In the first open quadrant set
\begin{equation}
 u(t)=y(t)/x(t). \label{ub2}
\end{equation}
The ordinary ODE gives
\begin{equation}
 \dot u=\frac{2x^{3/2}(1-u^{3/2})+\delta x(1-u)}{x^2}.
 \label{ub2p}
\end{equation}
Thus the drift of the ratio has the sign of $1-u$. This restoring sign
alone is not a finite-time zero-noise selection proof.
For $\omega=(1,k)/\sqrt{1+k^2}$, $k>0$, define
\begin{equation}
 V(r,\omega)=\langle b_2^\delta(r\omega),\omega\rangle. \label{vb2}
\end{equation}
Since $x=r/\sqrt{1+k^2}$, direct substitution gives
\begin{equation}
 V(r,k)=\frac{2\sqrt r(k+\sqrt k)}{(1+k^2)^{3/4}}
                  +\frac{\delta(1+k)}{(1+k^2)^{1/2}}. \label{Vrk}
\end{equation}
Its unique maximum within this quadrant is at $k=1$. To verify this
without relying on a stationary-point test, write
$A=\cos\theta$, $B=\sin\theta$, $s=A+B$ and $p=AB=(s^2-1)/2$.
For $A,B\ge0$ and $A^2+B^2=1$, $1\le s\le\sqrt2$, and
\[
 (A\sqrt B+B\sqrt A)^2=p(s+2\sqrt p).
\]
This is strictly increasing in $s$ on $(1,\sqrt2]$, as is the term
$\delta s$ when $\delta>0$. Hence the maximum occurs precisely at
$A=B=1/\sqrt2$, for every $r>0$ and $\delta\ge0$.
Independent coordinate reflections give the four diagonal maxima.
\begin{figure}[]
\centering
\includegraphics[width=0.6\textwidth]{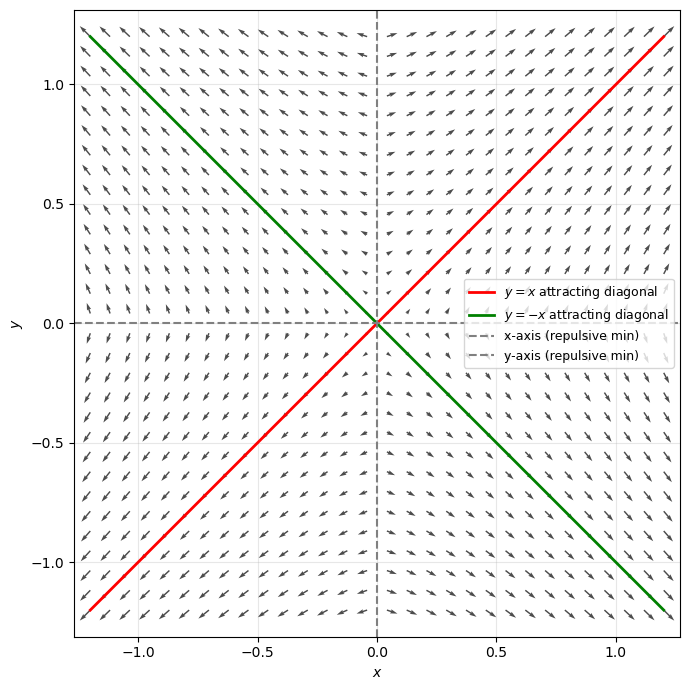}
\caption{Planar drift field $b_2^\delta$ with $\delta=0.5$.}
\label{fufankui}
\end{figure}

\noindent\emph{Finite-time selection.}
Let $X^{\varepsilon,\delta}$ solve the Cartesian SDE with this drift,
initial point zero, and diffusion $\varepsilon I_2$. For every $T<\infty$
and every fixed $\delta\ge0$,
\begin{equation}
 \mathcal L(X^{\varepsilon,\delta}|_{[0,T]})
 \Longrightarrow \frac14\sum_{\sigma_1,\sigma_2\in\{-1,1\}}
       \delta_{(\sigma_1q_\delta,\sigma_2q_\delta)|_{[0,T]}},
 \qquad \varepsilon\downarrow0, \label{eq:b2-selection}
\end{equation}
where $q_\delta'=2\sqrt{q_\delta}+\delta$, $q_\delta(0)=0$,
$q_\delta(t)>0$ for $t>0$, and $q_0(t)=t^2$.
For $\delta>0$ this is the function $x_+^\delta=y_+^\delta$ constructed
above. Moreover, if $\delta_\varepsilon\ge0$ and
$\delta_\varepsilon\to0$, the left side of \eqref{eq:b2-selection},
with $\delta=\delta_\varepsilon$, converges to its right side with
$q_0(t)=t^2$.

\noindent\emph{Proof.}
Bounded Borel cutoffs of $b_2^\delta$ have unique strong solutions.
They agree up to the first exit from the ball on which their drifts
coincide. For $0\le\delta\le D$ the growth estimate
$|b_2^\delta(z)|\le C_D(1+|z|)$ is uniform. The integral equation and
Gronwall's inequality bound the stopped path supremum on $[0,T]$ by
$C_{D,T}(1+\varepsilon\sup_{t\le T}|W_t|)$. Letting the cutoff radius
increase proves non-explosion, strong existence, and uniqueness. This
also proves tightness of $X^{\varepsilon,\delta}$ as
$\varepsilon\downarrow0$ with $\delta$ in a bounded set: on bounded
path ranges the drift is uniformly bounded, so its integrals have a
uniform Lipschitz modulus, while $\varepsilon W\to0$ uniformly in
probability.

Write $U=|X_1^{\varepsilon,\delta}|$ and
$V=|X_2^{\varepsilon,\delta}|$ in the rest of this proof; here $V$
is an absolute coordinate, not the radial potential.
Each coordinate has quadratic variation $\varepsilon^2dt$, so the
occupation-time formula implies zero Lebesgue time on its zero set.
Tanaka's formula therefore gives
\begin{align*}
 \mathrm{d}U&=(2\sqrt V+\delta)\mathrm{d}t+\varepsilon\,\mathrm{d}B_1+dL_1,\\
 \mathrm{d}V&=(2\sqrt U+\delta)\mathrm{d}t+\varepsilon\,\mathrm{d}B_2+dL_2.
\end{align*}
Here $B_i=\int\operatorname{sgn}(X_i)dW_i$ form a two-dimensional
Brownian motion: their quadratic variations are $t$ and their cross
variation is zero. Each $L_i$ is nondecreasing and increases only at
zero of the corresponding absolute coordinate.
The Skorokhod reflection formula, applied to the nondecreasing inputs
$A_1(t)=\int_0^t(2\sqrt V+\delta)ds$ and
$A_2(t)=\int_0^t(2\sqrt U+\delta)ds$, gives
\[
 L_i(t)=\sup_{s\le t}[-A_i(s)-\varepsilon B_i(s)]_+
       \le\varepsilon\sup_{s\le t}|B_i(s)|.
\]
Thus $\varepsilon B_i+L_i\to0$ uniformly in probability on $[0,T]$,
uniformly for $\delta$ in bounded sets. For any sequence
$\varepsilon_j\downarrow0$, $\delta_j\to\delta_*\ge0$, every cluster
of $(U,V)$ consequently solves
\begin{equation}
 u'=2\sqrt v+\delta_*,\qquad v'=2\sqrt u+\delta_* ,
 \qquad u(0)=v(0)=0. \label{eq:b2-absolute-limit}
\end{equation}
To justify this passage, use joint Skorokhod representation of the
absolute paths and the errors, uniform convergence on bounded ranges,
and continuity of the square-root function in the integral equations.

We must still exclude waiting when $\delta_*=0$. This is not a
consequence of \eqref{eq:b2-absolute-limit} alone.
Set $S=U+V$ and $h(s)=2\min\{\sqrt{s_+},1\}$.
Then
\[
 \mathrm{d}S=a_tdt+\sqrt2\varepsilon\,\mathrm{d}B+\mathrm{d}L_1+\mathrm{d}L_2,
 \quad a_t=2(\sqrt U+\sqrt V)+2\delta\ge2\sqrt S\ge h(S),
 \quad B=(B_1+B_2)/\sqrt2.
\]
Let $Z^\varepsilon$ solve
$\mathrm{d}Z^\varepsilon=h(Z^\varepsilon)\mathrm{d}t+\sqrt2\varepsilon\,\mathrm{d}B$,
$Z^\varepsilon_0=0$. We prove $S\ge Z^\varepsilon$ on this same
Brownian space. Put $\sigma_0=\sqrt2\varepsilon$ and
\[
 p(x)=\exp\left(-\frac{2}{\sigma_0^2}\int_0^x h(y)\mathrm{d}y\right),
 \qquad F(x)=\int_0^x p(y)\mathrm{d}y.
\]
The function $F$ is $C^2$ and strictly increasing. Ito's formula yields
\[
 \mathrm{d}F(S)=\sigma_0p(S)\mathrm{d}B+p(S)\mathrm{d}K,
 \qquad \mathrm{d}F(Z^\varepsilon)=\sigma_0p(Z^\varepsilon)\mathrm{d}B,
\]
where $K=\int(a_t-h(S_t))\mathrm{d}t+L_1+L_2$ is nondecreasing.
In the transformed coordinate the common diffusion coefficient
$\sigma_0p(F^{-1}(y))$ is locally Lipschitz; its derivative is
$-2h(F^{-1}(y))/\sigma_0$. Localize both processes to compact ranges.
For their difference, the occupation-time formula and the Lipschitz
bound imply zero local time at zero: the approximating local-time
integral is bounded by a constant times
$\eta^{-1}\int1_{\{|D_t|<\eta\}}|D_t|^2dt\le C\eta T$.
Tanaka's formula for $(F(Z^\varepsilon)-F(S))^+$ now has a martingale
term of mean zero and a nonpositive finite-variation term. It starts
at zero, so its expectation, and hence the positive part itself, is
zero. Remove the localization to obtain $Z^\varepsilon\le S$.

The drift $h$ is bounded and nonnegative on $\mathbb R$ and
$\int_0^1h(s)^{-1}ds=1$. Theorem 1.1 of \cite{fmp23} gives
$Z^\varepsilon\to\psi$ uniformly in probability, where
\[
 \psi(t)=\begin{cases}t^2,&0\le t\le1,\\2t-1,&t\ge1.\end{cases}
\]
This deterministic positive lower limit and the preceding comparison
imply $u(t)+v(t)\ge\psi(t)>0$ for every $t>0$ in each cluster law.
More explicitly, apply the closed-set Portmanteau argument to
$\{s:s(t)\ge\psi(t)-\eta\ \forall t\le T\}$ and intersect over
$\eta=1/n$.
For $\delta_*>0$, positivity follows as well from
$u,v\ge\delta_*t$. In either case, on every positive-time interval,
\[
 \frac{\mathrm{d}}{\mathrm{d}t}(u-v)^2
 =-\frac{4(u-v)^2}{\sqrt u+\sqrt v}\le0.
\]
Integrating from $\eta>0$ and letting $\eta\downarrow0$ gives $u=v$.
For $\delta_*>0$, separation of variables in
$q'=2\sqrt q+\delta_*$ gives its unique increasing solution. For
$\delta_*=0$, positivity for all $t>0$ gives
$(\sqrt q)'=1$ and hence $q=t^2$, without delay.

Thus every Cartesian cluster path has absolute coordinates equal to
$q_{\delta_*}>0$ for positive times. Continuity forces each coordinate
sign to be constant on $(0,T]$. The drift and isotropic Brownian law
are invariant under either coordinate reflection; uniqueness in law
therefore makes every prelimit path law invariant under these two
reflections. The same holds for every cluster law. The reflection
group permutes the four candidate paths transitively, so each has
weight $1/4$. Tightness and uniqueness of the cluster law prove
\eqref{eq:b2-selection} and the asserted joint limit. This argument is
in path-law topology, not convergence to a random sign on the original
Brownian space.

For clarity, the deterministic amplitudes also satisfy
\[
 t\le\sqrt{q_\delta(t)}\le t+\sqrt{\delta t}\quad(\delta>0).
\]
Indeed $q_\delta\ge\delta t$ and
$(\sqrt{q_\delta})'=1+\delta/(2\sqrt{q_\delta})$ on $t>0$;
integrating from zero gives both bounds. Thus $q_\delta\to t^2$
uniformly on finite intervals. In addition, for fixed $\varepsilon>0$,
$|b_2^\delta-b_2^0|\le\sqrt2\delta$. Finite-time Girsanov applied to
this bounded difference gives relative entropy at most
$\delta^2T/\varepsilon^2$ between the path laws, hence total variation
at most $\delta\sqrt{T}/\varepsilon$ by Pinsker's inequality. Novikov
holds for the difference even though the common drift is unbounded.
Consequently both iterated limits agree with the joint limit just
proved. This conclusion follows from the preceding estimates, not
from pointwise drift convergence alone.

\noindent\emph{Axes and action.}
Zero occupation time on the axes is not pathwise avoidance. On every
bounded ball, finite-time Girsanov equivalence transfers Brownian
coordinate zeros arbitrarily close to time zero to this diffusion.
Thus positive-noise paths do hit the axes. The axis solution
$(\delta t,0)$ and solutions which subsequently enter a quadrant are
valid deterministic instantaneous-escape solutions, but the absolute
coordinate argument above excludes them from the zero-noise law.
%No positive-action assertion for all off-diagonal ODE solutions, and
%no large-deviation principle for the discontinuous drift, is needed.
\begin{figure}[]
\centering
\includegraphics[width=0.8\textwidth]{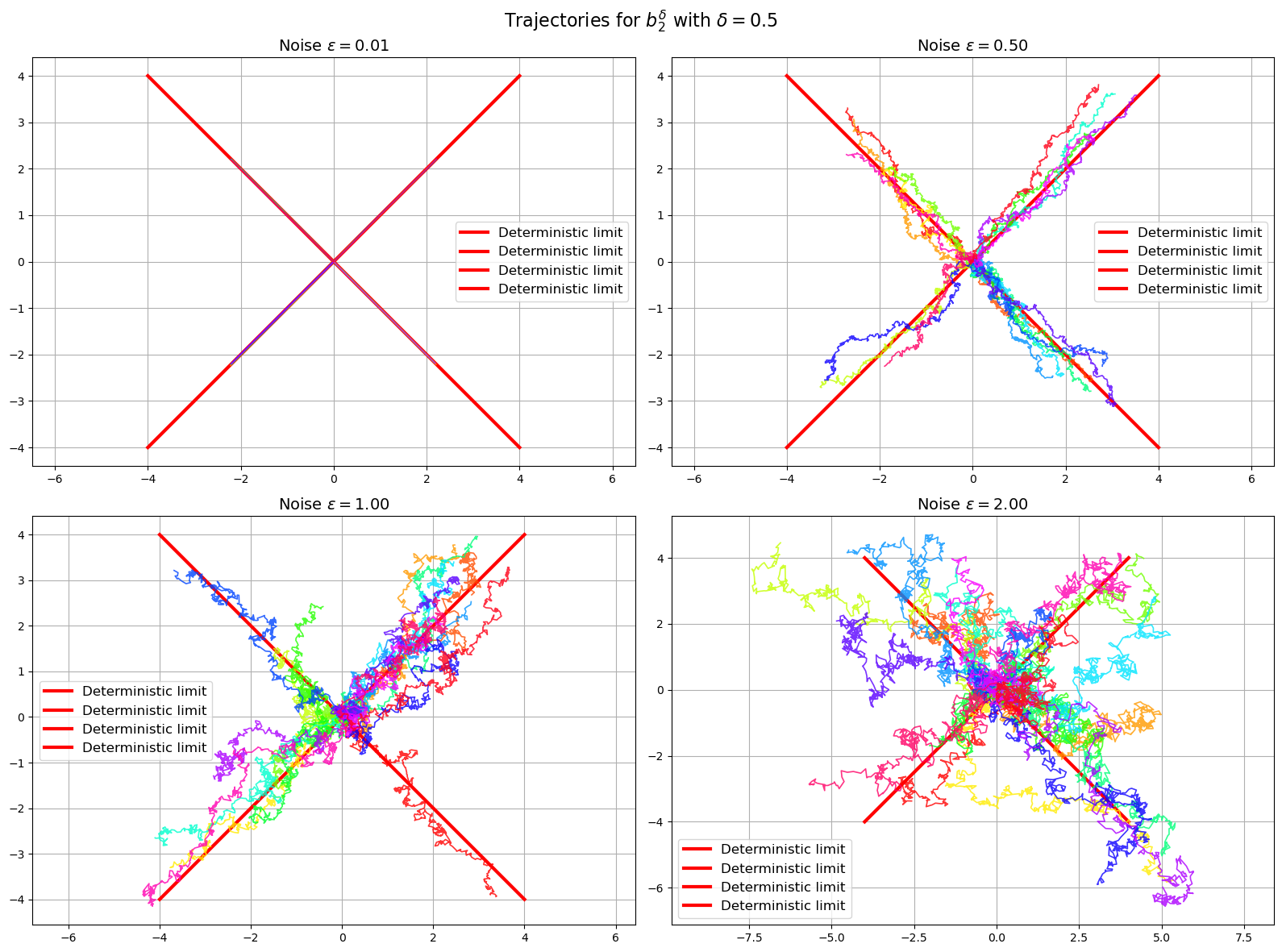}
\caption{Illustration of the perturbed drift $b_2^\delta$ for
$\delta=0.5$. The selection statement is proved in
\eqref{eq:b2-selection}, independently of the simulations.}
\label{fig:b2d}
\end{figure}
\end{example}

\begin{remark}
\label{rem:unbounded-sublinear}
\emph{Scope of localization.}
Let $b:\mathbb R^d\to\mathbb R^d$ be locally bounded and Borel, and
assume the explicit growth bound
\[
|b(x)|\le C(1+|x|),\qquad x\in\mathbb R^d,
\]
with a finite constant $C$. For every fixed $\varepsilon>0$ there
is a unique non-explosive global strong solution. On each finite
interval its laws are relatively compact as $\varepsilon\downarrow0$,
and every cluster law has the Filippov property in Proposition
\ref{pro1}. In dimension $d\ge2$, if all the other conditions in
(A1) hold, the no-waiting radial bound and the polar conclusions of
Theorem \ref{the1} also remain valid, with the original radial
profile $\varphi$. These finite-time assertions do not extend the
upper linear rate in Lemma \ref{lem:mart}, or by themselves the
infinite-time angular selection theorem.
Here is a bounded localization that proves these assertions and
preserves the radial hypotheses when they are present. For $R>0$ set
\[
\chi_R(r)=\begin{cases}
1,&0\le r\le R,\\
(1+R)/(1+r),&r>R,
\end{cases}
\qquad b^{(R)}(x)=\chi_R(|x|)b(x).
\]
This is a bounded Borel drift, with $|b^{(R)}|\le C(1+R)$, agreeing with $b$ on $B_R$ and still
satisfying $|b^{(R)}(x)|\le C(1+|x|)$. For fixed $R$ and positive
noise, bounded-drift strong existence and pathwise uniqueness apply
\cite[Theorems 1 and 2]{Ver1981}: the diffusion matrix is the
constant nondegenerate matrix $\varepsilon I_d$. Solutions for
different radii, driven by the same Brownian motion, agree up to
their common exit from the smaller ball. This stopped uniqueness
follows by continuing a stopped solution with the bounded equation
and using its global pathwise uniqueness.
For the stopped solutions started at $x$, write
\[
\tau_R=\inf\{t:|X_t|\ge R\},\qquad
S_R(t)=\sup_{u\le t}|X_{u\wedge\tau_R}|.
\]
The integral equation gives, for $t\le T$,
\[
S_R(t)\le |x|+CT+C\int_0^t S_R(u)\,\mathrm{d}u
+\varepsilon\sup_{u\le T}|W_u|.
\]
Gronwall's inequality therefore yields the bound independent of $R$,
\[
S_R(T)\le K_{x,T}+\varepsilon e^{CT}\sup_{u\le T}|W_u|,
\qquad K_{x,T}:=(|x|+CT)e^{CT}.
\]
For each Brownian path the right-hand side is finite. Consistency of
the stopped solutions and this bound imply that their exit times
tend to infinity, and hence give a global solution. Any two global
solutions agree up to every such exit, proving uniqueness. The same
bound holds for the global solution.
Fix $T$ and then choose a deterministic $R>K_{x,T}$. Coupling the
global solution with the bounded equation gives
\[
\mathbb P\{X^{x,\varepsilon}|_{[0,T]}
\ne X^{x,\varepsilon,(R)}|_{[0,T]}\}
\le\mathbb P(\tau_R\le T)
\le\mathbb P\left\{\varepsilon e^{CT}\sup_{u\le T}|W_u|
\ge R-K_{x,T}\right\}\longrightarrow0.
\]
Proposition \ref{pro1} for this fixed bounded drift proves tightness
and identifies every cluster as a Filippov cluster for $b^{(R)}$.
The path supremum in each cluster is at most $K_{x,T}$, by the
preceding bound and the closed-set Portmanteau theorem. Thus every
such path lies strictly inside $B_R$, where
$F_{b^{(R)}}(z)=F_b(z)$ by locality of the Filippov construction.
This proves the stated finite-time extension of Proposition
\ref{pro1}, including absolute continuity.
For the additional claim under the other hypotheses of (A1), start
at zero and choose its parameter $\delta_0>0$. If $R>2\delta_0$,
then $b^{(R)}(0)=0$,
$q^{(R)}(x)=\chi_R(|x|)q(x)$ is continuous, and
$\varphi^{(R)}(r)=\chi_R(r)\varphi(r)>0$. On
$0<r<\delta_0$, the sliding minimum defining $h^{(R)}$ is exactly
$h$, so the reciprocal integrability assumption is unchanged.
Thus $b^{(R)}$ satisfies (A1). Moreover,
\[
H(u):=\int_0^u\frac{\mathrm{d}r}{\varphi(r)}
\]
is finite for finite $u$, is strictly increasing, and tends to
infinity as $u\to\infty$: near zero use $h\le\varphi$, and at
infinity use $\varphi(r)\le C(1+r)$. Its inverse $\psi$ satisfies
$\psi'=\varphi(\psi)$ on positive times and $\psi(t)\le e^{Ct}-1\le K_{0,T}$ for $0\le t\le T$.
Choose $R>\max\{2\delta_0,K_{0,T}\}$. The bounded-drift comparison
profile $\psi^{(R)}$ agrees with $\psi$ throughout $[0,T]$, since
$\varphi^{(R)}=\varphi$ below $R$ and $\psi(T)<R$.
Theorem \ref{the1}, applied to $b^{(R)}$, and the same coupling
therefore give $|X_t^0|\ge\psi(t)>0$ simultaneously for all
$t\in(0,T]$ in every cluster law. The limiting paths stay inside
$B_R$, so the same Filippov velocity belongs to $F_b$ and supplies
both polar components as in that theorem. Taking integer horizons
and a diagonal subsequence gives the corresponding locally uniform
path-law statement. No uniqueness of the general cluster law or
selection of every immediate-departure solution is added.
The exit from each fixed ball also remains almost surely finite
under the local bound alone, once non-explosion is known. Indeed,
let $M_L=\sup_{|x|\le L}|b(x)|<\infty$ and
$E_j=\{\varepsilon|W_{j+1}-W_j|>2L+M_L\}$. Survival in $B_L$
through time $n$ implies $E_j^c$ for every $0\le j<n$, by the
integral equation on that interval and the reverse triangle
inequality. Independent Brownian increments consequently give
$\mathbb P(\tau_L>n)\le(1-\mathbb P(E_0))^n\to0$.
To see the obstruction at infinite time, take $d\ge2$, $0<\beta<1$,
and $b(x)=|x|^{\beta-1}x$ for $x\ne0$, with $b(0)=0$.
This field satisfies the displayed growth bound and all of (A1)
except global boundedness: $q(x)=|x|^{\beta+1}$ is continuous,
$\varphi(r)=h(r)=r^\beta>0$, and
$\int_0^{\delta_0}r^{-\beta}\,\mathrm{d}r<\infty$.
It also satisfies the smoothness away from zero and the radial
monotonicity in (A3). Localization of the positive-noise
non-hitting argument gives $R_t>0$ for every $t>0$ almost surely.
On positive-time intervals its radius satisfies
\[
\mathrm{d}R_t=
\left(R_t^\beta+\frac{(d-1)\varepsilon^2}{2R_t}\right)\mathrm{d}t
+\varepsilon\,\mathrm{d}B_t.
\]
For $a>0$, put $m=a^\beta$ and $h_a(r)=e^{-2mr/\varepsilon^2}$.
On $r\ge a$ its radial generator obeys
\[
\mathcal Lh_a(r)=h_a'(r)
\left(r^\beta+\frac{(d-1)\varepsilon^2}{2r}-m\right)\le0.
\]
Stop first in a finite annulus and at a finite time, and then let
the outer radius and time increase. Non-explosion and Fatou's
lemma give, when starting at $L>a$,
\[
\mathbb P_L(\tau_a<\infty)
\le\frac{h_a(L)}{h_a(a)}
=e^{-2a^\beta(L-a)/\varepsilon^2}.
\]
Every larger ball is exited almost surely by the preceding argument.
The strong Markov property at its first exit therefore bounds the
probability of visits to $[0,a]$ at arbitrarily large times by this
same quantity for every $L>a$. Letting $L\to\infty$ makes that
probability zero. Intersecting over positive integer $a$ proves
$R_t\to\infty$ almost surely. For every positive integer $K$, a
finite random time $T_K$ consequently exists after which
$R_t^\beta\ge K$. Integrating the radial equation pathwise from
that eventual-residence time gives
\[
R_t\ge Kt+\varepsilon B_t+C_K\quad(t\ge T_K),
\qquad C_K=R_{T_K}-KT_K-\varepsilon B_{T_K}\in\mathbb R.
\]
No stopping-time assertion about $T_K$ is needed for this pathwise
identity. Since $B_t/t\to0$, intersecting over integers $K$ yields
$R_t/t\to\infty$ almost surely. In particular, the finite upper
linear rate in Lemma \ref{lem:mart} cannot be retained under a
blanket replacement of global boundedness.
The growth bound for $b_2^\delta$ is uniform for $0\le\delta\le D$,
so its finite-time localization is covered here. When $\delta>0$
it also satisfies the remaining radial hypotheses, with
$\varphi=h=\delta$; when $\delta=0$ those strict radial hypotheses
fail on the axes. The cubic field has
$\varphi=h=r^\beta/2$ and is another nontrivial example covered by
the extended no-waiting result. Their four-branch selection laws,
and the zero-parameter and joint limits for $b_2^\delta$, still
use their separate proofs, not an unbounded version of Proposition
\ref{prop:select}.
\begin{table}[ht]
\centering
\renewcommand{\arraystretch}{1.5}
\begin{tabular}{|p{7.0cm}|c|p{4.6cm}|}
\hline
\textbf{Conclusion}
& \textbf{Globally bounded }$b$
& \textbf{Locally bounded with }$|b|\le C(1+|x|)$ \\
\hline
Strong existence, uniqueness and non-explosion for fixed $\varepsilon$
& $\checkmark$
& $\checkmark$ \\
\hline
Relative compactness on finite intervals; cluster laws have the
Filippov property (Proposition~\ref{pro1})
& $\checkmark$
& $\checkmark$ (by localization) \\
\hline
(A1) no-waiting radial bound and polarity at the origin
(Theorem~\ref{the1}, $d\ge2$)
& $\checkmark$
& $\checkmark$ (finite time / locally uniform) \\
\hline
Exit of every fixed ball almost surely
& $\checkmark$
& $\checkmark$ (local bound plus non-explosion suffice) \\
\hline
Upper linear rate of Lemma~\ref{lem:mart},
$\limsup_{t\to\infty}R_t/t<\infty$
& $\checkmark$
& $\times$ (counterexample gives $R_t/t\to\infty$) \\
\hline
Infinite-time angular selection of Proposition~\ref{prop:select}
& $\checkmark$
& not automatic; each example is proved separately \\
\hline
\end{tabular}
\caption{Scope of localization: finite-time conclusions survive under
local boundedness with linear growth, whereas the infinite-time linear
upper rate and angular selection do not follow by a blanket replacement
of global boundedness.}
\label{tab:localization-scope}
\end{table}
\end{remark}

\begin{example}
\label{exas} Consider the one-dimensional drift
\begin{equation}
b_{5}(x)=%
\begin{cases}
-x^{\alpha }\log x, & x>0, \\
0, & x=0, \\
-|x|^{\beta }, & x<0,%
\end{cases}%
\qquad 0<\alpha <\beta <1.  \label{b_5}
\end{equation}
This explicit parameter restriction is used below. No claim for
$\alpha\le0$ is made. The drift is continuous and locally bounded and
has at most linear growth; bounded cutoffs, pathwise uniqueness for
nonzero additive noise, and a Gronwall bound therefore give a unique
non-explosive strong solution.
For a sufficiently small fixed $a\in(0,e^{-1})$, use a bounded outward
extension of its restriction to $[-a,a]$. The reciprocal of its drift
is integrable on either side of zero. With the notation of Theorem 1.2
of \cite{fmp23}, write
\[
 I_+(\varepsilon)=\int_0^a
    \exp\left[-\frac2{\varepsilon^2}\int_0^x(-u^\alpha\log u)\mathrm{d}u\right]\mathrm{d}x,
 \qquad
 I_-(\varepsilon)=\int_0^a
    \exp\left[-\frac{2x^{\beta+1}}{(\beta+1)\varepsilon^2}\right]\mathrm{d}x.
\]
Since $-\log u\ge1$ on $(0,a)$, a change of variables gives
$I_+\le C_\alpha\varepsilon^{2/(\alpha+1)}$.
Integration over $0\le x\le\varepsilon^{2/(\beta+1)}$ gives
$I_-\ge c_\beta\varepsilon^{2/(\beta+1)}$ for small $\varepsilon$.
Hence $I_+/I_-\to0$, because $\alpha<\beta$. The scale-function weight
of the positive branch is $I_-/(I_-+I_+)\to1$. The cited theorem yields
weak convergence to the deterministic positive extremal solution
\[
 \psi_+(t)=H_+^{-1}(t),\qquad
 H_+(x)=\int_0^x\frac{\mathrm{d}u}{-u^\alpha\log u},\quad 0\le x\le a.
\]
On any fixed $[0,T]$ with $T<H_+(a)$ this also gives convergence
uniformly in probability, since the limit is deterministic. The limit
stays a positive distance from the boundary of $[-a,a]$ on $[0,T]$,
so the probability that the bounded-extension diffusion exits that
interval by $T$ tends to zero. Coupling the original and extended
solutions up to this exit transfers the result to the original drift.
Thus only the upper extremal solution is selected on such small time
intervals. This proof uses the stated bounded-drift theorem after
localization; it does not rely on an unverified parameter version of
Lemma 4.2 in \cite{BB}. Figure \ref{rpb_u} is an illustration for
parameters within the stated range.

The omitted parameter range cannot simply be added to this bounded
localization argument. There is even a direct obstruction to the
unqualified claim when singular entrance solutions of the usual SDE
are admitted. Set $\alpha=-2$, $\beta=-1$ and still assign $b_5(0)=0$.
For an integer $n\ge3$, let $\varepsilon_n=\sqrt{2/(n-1)}$ and let
$Y^n=\varepsilon_n BES_n$ be driven by $-W$ in the sense of
Definition \ref{bes}. Then $Y_t^n>0$ for $t>0$ and
\[
 \mathrm{d}Y_t^n=(Y_t^n)^{-1}\mathrm{d}t-\varepsilon_n \mathrm{d}W_t.
\]
Thus $X^n=-Y^n$ is a strong solution of the original Cartesian SDE
with this $b_5$, since it remains in the negative half-line and
$b_5(X_t^n)=-(Y_t^n)^{-1}$. Its drift integral at zero is finite,
as follows from the Bessel representation.
It\^{o}'s formula gives
$(Y_t^n)^2=(2+\varepsilon_n^2)t+M_t^n$, where
$M_t^n=-2\varepsilon_n\int_0^tY_s^n \mathrm{d}W_s$ and
$\mathbb E(Y_s^n)^2=(2+\varepsilon_n^2)s$. Doob's inequality yields
\[
 \mathbb E\sup_{t\le T}|M_t^n|^2
 \le8\varepsilon_n^2(2+\varepsilon_n^2)T^2\longrightarrow0.
\]
Consequently $X^n\to-\sqrt{2t}$ uniformly in probability on $[0,T]$,
not to the upper branch. This example lies outside the locally bounded
drift class. It shows why stating only $\alpha<\beta<1$, without a
solution-class restriction, cannot guarantee the asserted upper-branch
selection. The proved positive-exponent statement above does not claim
that its parameter restriction is optimal.
\begin{figure}[]
\centering          % 图片居中
\includegraphics[width=0.8\textwidth]{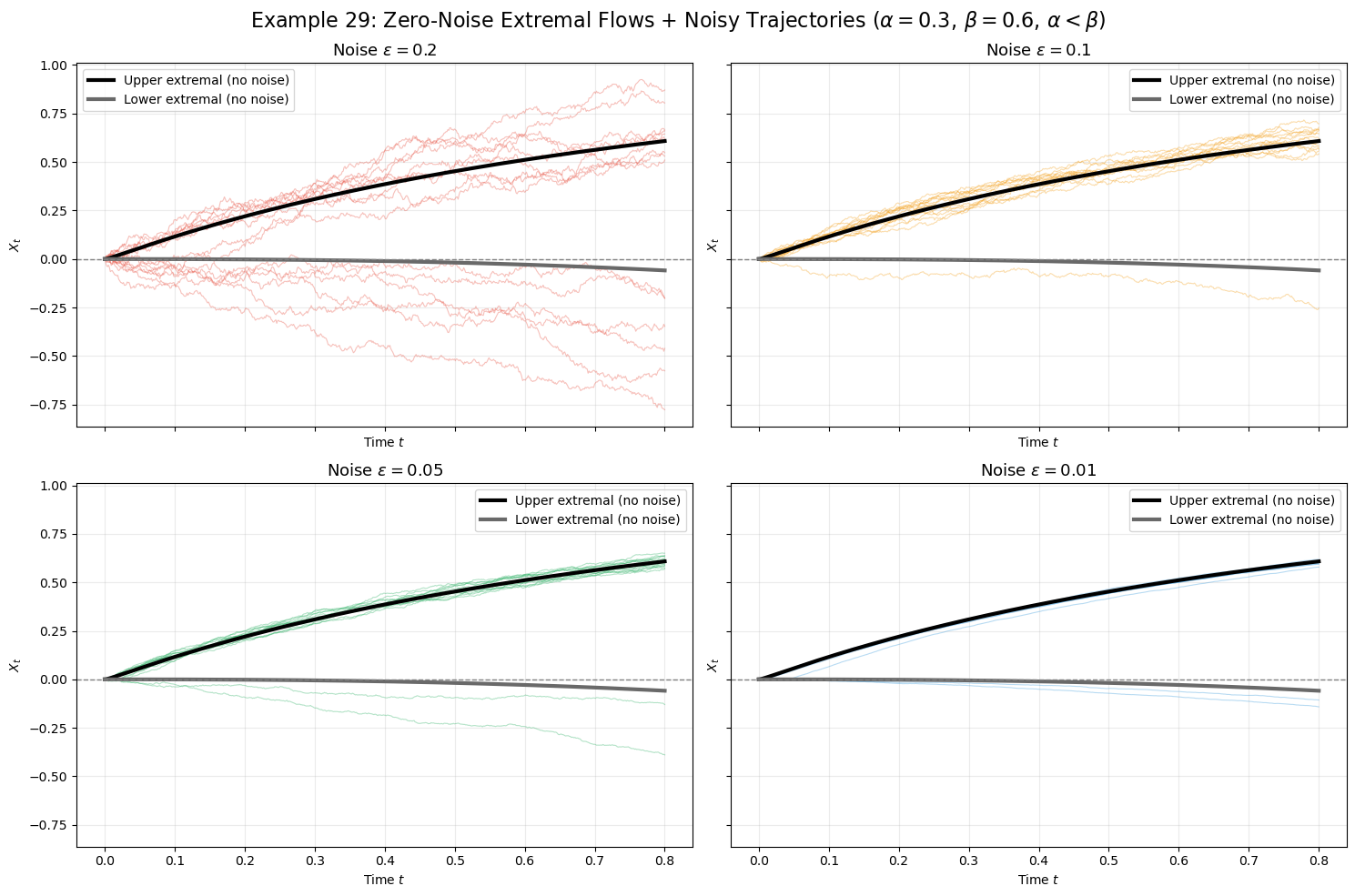} % 宽度 = 页面宽度的80%
\caption{Random perturbation of $1$-dim SDE with drift $b_{5}$}
\label{rpb_u}
\end{figure}
\end{example}
\begin{remark}
\label{rem:immediate-not-sufficient}
The third example shows that the support of the zero-noise limit is carried by immediate
departure solutions, but not every immediate departure solution is selected. On both sides of
the origin the drift points outward and its reciprocal is integrable, so extremal solutions
leave the origin instantly on either side, and both the positive and negative branches are
therefore immediate departure solutions. Nevertheless, for $0<\alpha<\beta<1$, the
scale-function weights are asymmetric: $I_+\le C_\alpha\varepsilon^{2/(\alpha+1)}$ and
$I_-\ge c_\beta\varepsilon^{2/(\beta+1)}$, with $I_+/I_-\to0$, so the positive branch receives
weight $I_-/(I_-+I_+)\to1$ while the negative immediate departure branch has zero limiting
mass. Membership in the set of immediate departure solutions only identifies the candidates;
the small-noise weighting, through the scale function, selects a strict subset of them. The
same distinction appears in the other two examples: the deterministic diagonal escape orbit
of the cubic field is an immediate departure solution at a potential minimum, yet the
zero-noise law is concentrated on the four axes; and the axis solution $(\delta t,0)$ of
$b_2^\delta$ departs immediately but is excluded by the absolute-coordinate comparison, which
concentrates the law on the four diagonals. The Bessel construction with $\alpha=-2$,
$\beta=-1$ sharpens the point: once singular entrance solutions are admitted outside the
locally bounded class, even the identity of the selected branch can reverse. Hence the
assertion that the limit is supported on immediate departure solutions must always be
accompanied by the explicit selection mechanism and the stated solution-class and parameter
restrictions; it cannot be read as a claim that every such solution occurs with positive
weight.
\end{remark}

\subsection{Zero-Noise Limiting Measure $\protect\mu ^{0}$}

\label{3.5}

%Let $X_{t}^{\varepsilon }$ denote the solution to the high-dimensional SDE (%
%\ref{2.1}). Let $\mu _{t}^{\varepsilon }=\mathcal{L}(X_{t}^{\varepsilon })$
%stands for the marginal distribution of $X_{t}^{\varepsilon }$ at fixed time
%$t>0$. As the noise intensity vanishes $\varepsilon \rightarrow 0^{+}$,
%there exists a subsequence $\{\mu _{t}^{\varepsilon _{n}}\}_{\varepsilon
%_{n}>0}$ admits a weak limit $\mu _{t}^{0}$, denoted collectively as $\mu
%^{0}$. We systematically characterize all analytical, geometric, topological
%properties of $\mu ^{0}$ based on the main theorems, polar decomposition
%arguments, stochastic comparison principles and numerical evidence in this
%paper.
%
%The support of the limiting measure is the core geometric object encoding
%the noise-induced selection rule among multiple Filippov solutions.

Fix $T>0$ and let $\mathbf{P}_T^{\varepsilon_n}\Rightarrow\mathbf{P}_T^0$ be
one of the path.law cluster points supplied by Theorem \ref{the1}. For $t\in[%
0,T]$, let $e_t:C([0,T];\mathbb{R}^d)\to\mathbb{R}^d$ be the continuous
evaluation map $e_t(x)=x(t)$ and define the corresponding marginal cluster
law by
\begin{equation*}
\mu_t^0:=(e_t)_{\#}\mathbf{P}_T^0.
\end{equation*}
The notation $\mathbf{P}_T^0$ is reserved for the path law and $\mu_t^0$ for
its time.$t$ marginal. Different subsequences may yield different cluster
laws and different supports unless uniqueness of the zero.noise limit is
proved separately.

\begin{theorem}
\label{thm:support.inclusion} Assume (A1), fix $T>0$, and let
$\mathbf P_T^{\varepsilon_n}\Rightarrow\mathbf P_T^0$.
Write $\mathcal S_T^{\rm inst}(0)$ for the paths in
$\mathcal S_T(0)$ which satisfy $|x(s)|>0$ for every $s\in(0,T]$.
For every $t\in[0,T]$,
\begin{equation}
 \operatorname{supp}\mu_t^0
 =\overline{e_t(\operatorname{supp}\mathbf P_T^0)}
 \subseteq\overline{\{x(t):x\in\mathcal S_T^{\rm inst}(0)\}}.
 \label{eq:support.inclusion}
\end{equation}
No converse inclusion or uniqueness of the cluster law is asserted.
The assertion does not assign zero mass to the endpoint set of all
delayed solutions.
\end{theorem}

\begin{proof}
For a continuous map between separable metric spaces, the support of
the pushforward of a Borel probability is the closure of the image of
its support. Indeed every neighbourhood of the image of a support
point has a preimage neighbourhood of positive measure; conversely
an open set disjoint from the closure of the image has zero
pushforward measure, since the original support has full measure.
Apply this fact to $e_t$ to obtain the equality.
By Theorem \ref{the1}, a random path with law $\mathbf P_T^0$ is
Filippov and instantaneous almost surely. The simultaneous positivity
event is measurable: it is the intersection of
$\{\min_{s\in[\eta,T]}|X_s^0|>0\}$ over rational $\eta\in(0,T)$,
together with the endpoint condition when necessary.
On that full-measure event its value at $t$ lies in the endpoint set
in \eqref{eq:support.inclusion}. Its closed closure therefore has
marginal probability one and contains the marginal support. This
proves the inclusion without a claim that every admissible path is
selected or that every support path is itself instantaneous.
\end{proof}

\begin{remark}
Path type is not in general identifiable from a single endpoint. The same
point at time $t$ may be reached by an immediately escaping path and by a
delayed path. Therefore a path space exclusion, even if available, cannot be
translated into zero mass for the set of all endpoints attained by at least
one delayed path without a separate disjointness argument.
\end{remark}

\begin{proposition}
\label{prop:marginal.compact} Assume (A1) and let $\mu_t^0$ be any marginal
cluster law as above. Then
\begin{equation*}
\emph{supp}\mu_t^0 \subseteq \overline{B}(0,M_bt),
\end{equation*}
where $M_b=\|b\|_{L^\infty(\mathbb{R}^d)}$. In particular, $\emph{supp}%
\mu_t^0$ is compact.
\end{proposition}

\begin{proof}
Every vector $v\in F_b(x)$ satisfies $|v|\le M_b$, because the closed convex
hull and the intersection in the Filippov construction preserve the uniform
norm bound. Hence every $x\in\mathcal{S}_T(0)$ satisfies
\begin{equation*}
|x(t)|\le\int_0^t|\dot x(s)|\,\mathrm{d}s\le M_bt.
\end{equation*}
Theorem \ref{thm:support.inclusion} gives the same bound for the marginal
support. A closed subset of the compact ball $\overline B(0,M_bt)$ is
compact.
\end{proof}

\begin{theorem}
\label{the4} Let $d\ge2$, fix $t>0$, and let $\mu_t^0$ be a marginal cluster
law. Assume, in addition to the hypotheses used to construct $\mu_t^0$, that
there exist a compact set $\Omega_t\subset\mathbb{S}^{d-1}$ and a Lipschitz
function $\rho_t:\Omega_t\to[0,\infty)$ such that
\begin{equation}
\emph{supp}\mu_t^0 =\{\rho_t(\omega)\omega:\omega\in\Omega_t\}.
\label{eq:radial.graph}
\end{equation}
Then
\begin{equation}
\dim_H\bigl(\emph{supp}\mu_t^0\bigr) \le \dim_H(\Omega_t)\le d-1<d.
\label{haus}
\end{equation}
Consequently,
\begin{equation}
\mu_t^{0,\mathrm{ac}}=0, \qquad \mu_t^0=\mu_t^{0,\mathrm{sing}}  \label{sing}
\end{equation}
in the Lebesgue decomposition with respect to $d$.dimensional Lebesgue
measure.
\end{theorem}

\begin{proof}
Define $F_t(\omega)=\rho_t(\omega)\omega$. Compactness of $\Omega_t$
and continuity of $\rho_t$ give
$R_t:=\sup_{\omega\in\Omega_t}\rho_t(\omega)<\infty$.
Let $L_t$ be a Lipschitz constant for $\rho_t$ in the Euclidean chord
distance on the sphere. Then
\begin{align*}
 |F_t(\omega_1)-F_t(\omega_2)|
 &\le\rho_t(\omega_1)|\omega_1-\omega_2|
                          +|\rho_t(\omega_1)-\rho_t(\omega_2)|\\
 &\le(R_t+L_t)|\omega_1-\omega_2|.
\end{align*}
Thus $F_t$ is Lipschitz. Theorem \ref{liph} gives
$\dim_H(F_t(\Omega_t))\le\dim_H(\Omega_t)\le d-1$.
By \eqref{eq:radial.graph}, this is \eqref{haus}.
A set of Hausdorff dimension below $d$ has zero $d$-dimensional
Lebesgue measure: choose $s$ strictly between its dimension and $d$,
use coverings with arbitrarily small total $s$-diameter power, and
bound the corresponding sum of $d$-diameter powers by
$\eta^{d-s}$ times that sum at covering scale $\eta$.
Since $\mu_t^0$ is carried by its support, its absolutely continuous
component vanishes, proving \eqref{sing}.
\end{proof}

\begin{remark}
The radial-graph representation \eqref{eq:radial.graph} is an independent
geometric hypothesis, not a consequence of (A1) or Theorem \ref{the1}. If a
single direction can correspond to several admissible radii, the support may
instead contain a radial interval and the dimension bound above need not
hold.
\end{remark}

If an independent stochastic analysis establishes that, at time $t$, the
cluster law is supported exactly on $t\mathbb{S}^{d-1}$, then Theorem \ref%
{the4} is sharp because that sphere has Hausdorff dimension $d-1$.
Rotational symmetry alone is not used here as a substitute for proving the
zero-noise limit and its angular law.

\begin{remark}
The above theorem \ref{the4} was stated under the condition $d\geq 2$. When $%
d=1$, the unit sphere $\mathbb{S}^{0}=\{-1,1\}$ is a finite two-point set,
hence $\Omega _{t}\subset \mathbb{S}^{0}$ is automatically finite. The
radial map $\rho _{t}$ is trivially Lipschitz on a finite set, so the
Lipschitz hypothesis is no longer needed. The support $\supp(\mu _{t}^{0})$
becomes a finite subset of $\mathbb{R}$, which has Hausdorff dimension zero
and zero one-dimensional Lebesgue measure. The conclusion of pure
singularity of $\mu _{t}^{0}$ still holds, and the limiting measure reduces
to a finite sum of Dirac masses.
\end{remark}

\begin{remark}
The radial Lipschitz condition is stated for the Euclidean chord
distance. If $a=\operatorname{dist}_{\mathbb S}(\omega_1,\omega_2)
\in[0,\pi]$, then $|\omega_1-\omega_2|=2\sin(a/2)$ and
\[
 |\omega_1-\omega_2|\le a
                 \le\frac\pi2|\omega_1-\omega_2|.
\]
The two inequalities follow respectively from $\sin x\le x$ and
concavity of sine on $[0,\pi/2]$, which gives $\sin x\ge2x/\pi$.
Thus a geodesic Lipschitz constant $L_t^{\rm geo}$ gives a chord
Lipschitz constant at most $\pi L_t^{\rm geo}/2$. The radial map then
has chord Lipschitz constant at most $R_t+\pi L_t^{\rm geo}/2$.
The constants need not be unchanged when switching metrics, but the
Hausdorff dimension conclusion is unchanged.
\end{remark}

\begin{example}
Take $b(x)=2\mathrm{sgn}(x)\sqrt{|x|}$. The unique instantaneous escape
trajectories are $x_{+}(t)=t^{2} $ and $x_{-}(t)=-t^{2}$. The limiting
measure reduces to a discrete equal-weight mixture of Dirac masses:
\begin{equation*}
\mu _{t}^{0}=\frac{1}{2}\delta _{t^{2}}+\frac{1}{2}\delta _{-t^{2}}.
\end{equation*}%
No continuous absolutely continuous component exists; all mass concentrates
on two single points at time $t$.
\end{example}

\begin{example}
Let $b(x)=x/|x|$ for $x\ne0$ and $b(0)=0$, in dimension $d\ge2$.
This bounded field satisfies (A1), with $\varphi=h=1$ on $(0,\infty)$
and $\psi(t)=t$. Theorem \ref{the1} shows that every radial cluster
has $R_t^0\ge t$. The original integral equation gives
$|X_t^\varepsilon|\le t+\varepsilon\sup_{s\le t}|W_s|$, so every
radial cluster also has $R_t^0\le t$ for all $t$.
Tightness and uniqueness of this deterministic radial cluster imply
$R^\varepsilon\to(t\mapsto t)$ uniformly in probability on compact
time intervals. For each fixed noise the drift, starting point, and
Brownian law are invariant under every orthogonal transformation;
uniqueness in law gives the same invariance for its marginals and
all their limits. At a fixed $t>0$, a limiting marginal is therefore
the unique rotation-invariant probability measure on $t\mathbb S^{d-1}$,
namely normalized surface measure. This proves
\[
 \operatorname{supp}\mu_t^0=t\mathbb S^{d-1}.
\]
The radial-graph hypothesis is realized by $\Omega_t=\mathbb S^{d-1}$
and the scalar profile $\rho_t(\omega)=t$, whose Lipschitz constant
is zero. The vector-valued map $F_t(\omega)=t\omega$ has Lipschitz
constant $t$. The support has Hausdorff dimension $d-1$ and zero
$d$-dimensional Lebesgue measure. This nontrivial example attains the
bound in Theorem \ref{the4}; it is not an example of (A2) or (A5),
since its radial potential is constant in angle.
\end{example}

\subsection{A continuous-mapping criterion for exit times}

\label{sec:exit.time} For $x\in C([0,T];\mathbb{R}^{d}),$ set $\tau _{\rho
,T}(x):=\inf \{t\in \lbrack 0,T]:|x(t)|\geq \rho \}\wedge T$. Let $\mathcal{C%
}_{\rho ,T}$ consist of those $x$ for which either

\begin{enumerate}
\item[(C1)] $\tau_{\rho,T}(x)=T$; or

\item[(C2)] $\tau:=\tau_{\rho,T}(x)<T$ and, for every $\eta>0$, there exists
$s\in(\tau,\min\{\tau+\eta,T\}]$ with $|x(s)|>\rho$.
\end{enumerate}
(When $\tau<T$, continuity and the first-hitting definition already imply $%
|x(t)|<\rho$ for $t<\tau$ and $|x(\tau)|=\rho$; (C2) is the genuine
post-exit crossing condition, excluding tangential contact.)

\begin{lemma}
\label{lem:exit.continuity} The map $\tau_{\rho,T}:C([0,T];\mathbb{R}^d)\to[%
0,T]$ is continuous (uniform norm) at every $x\in \mathcal{C}_{\rho ,T}$; if
(C2) fails at some $x$ with $\tau <T$, the radial contractions $(1-n^{-1})x$
show it is discontinuous there.
\end{lemma}

\begin{proof}
Let $x_n\to x$ uniformly.

\emph{Lower bound.} For any $\eta>0$, compactness and the first-hitting
property give $m_\eta:=\max_{0\le s\le\tau-\eta}|x(s)|<\rho$ (read this on $%
[0,T-\eta]$ when $\tau=T$). For $n$ large, $\max_{[0,\tau-\eta]}|x_n|\le(m_%
\eta+\rho)/2<\rho$, hence $\tau(x_n)\ge\tau-\eta$, so $\liminf_n\tau(x_n)\ge%
\tau$.

\emph{Upper bound (only when $\tau<T$).} By (C2), choose $%
s\in(\tau,\tau+\eta]$ with $|x(s)|>\rho$. With $\delta=(|x(s)|-\rho)/2$,
large $n$ satisfy $|x_n(s)|>\rho$, hence $\tau(x_n)\le s\le\tau+\eta$, so $%
\limsup_n\tau(x_n)\le\tau$.

When $\tau=T$ the lower bound together with $\tau(x_n)\le T$ already gives $%
\tau(x_n)\to T$. Necessity of (C2) follows from the radial-contraction
sequence stated in the lemma.
\end{proof}

%Fix $T>0$ and $\rho>0$. For $x\in C([0,T];\mathbb{R}^d)
%$ define the first exit time from the open ball by
%\begin{equation*}
%\tau_{\rho,T}(x) :=\inf\{t\in[0,T]:|x(t)|\ge\rho\}\wedge T,
%\end{equation*}
%with the convention $\inf\varnothing=\infty$. Let $\mathcal{C}_{\rho,T}$
%denote the set of paths satisfying one of the following alternatives:
%
%\begin{enumerate}
%\item $\tau_{\rho,T}(x)<T$, and
%
%\begin{enumerate}
%\item for every $\eta\in(0,\tau_{\rho,T}(x))$,
%\begin{equation*}
%\max_{0\le s\le\tau_{\rho,T}(x)-\eta}|x(s)|<\rho;
%\end{equation*}
%
%\item for every $\eta>0$ there exists $s\in[\tau_{\rho,T}(x),\min\{\tau_{%
%\rho,T}(x)+\eta,T\}]$ such that $|x(s)|>\rho$;
%\end{enumerate}
%
%\item $\tau_{\rho,T}(x)=T$ and $\max_{0\le s\le T}|x(s)|<\rho$.
%\end{enumerate}
%
%\begin{lemma}[Continuity points of the exit.time functional]
%\label{lem:exit.continuity} The map $\tau_{\rho,T}:C([0,T];\mathbb{R}^d)\to[%
%0,T]$ is continuous at every $x\in\mathcal{C}_{\rho,T}$ with respect to the
%uniform norm.
%\end{lemma}
%
%\begin{proof}
%Let $x_n\to x$ uniformly. The strict interior bound on every compact
%interval ending before $\tau_{\rho,T}(x)$ gives $\liminf_n\tau_{\rho,T}(x_n)%
%\ge\tau_{\rho,T}(x)$. If the exit occurs before $T$, the existence of points
%strictly outside the ball arbitrarily soon after the exit gives $%
%\limsup_n\tau_{\rho,T}(x_n)\le\tau_{\rho,T}(x)$. In the no-exit case, the
%uniform positive distance of the whole path from the boundary gives $%
%\tau_{\rho,T}(x_n)=T$ for all sufficiently large $n$.
%\end{proof}

\begin{theorem}
\label{thm:exit.time} Assume that, for some sequence $\varepsilon_n%
\downarrow0$,
\begin{equation*}
\mathbf{P}_T^{\varepsilon_n}\Longrightarrow\mathbf{P}_T^0 \quad\text{in }%
C([0,T];\mathbb{R}^d),
\end{equation*}
and assume
\begin{equation*}
\mathbf{P}_T^0(\mathcal{C}_{\rho,T})=1.
\end{equation*}
Then
\begin{equation}
\tau_{\rho,T}(X^{\varepsilon_n}) \Longrightarrow \tau_{\rho,T}(X^0) \quad%
\text{in distribution}.  \label{4.2.1}
\end{equation}
Here $X^0$ denotes the coordinate process under $\mathbf{P}_T^0$.
\end{theorem}

\begin{proof}
Apply the continuous-mapping theorem together with Lemma \ref%
{lem:exit.continuity}.
\end{proof}

\begin{remark}
Skorohod's representation theorem may produce copies on a new probability
space that converge uniformly almost surely, and on those copies the exit
times converge almost surely under the same continuity assumption. This does
not imply almost sure convergence of the original processes driven on their
original probability space.
\end{remark}

\section{Conclusion and Future Work}
\label{sect4}
Under (A1), the paper proves relative compactness of the zero-noise
path laws, their Filippov identification, and the uniform deterministic
lower radial bound $|X_t^0|\ge\psi(t)>0$ for every positive time in
every cluster law. The proof uses a comparison on the same Brownian
space from the origin, not continuous dependence obtained by a
Lipschitz mollification with uncontrolled constants. The polar limit
uses the same Filippov velocity in its radial and angular components,
and its angular path is considered only on positive-time intervals.

For smooth drifts away from zero, the unit-sphere gradient is
$\nabla_{\mathbb S}V=r(A+B)$, whereas the actual angular velocity is
$B$. The complete It\^{o} equation retains the curvature drift. Under the
explicit separable structure (A5), the angular direction converges
almost surely to the maximum set for every fixed positive noise.
The proof uses finite angular quadratic variation, an infinite
gradient clock with a convergent perturbation, a gradient energy
argument, and a direct expanding-coordinate exclusion of nonmaxima. It neither
uses the state radius as a stochastic time parameter nor asserts a
finite expected whole angular clock in dimension two. The additional
structure is a genuine restriction; no such general selection theorem
under (A1)--(A4) and (A2$^{\prime}$) alone is claimed here.

The unbounded cubic drift and the discontinuous $b_2^\delta$ drift
have separate proofs of their four-branch finite-time path-law limits.
For $b_2^\delta$, the absolute-coordinate comparison also handles
$\delta=0$, joint vanishing of $\delta$ and noise, and the two iterated
limits. These conclusions are weak path-law statements, not fair
random-branch convergence on a fixed original Brownian space.

At a fixed time, every marginal cluster has compact support contained
in the closure of instantaneous Filippov endpoints. Neither a converse
inclusion nor uniqueness of general cluster laws is asserted.
Hausdorff dimension at most $d-1$ and singularity follow when the
independent compact Lipschitz radial-graph hypothesis is imposed;
normalized surface measure for the radial unit drift attains that
bound. Exit-time convergence remains subject to the stated
continuous-mapping condition. The inward-drift estimate proves
uniform convergence to the zero path on compact time intervals by a
Cartesian exponential supermartingale.
The existence and selection of final directions outside (A5), and
singularity without the radial-graph hypothesis, are not established
by these arguments. Those limitations are distinct from the results
proved above.

\section{Appendix}
\label{app}
\subsection{Proof of Proposition \ref{pro1}}

\begin{proof}
Write $M=\sup_y|b(y)|$. The paths
$Y^\varepsilon=X^{x,\varepsilon}-\varepsilon W$ are $M$-Lipschitz and
start at $x$, while $\varepsilon W\to0$ uniformly in probability on
$[0,T]$. The Arzela--Ascoli theorem proves tightness. Along any convergent
subsequence, apply the Skorohod representation theorem to the joint laws
of $(X^{x,\varepsilon},\varepsilon W)$. Denote the represented processes
by $(X_n,N_n)$ and their almost surely uniform limit by $(X,0)$. In
particular, $X$ is $M$-Lipschitz and $X(0)=x$.

Let $\mathcal D$ be the countable family of open balls in $\mathbb R^d$ with rational
centres and positive rational radii. Let $\mathcal U:=\mathbb Q^d$ be the countable set
of rational vectors, which we use as test directions in the support-function description
of convex sets; it is dense in $\mathbb R^d$. For $D\in\mathcal D$ and $u\in\mathcal U$,
define
\[
 h_D(u):=\operatorname*{ess\,sup}_{y\in D}\langle u,b(y)\rangle ,
\]
where the essential supremum is taken with respect to $d$-dimensional Lebesgue measure.

%We give the essential-value argument that identifies its derivative.
%Let $D$ run through the open balls with rational centers and positive
%rational radii, and let $u\in\mathbb{Q}^d$. Put
%\[
% h_D(u)=\mathop{\rm ess\,sup}_{y\in D}\langle u,b(y)\rangle.
%\]
%There is a single Lebesgue-null Borel set $N$, obtained by taking the
%countable union of the exceptional sets for these inequalities, such
%that $\langle u,b(y)\rangle\le h_D(u)$ whenever $y\in D\setminus N$.

%For each fixed positive noise, finite-time Girsanov equivalence with
%Brownian motion gives a density at each positive time. Fubini's theorem
%therefore gives $\int_0^T1_N(X^{x,\varepsilon}_t)\,\mathrm{d}t=0$ almost surely.

For fixed $\varepsilon>0$ and $T<\infty$, define
\[
Z_T^\varepsilon=\exp\left\{
-\varepsilon^{-1}\int_0^T
\langle b(X_s^{x,\varepsilon}),\mathrm{d}W_s\rangle
-\frac1{2\varepsilon^2}\int_0^T
|b(X_s^{x,\varepsilon})|^2\,\mathrm{d}s\right\}.
\]
With $M=\sup_y|b(y)|$, Novikov's condition holds because
\[
\mathbb E\exp\left\{\frac1{2\varepsilon^2}
\int_0^T|b(X_s^{x,\varepsilon})|^2\,\mathrm{d}s\right\}
\le e^{M^2T/(2\varepsilon^2)}<\infty.
\]
Thus $\mathbb E Z_T^\varepsilon=1$ and $Z_T^\varepsilon>0$ almost
surely. Define $\mathrm{d}\mathbb Q_T=Z_T^\varepsilon\,
\mathrm{d}\mathbb P$ on $\mathcal F_T$. These two finite-horizon
measures are equivalent, and Girsanov's theorem gives the
$\mathbb Q_T$-Brownian motion
\[
\widetilde W_t=W_t+\varepsilon^{-1}\int_0^t
b(X_s^{x,\varepsilon})\,\mathrm{d}s,
\qquad 0\le t\le T.
\]
Substitution in the original SDE now cancels the drift:
$X_t^{x,\varepsilon}=x+\varepsilon\widetilde W_t$.
For $0<t\le T$, its law under $\mathbb Q_T$ is
$N(x,\varepsilon^2tI_d)$. Consequently every Lebesgue-null Borel
set $N_0$ has $\mathbb Q_T(X_t^{x,\varepsilon}\in N_0)=0$ and,
by equivalence, $\mathbb P(X_t^{x,\varepsilon}\in N_0)=0$.
Tonelli's theorem yields
\[
\mathbb E\int_0^T\mathbf 1_{N_0}(X_s^{x,\varepsilon})\,\mathrm{d}s
=\int_0^T\mathbb P(X_s^{x,\varepsilon}\in N_0)\,\mathrm{d}s=0.
\]
The value at time zero does not affect this integral.
For the simultaneous support-function inequalities, use the single
Borel set
\[
N_*:=\bigcup_{D\in\mathcal D,\ u\in\mathcal U}
\{y\in D:\langle u,b(y)\rangle>h_D(u)\}.
\]
Each set in this countable union is Lebesgue-null by the definition of essential supremum. Hence $N_*$ is null, and outside $N_*$ all
of these inequalities hold simultaneously. The preceding
zero-occupation statement applied to $N_*$ is a measurable property
of the path law. It therefore holds for each represented $X_n$;
a countable intersection makes it simultaneous for every $n$.
This is the probability-one event used in the following increment
inequalities. Only finite-horizon measure equivalence has been used.

After the Skorokhod representation, write $Y_{n}:=X_{n}-N_{n}$ for the
finite-variation part, so that
\begin{equation*}
Y_{n}(t)-Y_{n}(s)=\int_{s}^{t}b(X_{n}(r))\,\mathrm{d}r.
\end{equation*}%
On the probability-one event just established, for every $D\in \mathcal{D}$,
every $u\in \mathcal{U}$, and every compact interval $[s,t]\subset \lbrack
0,T]$ whose image $X_{n}([s,t])$ is contained in $D$,
\begin{equation}
\big\langle u,Y_{n}(t)-Y_{n}(s)\big\rangle=\int_{s}^{t}\langle
u,b(X_{n}(r))\rangle \,\mathrm{d}r\leq h_{D}(u)(t-s).  \label{3.x.1}
\end{equation}%
This is a measurable property of the joint law of $(X_{n},N_{n})$; it
suffices to impose it first for intervals with rational endpoints and to
extend it afterwards by path continuity.

Now fix $D\in \mathcal{D}$ and $u\in \mathcal{U}$, and let $[s,t]$ be a
non-degenerate compact interval such that $X([s,t])\Subset D$ (that is, $%
X([s,t])$ is at a positive distance from $D^{c}$). Uniform convergence $%
X_{n}\rightarrow X$ implies $X_{n}([s,t])\subset D$ for all sufficiently
large $n$. Since $N_{n}\rightarrow 0$ uniformly, passing to the limit in (%
\ref{3.x.1}) gives
\begin{equation*}
\big\langle u,X(t)-X(s)\big\rangle\leq h_{D}(u)(t-s).
\end{equation*}%
The set $G_{D}:=\{r\in (0,T):X(r)\in D\}$ is open. On each connected
component of $G_{D}$, apply the preceding inequality to compact subintervals
and use the absolute continuity of $X$; differentiating at Lebesgue-almost
every $r\in G_{D}$ gives
\begin{equation}
\langle u,\dot{X}(r)\rangle \leq h_{D}(u).  \label{3.x.2}
\end{equation}%
The families $\mathcal{D}$ and $\mathcal{U}$ are countable, so (\ref{3.x.2}) holds
simultaneously for all $D\in \mathcal{D}$ and $u\in \mathcal{U}$ outside a
single Lebesgue-null set of times.

For $D\in \mathcal{D}$ let
\begin{equation*}
C_{D}:=\overline{\mathrm{co}}\mathrm{ess}_{D}b
\end{equation*}%
be the essential closed convex hull of $b$ on $D$, i.e. the closed convex
hull of its essential range there. The essential range of $b$ on $D$ is
\begin{equation*}
\mathrm{ess}_{D}b:=\Big\{v\in \mathbb{R}^{d}:\ \text{for every open
neighbourhood }U\text{ of }v,\ \mathcal{L}^{d}\big(\{y\in D:b(y)\in U\}\big)%
>0\Big\},
\end{equation*}%
where $\mathcal{L}^{d}$ denotes $d$-dimensional Lebesgue measure.
Equivalently, a value $v$ belongs to the essential range precisely when $b$
takes values arbitrarily close to $v$ on sets of positive Lebesgue measure;
values attained only on a Lebesgue-null set are excluded. Its support
function is $u\mapsto h_{D}(u)$, and therefore
\begin{equation}
C_{D}=\bigcap_{u\in \mathcal{U}}\big\{v\in \mathbb{R}^{d}:\langle u,v\rangle
\leq h_{D}(u)\big\}.  \label{3.x.3}
\end{equation}%
Indeed, the supporting-hyperplane theorem identifies the intersection over
all real directions $u$ with $C_{D}$; since $b$ is bounded, $C_{D}$ is
bounded and its support function is continuous, so the real directions may
be replaced by the dense set $\mathcal{U}$ without changing the intersection.

Finally, let $r$ be a time at which $\dot{X}(r)$ exists and (\ref{3.x.2})
holds. For every rational ball $D$ containing $X(r)$, (\ref{3.x.2}) and (\ref%
{3.x.3}) imply $\dot{X}(r)\in C_{D}$. Hence
\begin{equation*}
\dot{X}(r)\in \bigcap_{\substack{ D\in \mathcal{D} \\ X(r)\in D}}C_{D}.
\end{equation*}%
We identify this intersection with the convex-valued Filippov map
in Definition \ref{filir}. For any bounded open ball $D$, put
$E_D=\mathrm{ess}_D b$. Its complement is the union of the members
of a countable base of the value space whose inverse images in $D$
have zero Lebesgue measure. It follows that $E_D$ is closed and
that $b(y)\in E_D$ for almost every $y\in D$. Boundedness of $b$
makes $E_D$ compact, and it is nonempty since $D$ has positive
measure. Let $N_D\subset D$ be the null Borel set on which
$b(y)\notin E_D$.
For every null Borel set $N_0$, every neighbourhood of a point of
$E_D$ has a positive-measure inverse image in $D$, and so still has an inverse image point in $D\setminus N_0$. Thus
$E_D\subset\overline{b(D\setminus N_0)}$. Conversely,
$b(D\setminus N_D)\subset E_D$. Taking closed convex hulls proves
\[
C_D=\overline{\mathrm{co}}\,E_D
=\bigcap_{\mathcal L^d(N_0)=0}
\overline{\mathrm{co}}\,b(D\setminus N_0).
\]
Allowing all Lebesgue-null sets instead of Borel null sets gives
the same intersection, since every null set is contained in a
Borel null set. Therefore the definition of $F_b$ gives
\[
F_b(z)=\bigcap_{\delta>0}
\overline{\mathrm{co}}\,\mathrm{ess}_{B(z,\delta)}b
=\bigcap_{\substack{D\in\mathcal D\\z\in D}}C_D.
\]
Open or closed centred balls give the same intersection, by
cofinality of the two neighbourhood families. For the last equality,
every ball containing $z$ contains a centred ball about $z$, and
every centred ball contains a rational ball containing $z$; use
also $C_{D_1}\subset C_{D_2}$ when $D_1\subset D_2$.
The closed convex hull is retained throughout. The intersection
already obtained for $\dot X(r)$ is consequently $F_b(X(r))$. Thus
\begin{equation*}
\dot{X}(t)\in F_{b}(X(t))\qquad \text{for Lebesgue-almost every }t\in
\lbrack 0,T],
\end{equation*}%
with $X(0)=x$. The law of $X$ is the cluster law, so every cluster law is
concentrated on absolutely continuous paths $z$ satisfying $z(0)=x$ and $%
\dot{z}(t)\in F_{b}(z(t))$ almost everywhere. This proves the proposition.
\end{proof}

\subsection{Hitting probabilities for a 1-D diffusion}

Consider the one-dimensional SDE
\begin{equation}
\mathrm{d}X_{t}=b(X_{t})\mathrm{d}t+\sigma (X_{t})\mathrm{d}W_{t},\qquad
X_{0}=x,  \label{sdel}
\end{equation}%
on an open interval $I\subset \mathbb{R}$. We assume throughout that $\sigma
(x)>0$ on $I$, that $b\in L_{\mathrm{loc}}^{1}(I)$, that $\sigma $ is
locally bounded and that $\sigma ^{2}$ is locally bounded away from zero.
Then $2b/\sigma ^{2}\in L_{\mathrm{loc}}^{1}(I)$, and~\eqref{sdel} admits a
unique weak solution which is a strong Markov continuous diffusion on $I$
\cite{stroock_varadhan79,ks91}.

Let $a<x<b$ be three points of $I$, and define
\begin{equation*}
\tau _{a}=\inf \{t>0:X_{t}=a\},\qquad \tau _{b}=\inf \{t>0:X_{t}=b\},\qquad
\tau =\tau _{a}\wedge \tau _{b}.
\end{equation*}%
The \textbf{scale function} $S:I\rightarrow \mathbb{R}$ is defined, up to an
additive constant, by
\begin{equation*}
S^{\prime }(y)=\exp \!\left( -\int_{x_{0}}^{y}\frac{2b(z)}{\sigma ^{2}(z)}\,%
\mathrm{d}z\right) ,
\end{equation*}%
with an arbitrary reference point $x_{0}\in I$. Since the exponential is
strictly positive, $S^{\prime }(y)>0$ and $S$ is strictly increasing.

\begin{theorem}
\label{hit} Under the above hypotheses,
\begin{equation*}
\mathbb{P}_{x}(\tau _{a}<\tau _{b})=\frac{S(b)-S(x)}{S(b)-S(a)}.
\end{equation*}%
The value is independent of the reference point $x_{0}$ used in the
definition of $S$.
\end{theorem}

\begin{proof}
\noindent \emph{Part 1: $S(X_{t})$ is a local martingale and optional
stopping.} Differentiating the scale density gives, almost everywhere,
\begin{equation*}
S^{\prime \prime }(y)=-\frac{2b(y)}{\sigma ^{2}(y)}S^{\prime }(y),\qquad
\text{equivalently}\qquad b(y)S^{\prime }(y)+\tfrac{1}{2}\sigma
^{2}(y)S^{\prime \prime }(y)=0\quad \text{a.e. on }I.
\end{equation*}%
Indeed $S\in W_{\mathrm{loc}}^{2,1}(I)$: $S'$ is locally absolutely continuous
(hence $S''$ exists a.e.) and $S''\in L^1_{\mathrm{loc}}(I)$. Applying
the generalized It\^{o}--Tanaka formula (legitimate because $S^{\prime }$ is
absolutely continuous; see \cite{ks91,dm}) yields
\begin{align*}
S(X_{t})& =S(x)+\int_{0}^{t}S^{\prime }(X_{s})\mathrm{d}X_{s}+\frac{1}{2}%
\int_{0}^{t}S^{\prime \prime }(X_{s})\,\mathrm{d}\langle X,X\rangle _{s} \\
& =S(x)+\int_{0}^{t}b(X_{s})S^{\prime }(X_{s})\mathrm{d}s+\int_{0}^{t}S^{%
\prime }(X_{s})\sigma (X_{s})\,\mathrm{d}W_{s}+\frac{1}{2}%
\int_{0}^{t}S^{\prime \prime }(X_{s})\sigma ^{2}(X_{s})\mathrm{d}s.
\end{align*}%
The stochastic integral is a local martingale ($S^{\prime }$ is bounded and $%
\sigma $ is locally bounded on every compact subinterval), while the two
finite variation terms cancel a.e.; hence
\begin{equation}
S(X_{t})=S(x)+\int_{0}^{t}S^{\prime }(X_{s})\sigma (X_{s})\,\mathrm{d}W_{s}
\label{S1}
\end{equation}%
is a local martingale.

Stop at $\tau $. Up to time $\tau $ the process stays in the compact
interval $[a,b]$, on which the continuous function $S$ is bounded;
consequently $S(X_{t\wedge \tau })$ is a bounded local martingale, hence a
true martingale, and
\begin{equation*}
\mathbb{E}_{x}[S(X_{t\wedge \tau })]=S(x)\qquad \text{for every }t\geq 0.
\end{equation*}%
Non-degeneracy ($\sigma ^{2}$ bounded away from zero) implies that a regular
one-dimensional diffusion exits the bounded interval $(a,b)$ in finite time
almost surely, $\tau <\infty $ a.s. \cite{ks91,ItoMcKean1974}. Passing $%
t\rightarrow \infty $ with the bounded convergence theorem gives $\mathbb{E}%
_{x}[S(X_{\tau })]=S(x)$. By continuity $X_{\tau }\in \{a,b\}$ (a continuous
path cannot equal $a\neq b$ simultaneously), thus, with $p:=\mathbb{P}%
_{x}(\tau _{a}<\tau _{b})$ and $\mathbb{P}_{x}(\tau _{b}<\tau _{a})=1-p$,
\begin{equation*}
S(a)\,p+S(b)(1-p)=S(x).
\end{equation*}%
Solving for $p$ and using $S(b)-S(a)>0$,
\begin{equation}
p=\frac{S(b)-S(x)}{S(b)-S(a)}.  \label{p}
\end{equation}

\noindent\emph{Part 2: independence of the reference point.} Let $x_0,\tilde
x_0\in I$ produce $S$ and $\tilde S$. Splitting $\int_{\tilde
x_0}^{y}=\int_{\tilde x_0}^{x_0}+\int_{x_0}^{y}$,
\begin{equation*}
\tilde S^{\prime }(y) =\exp\!\left(-\int_{\tilde x_0}^{x_0}\frac{2b}{\sigma^2%
}\right)S^{\prime }(y) =:K\,S^{\prime }(y),\qquad K>0,
\end{equation*}
and integration gives the affine relation $\tilde S(y)=K S(y)+C$.
Substitution into~\eqref{p} makes both $K$ and the additive constant $C$
cancel:
\begin{equation*}
\frac{\tilde S(b)-\tilde S(x)}{\tilde S(b)-\tilde S(a)} =\frac{K(S(b)-S(x))}{%
K(S(b)-S(a))} =\frac{S(b)-S(x)}{S(b)-S(a)}.
\end{equation*}
The hitting probability is therefore independent of $x_0$, completing the
proof.
\end{proof}

\begin{remark}
\label{rem:bm-vs-bessel} For standard one-dimensional Brownian motion, $%
b\equiv 0$, $\sigma \equiv 1$, so the scale function is linear, $S(y)=y$.
For $0<a<r<c$,
\begin{equation*}
\mathbb{P}_{r}(\tau _{a}<\tau _{c})=\frac{c-r}{c-a},\qquad \mathbb{P}%
_{r}(\tau _{c}<\tau _{a})=\frac{r-a}{c-a}.
\end{equation*}%
As $a\downarrow 0^{+}$,
\begin{equation*}
\mathbb{P}_{r}(\tau _{c}<\tau _{a})\longrightarrow \frac{r}{c}\in
(0,1),\qquad \mathbb{P}_{r}(\tau _{a}<\tau _{c})\longrightarrow 1-\frac{r}{c}%
;
\end{equation*}%
the probability of reaching the outer level $c$ first does \emph{not} tend
to $1$. The radial part of $d$-dimensional Brownian motion obeys
\begin{equation*}
\mathrm{d}R_{t}=\frac{\varepsilon ^{2}(d-1)}{2R_{t}}\mathrm{d}t+\varepsilon
\,\mathrm{d}B_{t},
\end{equation*}%
whose scale density is
\begin{equation*}
S^{\prime }(r)=\exp \!\left( -\int^{r}\frac{2\,\varepsilon ^{2}(d-1)/(2u)}{%
\varepsilon ^{2}}\,\mathrm{d}u\right) =r^{-(d-1)},
\end{equation*}%
i.e.\ $S(r)=\ln r$ for $d=2$ and $S(r)=r^{2-d}/(2-d)$ for $d\geq 3$, with $%
S(0^{+})=-\infty $. Hence $\mathbb{P}_{r}(\tau _{a}<\tau
_{c})=(S(c)-S(r))/(S(c)-S(a))\rightarrow 0$ as $a\downarrow 0^{+}$, or
equivalently $\mathbb{P}_{r}(\tau _{c}<\tau _{a})\rightarrow 1$. The outward
preference is therefore a distinctive feature of the $d\geq 2$ radial
diffusion and stems precisely from the geometric drift term \cite%
{dm,ItoMcKean1974,BorodinSalminen2002}.
\end{remark}

\begin{lemma}
\label{Lf}
Let $X_t$ solve $\mathrm dX_t=b(X_t)\mathrm dt+\varepsilon\,\mathrm dW_t$
with fixed $\varepsilon>0$, $d$-dimensional Brownian motion $W$, and
$M:=\sup_x|b(x)|<\infty$.
For every $L>0$, the exit time
\[
\tau_L:=\inf\{t\ge0:|X_t|\ge L\}
\]
is finite almost surely.
\end{lemma}

\begin{proof}
%Split time into unit intervals and set $\Delta W_n=W_{n+1}-W_n\sim N(0,1)$, independent.
%On the event that $X_t$ stays in the ball $\{|x|\le L\}$ up to time $n+1$, the increment over
%$[n,n+1]$ is bounded by drift plus noise:
%\[
%|X_{n+1}-X_n|\le M+\varepsilon|\Delta W_n|.
%\]
%If $\varepsilon|\Delta W_n|>2L+M$, then regardless of the point $\{|X_n|\le L\}$ at the
%worst endpoint, $|X_{n+1}|>L$, i.e. exit occurs during that interval. Thus
%\[
%p:=\mathbb P\big(\varepsilon|\Delta W_1|>2L+M\big)>0.
%\]
%The increments are independent, so the probability of remaining in the ball for $n$ consecutive
%intervals is $(1-p)^n\to0$. Hence $\mathbb P(\tau_L=\infty)=0$, i.e. $\tau_L<\infty$ a.s.
The diffusion reaches every radius $L$ in finite time almost surely. Let $%
\tau _{L}:=\inf \{t\geq 0:|X_{t}^{x,\varepsilon }|>L\}$ be the exit time
from the closed ball $\overline{B}_{L}=\{y:|y|\leq L\}$. On each unit
interval $[j,j+1]$,
\begin{equation*}
X_{j+1}^{x,\varepsilon }-X_{j}^{x,\varepsilon
}=\int_{j}^{j+1}b(X_{s}^{x,\varepsilon })\,\mathrm{d}s+\varepsilon
\,(W_{j+1}-W_{j}),
\end{equation*}%
and the drift integral satisfies $\big|\int_{j}^{j+1}b(X_{s}^{x,\varepsilon
})\,\mathrm{d}s\big|\leq M$. Set
\begin{equation*}
E_{j}:=\big\{\varepsilon \,|W_{j+1}-W_{j}|>2L+M\big\}.
\end{equation*}%
If the path remains in $\overline{B}_{L}$ throughout $[j,j+1]$, then $%
|X_{j+1}^{x,\varepsilon }-X_{j}^{x,\varepsilon }|\leq 2L$. On $E_{j}$,
however,
\begin{equation*}
|X_{j+1}^{x,\varepsilon }-X_{j}^{x,\varepsilon }|\geq \varepsilon
\,|W_{j+1}-W_{j}|-M>2L.
\end{equation*}%
Thus the occurrence of $E_{j}$ forces an exit from $\overline{B}_{L}$ during
that interval, which gives the one-sided inclusion
\begin{equation*}
\{\tau _{L}>n\}\subseteq \bigcap_{j=0}^{n-1}E_{j}^{c}.
\end{equation*}%
The increments $\Delta W_{j}:=W_{j+1}-W_{j}$ are independent and identically
distributed as $N(0,I_{d})$. Writing
\begin{equation*}
p:=\mathbb{P}\big(\varepsilon \,|\Delta W_{0}|>2L+M\big)=\mathbb{P}\big(%
|Z|>(2L+M)/\varepsilon \big)>0,\qquad Z\sim N(0,I_{d}),
\end{equation*}%
independence yields the geometric upper bound
\begin{equation*}
\mathbb{P}(\tau _{L}>n)\leq \prod_{j=0}^{n-1}\mathbb{P}(E_{j}^{c})=(1-p)^{n}%
\xrightarrow[n\to\infty]{}0.
\end{equation*}%
Hence $\mathbb{P}(\tau _{L}=\infty )=0$, that is, the diffusion exits every
bounded ball in finite time almost surely.
\end{proof}

\subsection{Examples of radial drift with non-negative or negative sign}

\label{examples}

\paragraph{Case 1: Non-negative radial drift.}

For a radial field $b(x)=\alpha (|x|)\,x/|x|$ (for a general field, replace $%
\alpha (r)$ below by the worst-case lower bound $\varphi
(r)=\inf_{|x|=r}\langle b(x),x/|x|\rangle $), the radial coefficient is
\begin{equation*}
\alpha (r):=\frac{1}{r}\langle b(x),x\rangle \geq 0,\qquad r=|x|>0.
\end{equation*}%
It\^{o}'s formula gives the radial equation
\begin{equation*}
\mathrm{d}R_{t}^{\varepsilon }=\left[ \alpha (R_{t}^{\varepsilon })+\frac{%
\varepsilon ^{2}(d-1)}{2R_{t}^{\varepsilon }}\right] \mathrm{d}t+\varepsilon
\mathrm{d}B_{t},\qquad R_{0}^{\varepsilon }=0.
\end{equation*}%
For bounded drift with nonnegative radial component in dimension
$d\ge2$, the radial equation from zero is justified directly by the
regularization $f_\eta(x)=\sqrt{|x|^2+\eta^2}$ used in
\eqref{4.1}: its nonnegative drift terms give
$\int_0^T R_t^{-1}dt<\infty$ almost surely, and the lower time endpoint
can then tend to zero in the ordinary radial Ito formula.
Under (A1) the same-noise comparison is Lemma
\ref{lem:radial-comparison-origin}; it uses positive geometric drift,
not the invalid uniform mollification estimate. The exit bound below
uses the radial equation itself, stopped before reaching the outer
boundary, and does not require a comparison with a nonmonotone
stochastic radius parameter.

Assume there exist $r_{0},\kappa >0$ such that
\begin{equation*}
\alpha (r)\geq \kappa >0,\qquad 0<r<r_{0},
\end{equation*}%
i.e.\ the drift is uniformly positive near the origin and does not
degenerate. Fix $t_{0}>0$ with $\kappa t_{0}<r_{0}$, and, for $0<\delta
<\kappa t_{0}$, introduce the first up-crossing time
\begin{equation*}
\tau _{\delta }:=\inf \{t>0:R_{t}^{\varepsilon }\geq \delta \}.
\end{equation*}%
For $s\leq \tau _{\delta }$, continuity gives $R_{s}^{\varepsilon }\leq
\delta <r_{0}$, hence $\alpha (R_{s}^{\varepsilon })\geq \kappa $ and the
geometric term is non-negative, so
\begin{equation*}
R_{t\wedge \tau _{\delta }}^{\varepsilon }=\int_{0}^{t\wedge \tau _{\delta
}} \left[ \alpha (R_{s}^{\varepsilon })+\frac{\varepsilon ^{2}(d-1)}{%
2R_{s}^{\varepsilon }}\right] \mathrm{d}s+\varepsilon \,\beta _{t\wedge \tau
_{\delta }}\geq \kappa (t\wedge \tau _{\delta })+\varepsilon \,\beta
_{t\wedge \tau _{\delta }}.
\end{equation*}%
On the event $\{\tau _{\delta }>t_{0}\}$ one has $t_{0}\wedge \tau _{\delta
}=t_{0}$ and $R_{t_{0}}^{\varepsilon }<\delta $, whence
\begin{equation*}
\delta >R_{t_{0}}^{\varepsilon }\geq \kappa t_{0}+\varepsilon \,\beta
_{t_{0}},\qquad \text{i.e.}\qquad \beta _{t_{0}}<\frac{\delta -\kappa t_{0}}{%
\varepsilon }.
\end{equation*}%
Thus $\{\tau _{\delta }>t_{0}\}\subseteq \{\beta _{t_{0}}<(\delta -\kappa
t_{0})/\varepsilon \}$. Setting $z_{\varepsilon }:=(\kappa t_{0}-\delta
)/(\varepsilon \sqrt{t_{0}})>0$ and writing $\Phi $ for the standard normal
distribution function,
\begin{align}
\mathbb{P}(\tau _{\delta }\leq t_{0})& \geq 1-\Phi \!\left( \frac{\delta
-\kappa t_{0}}{\varepsilon \sqrt{t_{0}}}\right) =\Phi (z_{\varepsilon })\geq
1-\frac{1}{\sqrt{2\pi }\,z_{\varepsilon }}\exp \!\left( -\frac{%
z_{\varepsilon }^{2}}{2}\right)  \notag \\
& =1-C_{\varepsilon }\exp \!\left( -c\,\frac{(\kappa t_{0}-\delta )^{2}}{%
\varepsilon ^{2}}\right) ,  \label{exit-tail} \\
C_{\varepsilon }& =\frac{\varepsilon \sqrt{t_{0}}}{\sqrt{2\pi }\,(\kappa
t_{0}-\delta )},\qquad c=\frac{1}{2t_{0}}.  \notag
\end{align}%
Here the Gaussian estimate is the standard Mills bound $1-\Phi (z)\leq (2\pi
)^{-1/2}z^{-1}e^{-z^{2}/2}$; for a bound uniform in $0<\varepsilon \leq 1$
one may absorb the factor $\varepsilon $ and use the $\varepsilon $%
-independent constant $C=\sqrt{t_{0}}/[\sqrt{2\pi }(\kappa t_{0}-\delta
)]\geq C_{\varepsilon }$. Because the estimate is phrased through the first
up-crossing time $\tau _{\delta }$, the process never leaves $[0,\delta
]\subset \lbrack 0,r_{0})$ before $\tau _{\delta }$, so no local-vs-global
comparison issue arises. Equation \eqref{exit-tail} gives
$\mathbb P(\tau_\delta\le t_0)\to1$ only for the displayed range
$0<\delta<r_0$ and $\kappa t_0>\delta$. It does not imply
$\tau_\delta\to0$ for a fixed positive $\delta$.
In fact, if $|b|\le M$, then for $0<t<\delta/M$,
\[
 \mathbb P(\tau_\delta\le t)
 \le\mathbb P\{\varepsilon\sup_{s\le t}|W_s|\ge\delta-Mt\}
 \longrightarrow0.
\]
The appropriate shrinking-neighbourhood consequence is
\[
 \lim_{\delta\downarrow0}\limsup_{\varepsilon\downarrow0}
                \mathbb P(\tau_\delta>t_0)=0\qquad(t_0>0).
\]
To see this, choose a fixed $t_1\in(0,t_0]$ with $\kappa t_1<r_0$ and then
$\delta<\min\{r_0,\kappa t_1\}$ in \eqref{exit-tail};
$\{\tau_\delta>t_0\}\subset\{\tau_\delta>t_1\}$.
No unstopped marginal tail estimate is obtained from the exit bound
without a separate control of returns.

For the one-dimensional sign example $b_1$ in \eqref{sgf}, the
radial coefficient is one, but Tanaka's additional nondecreasing
local-time term must be kept. The same stopped lower bound follows
from $R_t=t+\varepsilon B_t+L_t$, as proved in that example.

\paragraph{Case 2. Negative Radial Drift.}
\label{sect10.3}
Assume $b$ is bounded and Borel, $X_0^\varepsilon=0$, and
$\langle b(x),x\rangle\le-\kappa|x|$ for $x\ne0$, where $\kappa>0$.
In fact the following estimate only needs
$\langle b(x),x\rangle\le0$, and works in every integer dimension
$d\ge1$. For $T,\delta>0$,
\begin{equation}
 \mathbb P\left\{\sup_{0\le t\le T}|X_t^\varepsilon|\ge\delta\right\}
 \le \min\left\{1,2^{d/2}
             \exp\left[-\frac{\delta^2}{4\varepsilon^2T}\right]\right\}.
 \label{eq:inward-sup-tail}
\end{equation}
To prove this, for $\lambda>0$ put
\[
 F(t,x)=(1+2\lambda\varepsilon^2t)^{-d/2}
       \exp\left(\frac{\lambda|x|^2}{1+2\lambda\varepsilon^2t}\right).
\]
Direct differentiation gives
$(\partial_t+\tfrac12\varepsilon^2\Delta)F=0$ and
\[
 b(x)\cdot\nabla F(t,x)=
     \frac{2\lambda F(t,x)}{1+2\lambda\varepsilon^2t}
                          \langle b(x),x\rangle\le0.
\]
Let $\tau_\delta=\inf\{t:|X_t^\varepsilon|\ge\delta\}$.
Ito's formula stopped at $\tau_\delta\wedge T$ gives
$\mathbb E F(\tau_\delta\wedge T,X_{\tau_\delta\wedge T}^\varepsilon)
\le F(0,0)=1$; the stopped spatial derivatives are bounded, so the
martingale has mean zero. On $\{\tau_\delta\le T\}$, continuity and
monotonicity of the explicit lower bound give
\[
 F(\tau_\delta,X_{\tau_\delta}^\varepsilon)
 \ge(1+2\lambda\varepsilon^2T)^{-d/2}
            \exp\left(\frac{\lambda\delta^2}
                            {1+2\lambda\varepsilon^2T}\right).
\]
It follows that the exit probability is at most
$(1+2\lambda\varepsilon^2T)^{d/2}
 \exp[-\lambda\delta^2/(1+2\lambda\varepsilon^2T)]$.
Choose $\lambda=1/(2\varepsilon^2T)$ to obtain
\eqref{eq:inward-sup-tail}.
Therefore $X^\varepsilon\to0$ uniformly in probability on every
finite time interval, and, in particular,
$\mathbb P(|X_{t_0}^\varepsilon|\ge\delta)\to0$ for every fixed
$t_0,\delta>0$. The inward unit field
$b(x)=-x/|x|$ for $x\ne0$, $b(0)=0$, is an example.
This Cartesian proof keeps all It\^{o} corrections automatically. It
uses an upper bound on the exit probability; neither a Gaussian tail
lower bound nor an infimum over all noise intensities establishes
this convergence.

\subsection{Plots depicting system behavior under varying drifts}

\label{F} The following supplied plots illustrate the planar drift
examples. The finite-time selection statements for the cubic field
and $b_2^\delta$ are proved in \eqref{eq:cubic-selection} and
\eqref{eq:b2-selection}; they do not follow from the plots.
Multiple possible branches in a limiting law should not be confused
with nonuniqueness of that law: both displayed examples have proved
equal-weight limits. The original simulation programs, step sizes,
seeds, and raw output are not specified here, so the figures are not
used to establish probabilities, limiting rates, or robustness under
other drift perturbations.

%Consider
%\begin{equation*}
%b_{3}\left(
%\begin{bmatrix}
%x \\
%y%
%\end{bmatrix}%
%\right) =2%
%\begin{bmatrix}
%\sgn(y)\sqrt{|y|} \\
%\sgn(x)\sqrt{|x|}%
%\end{bmatrix}%
%.
%\end{equation*}%
%\begin{figure}[H]
%\centering          % 图片居中
%\includegraphics[width=0.8\textwidth]{rpb_3} % 宽度 = 页面宽度的80%
%\caption{Random perturbation of $2$-dim SDE with drift $b_{3}$}
%\label{fig:b_3}
%\end{figure}
%This $3$-D plot shows the trajectories of the SDE with $b_{3}$ near the
%origin, focusing on the two dominant immediate exit solutions pointing to
%Quadrant 1 ($Q_{1}$) and Quadrant 3 ($Q_{3}$). The trajectories are
%displayed for different noise intensities: $\varepsilon =0$ (deterministic
%limit, black lines), $\varepsilon =0.02$ (weak noise, light blue), $%
%\varepsilon =0.08$ (mild noise, green), $\varepsilon =0.2$ (strong noise,
%orange), and $\varepsilon =0.4$ (very strong noise, red). As $\varepsilon $
%decreases, the random paths converge tightly to the deterministic
%straight-line trajectories, visually confirming the origin's instability and
%the rapid escape from the origin along the two characteristic directions.

Consider%
\begin{equation}
b_{2}%
\begin{pmatrix}
x \\
y%
\end{pmatrix}%
=2%
\begin{pmatrix}
\sgn(x)\sqrt{|y|} \\
\sgn(y)\sqrt{|x|}%
\end{pmatrix}%
.  \label{b1b2}
\end{equation}

\begin{figure}[]
\centering          % 图片居中
\includegraphics[width=0.8\textwidth]{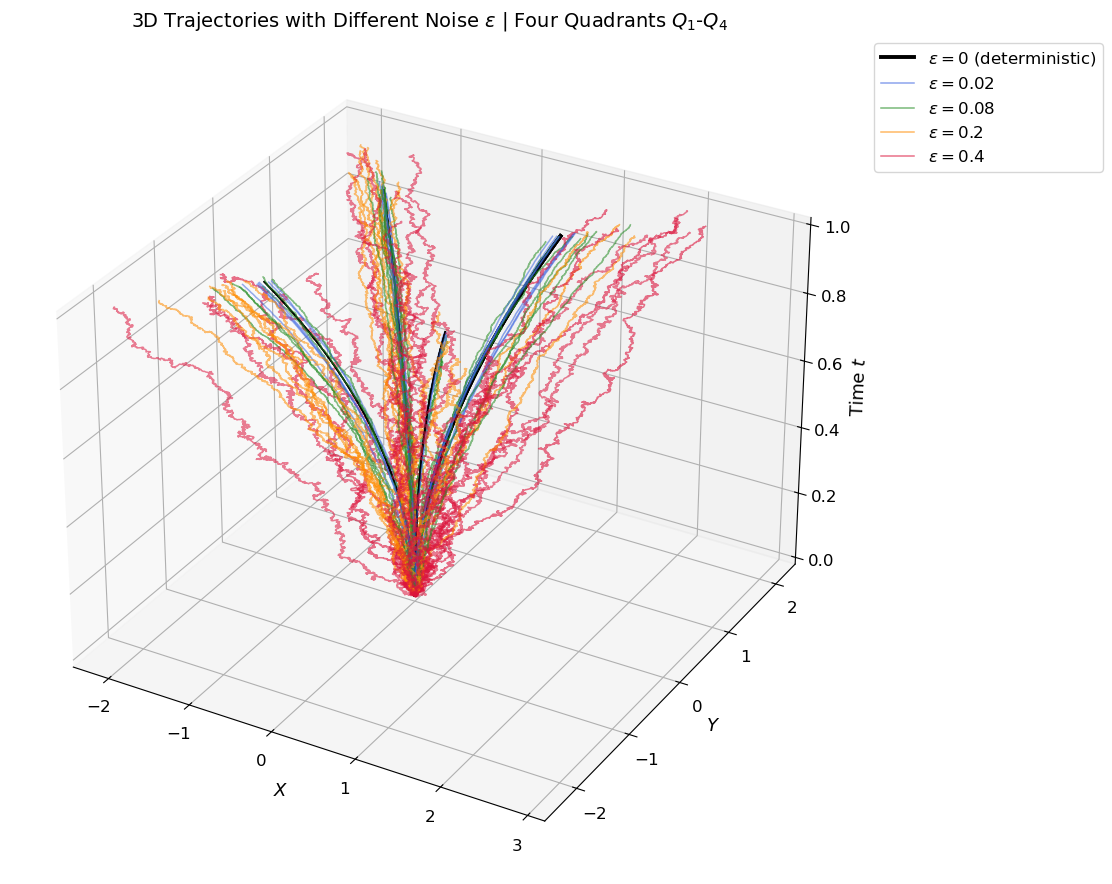} % 宽度 = 页面宽度的80%
\caption{ Random perturbation of $2$-dim SDE with drift $b_{2}$}
\label{fig:b_2}
\end{figure}
%The four trajectories strictly depart from the origin along the
For this drift $\langle b_2(x),x\rangle\ge0$, but it equals zero on
the coordinate axes, so the strict radial hypothesis of (A1) fails.
Example \ref{exa2d}, with $\delta=0$, proves the fixed-time path-law
limit as the equal mixture of $(\sigma_1t^2,\sigma_2t^2)$,
$\sigma_i\in\{-1,1\}$. These zero-noise paths lie in the four open
quadrants for positive times. Positive-noise paths can cross the
axes; a coordinate hyperplane is not a polar set merely because it
has Lebesgue measure zero. A single point in dimension two is polar
for this locally bounded, non-explosive diffusion by localization and
Girsanov, which is a different assertion. The curves in Figure
\ref{fig:b_2} illustrate the example but are not a proof of selection
or a quantitative return estimate.

\begin{figure}[]
\centering          % 图片居中
\includegraphics[width=0.8\textwidth]{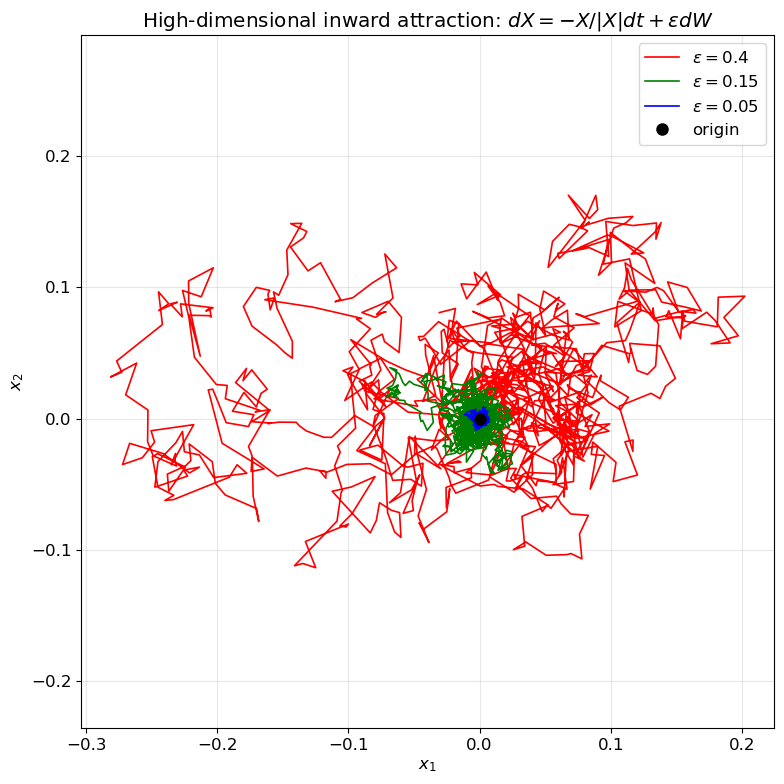} % 宽度 = 页面宽度的80%
\caption{Two-dimensional inward attracting SDE with drift $b_{4} $}
\label{fig:high}
\end{figure}

This plot illustrates sample paths of the $2$-D SDE drift $b_{4}$, for
decreasing noise intensities $\varepsilon =0.4$ (red), $0.15$ (green), and $%
0.05$ (blue). The origin (black dot) is a globally asymptotically stable
equilibrium: as $\varepsilon \rightarrow 0$, trajectories converge tightly
to the origin, visually confirming that the inward radial drift traps the
particle near the origin in the zero-noise limit.

\begin{lemma}
\label{mco}
Let $M=(M_t)_{t\ge0}$ be a continuous local martingale with $M_0=0$. If
\[
\langle M\rangle_\infty:=\lim_{t\to\infty}\langle M\rangle_t<\infty
\qquad\text{a.s.},
\]
then $M_t$ converges almost surely as $t\to\infty$.
\end{lemma}

\begin{proof}
For each integer $n\ge1$, set $\tau_n:=\inf\{t\ge0:\langle M\rangle_t>n\}$ and define the stopped
process $M^{(n)}_t:=M_{t\wedge\tau_n}$. Then
\[
\langle M^{(n)}\rangle_\infty=\langle M\rangle_{\tau_n}\le n,
\]
so $M^{(n)}$ is a true square-integrable martingale with bounded second moment. By Doob's
martingale convergence theorem, $M^{(n)}_t$ converges almost surely as $t\to\infty$.

On the event $\{\langle M\rangle_\infty<\infty\}$, there exists a random integer $N$ such that
$\langle M\rangle_\infty\le N$. For every $n\ge N$ we have $\tau_n=\infty$, hence
$M^{(n)}_t=M_t$ for all $t\ge0$. Since $M^{(n)}$ converges, so does $M$. Thus $M_t$ converges
almost surely.
\end{proof}

\subsection{Lipschitz maps control Hausdorff dimension}

Let $(X,d_X)$ and $(Y,d_Y)$ be metric spaces. For any subset $F$ of a metric
space, write
\begin{equation*}
|F| = \sup_{x,y\in F} d(x,y)
\end{equation*}
for the diameter of $F$. For $s\ge 0$ and $\delta>0$, define the \textit{$%
\delta$-truncated $s$-dimensional Hausdorff measure}:
\begin{equation*}
\mathcal{H}^s_\delta(F) = \inf\left\{ \sum_{i=1}^\infty |U_i|^s \;\bigg|\;
F\subset \bigcup_{i=1}^\infty U_i,\ |U_i|<\delta \right\}.
\end{equation*}
The full $s$-dimensional Hausdorff measure is
\begin{equation*}
\mathcal{H}^s(F) = \lim_{\delta\to 0^+} \mathcal{H}^s_\delta(F).
\end{equation*}
The Hausdorff dimension of $F$ is
\begin{equation*}
\dim_H F = \inf\big\{ s\ge 0 \,\big|\, \mathcal{H}^s(F)=0 \big\}.
\end{equation*}

\begin{theorem}
\label{liph} Suppose $f\colon X\to Y$ is $L$-Lipschitz for some constant $%
L>0 $, i.e.,
\begin{equation*}
d_Y(f(x_1),f(x_2)) \le L\, d_X(x_1,x_2), \quad \forall x_1,x_2\in X.
\end{equation*}
Then for every subset $E\subset X$,
\begin{equation*}
\dim_H f(E) \le \dim_H E.
\end{equation*}
\end{theorem}

\begin{proof}
Fix arbitrary $s\geq 0$ and $\delta >0$. Let $\{U_{i}\}_{i\in \mathbb{N}}$
be any $\delta $-cover of $E$, so $E\subset \bigcup_{i}U_{i}$ and $%
|U_{i}|<\delta $ for all $i$. First estimate the diameter of each image set $%
f(U_{i})$:
\begin{align*}
|f(U_{i})|& =\sup_{a,b\in U_{i}}d_{Y}(f(a),f(b)) \\
& \leq \sup_{a,b\in U_{i}}L\,d_{X}(a,b)=L\sup_{a,b\in
U_{i}}d_{X}(a,b)=L|U_{i}|.
\end{align*}%
In particular, $|f(U_{i})|<L\delta $, so $\{f(U_{i})\}_{i\in \mathbb{N}}$ is
an $L\delta $-cover of $f(E)$. Raise the diameter bound to power $s$ and sum
over all covering sets:
\begin{equation*}
\sum_{i=1}^{\infty }|f(U_{i})|^{s}\leq \sum_{i=1}^{\infty }\big(L|U_{i}|\big)%
^{s}=L^{s}\sum_{i=1}^{\infty }|U_{i}|^{s}.
\end{equation*}%
%
%
%
%
%
%
%
%
%
%
%
%
%
%
%
%
%
%
%
%
%
%
%
%
%
%
%
%
%
%
%
%
%
%
%
%
%
%
%
%
%
%
%
%
%
%
%
%
%
%
%
%
%
%
%
%
%
%
%
%
%
%
%
%
%
%
%
%
%
%
%
%
%
%
%
%
%
%
%
%
%
%
%
%
%
%
%
%
%Take the infimum over all $\delta$-covers $\{U_i\}$ of $E$ on both sides. By
%definition of truncated Hausdorff measure, the left-hand side is bounded
%below by $\mathcal{H}^s_{L\delta}(f(E))$, hence
%\begin{equation*}
%\mathcal{H}^s_{L\delta}(f(E)) \le L^s \mathcal{H}^s_\delta(E).
%\end{equation*}
Take the infimum over all $\delta $-covers $\{U_{i}\}$ of E on both sides.
For a fixed $\delta $-cover $\{U_{i}\}$, the sum $\sum_{i}|f(U_{i})|^{s}$
corresponds to one candidate sum for the $L\delta $-covering of $f(E)$.
Since $\mathcal{H}_{L\delta }^{s}(f(E))$ is the infimum of all such sums
over all $L\delta $-covers, we have
\begin{equation*}
\mathcal{H}_{L\delta }^{s}(f(E))\leq \sum_{i=1}^{\infty }|f(U_{i})|^{s}\leq
L^{s}\sum_{i=1}^{\infty }|U_{i}|^{s}.
\end{equation*}
Taking infimum over all $\delta $-covers on the right-hand side yields $%
\mathcal{H}_{L\delta }^{s}(f(E))\leq L^{s}\mathcal{H}_{\delta }^{s}(E).$ Now
send $\delta \rightarrow 0^{+}$. As $\delta \rightarrow 0$, we have $L\delta
\rightarrow 0$, and the limit defining full Hausdorff measure gives
\begin{equation*}
\mathcal{H}^{s}(f(E))\leq L^{s}\mathcal{H}^{s}(E).
\end{equation*}%
Let $t_{0}=\dim _{H}E$. Take any real number $s>t_{0}$. By definition of
Hausdorff dimension, $\mathcal{H}^{s}(E)=0$. Substitute into the measure
inequality above:
\begin{equation*}
0\leq \mathcal{H}^{s}(f(E))\leq L^{s}\cdot 0=0\implies \mathcal{H}%
^{s}(f(E))=0.
\end{equation*}%
This shows every $s>t_{0}$ belongs to the set $\{s\geq 0:\mathcal{H}%
^{s}(f(E))=0\}$. Taking the infimum over all such $s$,
\begin{equation*}
\dim _{H}f(E)=\inf \big\{s\geq 0\,\big|\,\mathcal{H}^{s}(f(E))=0\big\}\leq
\inf \big\{s\in \mathbb{R}\,\big|\,s>t_{0}\big\}=t_{0}=\dim _{H}E.
\end{equation*}%
The proof is complete.
\end{proof}

\noindent \textbf{Competing Interests Statement.}
The authors declare that there are no known competing financial interests or
personal relationships that could have appeared to influence the work reported in this paper.

\end{document}